\documentclass[reqno]{amsart}
\usepackage{amssymb, amsmath, amsthm,bbm, color, enumerate, tikz, tikz-3dplot}
\usepackage{ulem}
\usepackage{soul}
\usepackage{xcolor,cancel}

\def\beq{\begin{equation}}
\def\endeq{\end{equation}}
\def\lesim{\lesssim}

\def\lesim{\lesssim}

\numberwithin{equation}{section}

\usepackage{amsfonts}
\usepackage{amsthm}
\usepackage[mathscr]{eucal}
\usepackage{graphicx}
\usepackage{verbatim}

\numberwithin{equation}{section}

\newcommand{\bbH}{\mathbb{H}}
\newcommand{\C}{\mathbb{C}}
\newcommand{\R}{\mathbb{R}}
\newcommand{\Z}{\mathbb{Z}}
\newcommand{\N}{\mathbb{N}}

\newcommand{\ud}{\mathrm{d}}

\newcommand{\supp}{\mathrm{supp}\,}

\newcommand{\fa}{\mathfrak{a}}
\newcommand{\fb}{\mathfrak{b}}
\newcommand{\fc}{\mathfrak{c}}

\newcommand{\ba}{\mathbbm{a}}
\newcommand{\bc}{\mathbbm{c}}

\theoremstyle{plain}
\newtheorem{theorem}{Theorem}[section]
\newtheorem{lemma}[theorem]{Lemma}
\newtheorem{proposition}[theorem]{Proposition}
\newtheorem{corollary}[theorem]{Corollary}
\newtheorem*{key estimate}{Key estimate}

\theoremstyle{definition}
\newtheorem{definition}[theorem]{Definition}

\newtheorem*{key example}{Key example}

\newcommand{\Be}{\begin{equation}}
\newcommand{\Ee}{\end{equation}}
\newcommand{\Bea}{\begin{align}}
\newcommand{\Eea}{\end{align}}
\newcommand{\Beas}{\begin{align*}}
\newcommand{\Eeas}{\end{endalign*}}
\newcommand{\Benu}{\begin{enumerate}}
\newcommand{\Eenu}{\end{enumerate}}
\newcommand{\Bi}{\begin{itemize}}
\newcommand{\Ei}{\end{itemize}}

\newcommand\bbone{{\mathbbm 1}}

\newcommand{\normm}[1]{{\left\vert\kern-0.25ex\left\vert\kern-0.25ex\left\vert #1 \right\vert\kern-0.25ex\right\vert\kern-0.25ex\right\vert}}

\newcommand{\floor}[1]{\lfloor #1 \rfloor }

\def\dil{{\text{\rm Dil}}}

\def\intslash{\rlap{\kern  .32em $\mspace {.5mu}\backslash$ }\int}
\def\qsl{{\rlap{\kern  .32em $\mspace {.5mu}\backslash$ }\int_{Q_x}}}

\def\vth{\vartheta}

\def\R{\mathbb R}

\def\N{\mathbb N}

\def\floor#1{{\lfloor #1 \rfloor }}
\def\emph#1{{\it #1 }}

\def\supp{{\text{\rm supp }}}

\def\inn#1#2{\langle#1,#2\rangle}
\def\biginn#1#2{\big\langle#1,#2\big\rangle}

\def\noi{\noindent}

\def\meas{{\text{\rm meas}}}

\def\lc{\lesssim}
\def\gc{\gtrsim}

\def\gam{\gamma}             
             
\def\eps{\varepsilon}

\def\ka{\kappa}
            
\def\la{\lambda}

\def\om{\omega}

\def\fA{{\mathfrak {A}}}
\def\fB{{\mathfrak {B}}}

\def\fM{{\mathfrak {M}}}

\def\fP{{\mathfrak {P}}}

\def\fa{{\mathfrak {a}}}
\def\fb{{\mathfrak {b}}}
\def\fc{{\mathfrak {c}}}

\def\fu{{\mathfrak {u}}}
\def\fv{{\mathfrak {v}}}

\def\bbH{{\mathbb {H}}}

\def\bbN{{\mathbb {N}}}

\def\bbR{{\mathbb {R}}}

\def\bbZ{{\mathbb {Z}}}

\def\cG{{\mathcal {G}}}

\def\cJ{{\mathcal {J}}}

\def\cL{{\mathcal {L}}}
\def\cM{{\mathcal {M}}}

\def\be#1{\begin{equation}\label{#1}}
\def\endeq{\end{equation}}
\def\endal{\end{align}}
\def\bas{\begin{align*}}
\def\eas{\end{align*}}
\def\bi{\begin{itemize}}
\def\ei{\end{itemize}}

\begin{document}
\title
[The circular maximal operator on Heisenberg radial functions]{The circular maximal operator \\on Heisenberg radial functions}

\author[D. Beltran, S. Guo, J. Hickman,  A. Seeger]{David Beltran \quad   Shaoming Guo \\  Jonathan Hickman   \quad  Andreas Seeger}
\date{\today}

\makeatletter
\def\author@andify{%
  \nxandlist {\unskip ,\penalty-1 \space\ignorespaces}%
    {\unskip {} \@@and~}%
    {\unskip \penalty-2 \space \@@and~}%
}
\makeatother

\address{David Beltran: Department of Mathematics, University of Wisconsin, 480 Lincoln Drive, Madison, WI, 53706, USA.}
\email{dbeltran@math.wisc.edu}
\address{Shaoming Guo: Department of Mathematics, University of Wisconsin, 480 Lincoln Drive, Madison, WI, 53706, USA.}
\email{shaomingguo@math.wisc.edu}
\address{Jonathan Hickman: School of Mathematics, James Clerk Maxwell Building, The King's Buildings, Peter Guthrie Tait Road, Edinburgh, EH9 3FD, UK.}
\email{jonathan.hickman@ed.ac.uk}
\address{Andreas Seeger: Department of Mathematics, University of Wisconsin, 480 Lincoln Drive, Madison, WI, 53706, USA.}
\email{seeger@math.wisc.edu}

\tdplotsetmaincoords{70}{15}
\tikzset{every circle/.append style={x=1cm, y=1cm}}

\begin{abstract} 
 Lebesgue space estimates are obtained for the  circular maximal function on the Heisenberg group $\bbH^1$ restricted to a class of  Heisenberg radial functions. Under this  assumption, the problem reduces to studying a maximal operator on the Euclidean plane. This  operator has a number of interesting features: it is associated to a non-smooth curve distribution and, furthermore, fails both the usual rotational curvature and cinematic curvature conditions. 
\end{abstract}

\maketitle

\section{Introduction} Let $\bbH^n$ denote the Heisenberg group given by endowing $\R\times \bbR^{2n}$ with the non-commutative group operation
\begin{equation*}
 (u, x) \cdot (v, y) := \big(u + v + x^{\top}By, x + y\big) \qquad \textrm{for all $(u,x), (v,y) \in \R \times \R^{2n}$}
\end{equation*}
where $B = bJ$ with
$J :=\begin{pmatrix} 
0 & -I_n \\
I_n & 0\end{pmatrix} $
the matrix associated to the standard symplectic form on $\bbR^{2n}$ and
 $b \neq 0$ (usually one takes $b = 1/2$).

Let $\mu\equiv\mu_1$ denote the normalised surface measure on the sphere 
\begin{equation*}
    \{0\} \times S^{2n-1} := \{ (0,y) \in \R \times \R^{2n} : |y| =1\}.
\end{equation*} 
If $\dil_t(u,x):=(t^2u, tx)$ are the automorphic dilations on $\bbH^n$, then the normalised surface measure $\mu_t$ supported on $tS^{2n-1}$ can be viewed as a dilate of $\mu_1$ in the sense that $\inn{f}{\mu_t}= \inn {f(\dil_t\,\cdot\,)}{\mu} $.

Given a function $f$ on $\bbH^n$ belonging to a suitable \emph{a priori} class consider the spherical means 
\begin{equation*}
f* \mu_t (u,x) := \int_{S^{2n-1}} f(u - t x^{\top}By,x -ty )\,\ud \mu(y) \quad 
 \textrm{for $(u,x) \in \bbH^n$ and $t > 0$.}
\end{equation*}
For smooth functions $f$ one has $f*\mu_t(u,x)\to f(u,x)$ pointwise as $t\to 0$. It is of interest to extend this convergence result to an almost everywhere convergence result for functions on $L^p(\bbH^n)$, in a suitable range of $p$. Such a result  follows from $L^p$ bounds for the associated spherical maximal function
\begin{equation}\label{circular maximal operator}
Mf(u,x):=  \sup_{t > 0} |f* \mu_t (u,x)|.
\end{equation} 
The operator $M$ can be understood as a Heisenberg analogue of the classical (Euclidean) spherical maximal function of Stein~\cite{Stein1976} and Bourgain~\cite{Bourgain1986} (see also \cite{Mockenhaupt1992, SS1997, Schlag1998}).
The maximal function \eqref{circular maximal operator} was introduced by Nevo and Thangavelu in~\cite{Nevo1997} where $L^p$ estimates were proven in dimensions $n \geq 2$ for $p$ belonging to a non-sharp range.
By choosing $f$ to be the standard example 
\begin{equation*}
  f(u,x) := \big(|x|\log(1/|x|)\big)^{1-2n} \chi(u,x)  
\end{equation*} 
for an appropriate choice of cutoff function $\chi$, it follows that $L^p\to L^p$ estimates can only hold for $p> \frac{2n}{2n-1}$. For $n\ge 2$ the sufficiency of this condition was established independently by M\"uller and the fourth author~\cite{MS2004}  and by Narayanan and Thangavelu~\cite{Narayanan2004}; the work in~\cite{MS2004}  also treats a wider class of operators defined on M\'etivier groups. Results in a more general variable coefficient setting can be found in a recent paper by Kim \cite{joonil2019}.
Related to these investigations the $L^p$ results of~\cite{MS2004, Narayanan2004} were extended in~\cite{ACPS} to deal with variants of the operator \eqref{circular maximal operator} where the original sphere, centred at the origin, does not lie in the subspace 
$\{0\}\times \bbR^{2n}$ (that is, the corresponding dilates of $\mu$  are no longer supported in a fixed hyperplane). The  latter paper 
is closely related to 
\cite{PrSe2007}, \cite{PS2019}  which establish sharp $L^p$-Sobolev bounds for certain Radon-type operators associated to curves in three-dimensional manifolds;  in particular \cite{PS2019} covers the averages  $f \mapsto f \ast \mu_t$ in $\mathbb{H}^1$, and perturbations of these operators, when acting on compactly supported functions.
Mapping properties and sparse domination for a lacunary version of $M$ have been recently studied in~\cite{BHRT}, also under the assumption $n \geq 2$.
We note that for   the proofs of  the positive  results on the Heisenberg spherical maximal functions   mentioned above it was  essential that a boundedness result holds for $p=2$, which leads to the restriction $n\ge 2$. Such an $L^2$ result fails to hold on $\bbH^1$, and 
 it is currently not known whether the   circular maximal operator \eqref{circular maximal operator}  on the Heisenberg group $\bbH^1$ is bounded on $L^p(\bbH^1)$ for {\it any} $p<\infty$.

In this paper we consider the problem of estimating the maximal function \eqref{circular maximal operator} on the  sub-algebra of {\it Heisenberg-radial} (or {\it $\mathbb H$-radial}) functions on $\bbH^1$. 
Here  a function $f \colon \bbH^1 \to \C$ is said to be  $\mathbb{H}$-radial if 
$f(u,Rx)=f(u,x)$ for all $R\in \mathrm{SO}(2)$.
Given the underlying symmetries of the maximal operator, this is a natural condition to impose on the input function: indeed, if $f$ is $\bbH$-radial then, $Mf$ is also $\bbH$-radial. Our main theorem characterises the $L^p$ mapping properties of $M$ acting on  $\mathbb{H}$-radial functions.


\begin{theorem}\label{circular maximal theorem}
For $2 < p \leq \infty$ the \emph{a priori} estimate
\begin{equation*}
\|Mf \|_{L^p(\bbH^1)} \leq C_p \|f \|_{L^p(\bbH^1)}
\end{equation*}
holds for $\mathbb{H}$-radial functions on $\bbH^1$. Here $C_p$ is a constant depending only on $p$. 
\end{theorem}


We  shall reduce 
Theorem~\ref{circular maximal theorem} 
to bounding a maximal function $\sup_{t>0} |A_tf|$ where the $A_t$ are non-convolution averaging operators 
in two dimensions.
We aim to follow the strategy used in \cite{Mockenhaupt1992, Mockenhaupt1993} to study the  Euclidean circular maximal function and its relatives. However, in comparison with~\cite{Mockenhaupt1993}, substantial new difficulties arise. First, we need to consider a  distribution of  curves which is  not smooth. Moreover, 
the \textit{rotational curvature} and \textit{cinematic curvature} conditions (as formulated in~\cite{Sogge1991,Mockenhaupt1993}) fail to hold, and hence $\sup_{t>0} |A_tf|$ does not belong to the classes of variable coefficient maximal functions considered in~\cite{Mockenhaupt1993}. 
Significant technical challenges are encountered when dealing with the various singularities of the operator, and our arguments are based on the analysis of a class 
 of oscillatory integral operators with 2-sided fold singularities which extends the work in \cite{phong-stein91} and \cite{Cuccagna}. A more detailed discussion of the proof strategy can be found in \S\ref{sec:proof strategy} below.

\begin{figure}
\begin{tikzpicture}[tdplot_main_coords, scale=2]
 \begin{scope}[rotate around z=0, rotate=0]

\foreach \y in {100, 95,...,0}
    {
\def\x{0};
\draw[color=white!\y!blue, very thin] 
		({\x - 1},{\y/50 },{- \y/50))})
    \foreach \a in {0,0.2,...,6.5}
    {
       --({\x - cos(deg(\a))},{\y/50 - sin(deg(\a))},{ - (\y/50)*cos(deg(\a)) + \x*sin(deg(\a))})
    }
    ;
}

\draw[->][black] 
		(-1.5,0,0)->(1.5,0,0) node[anchor=north]{$x_1$};
\draw[->][black] 
		(0,-1.5,0)->(0,3,0) node[anchor=east]{$x_2$};
\draw[->][black] 
		(0,0,-1)->(0,0,1.5) node[anchor=east]{$u$};	
	\end{scope}
	\end{tikzpicture}
\caption{The unit circle tilts and stretches as it is translated along the $x_2$-axis under the Heisenberg operation.}
\end{figure}
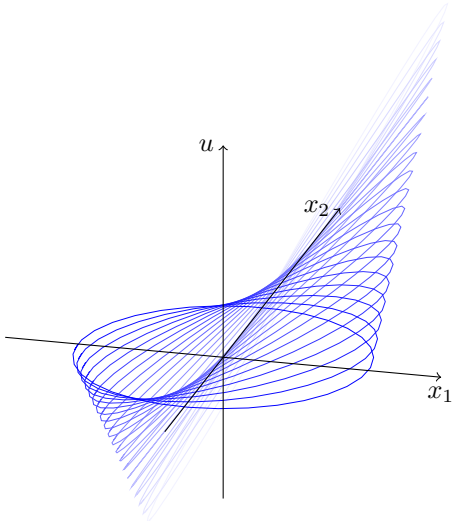

\subsection*{\it Structure of the paper.} Section~\ref{sec:proof strategy} reviews the strategy for bounding the Euclidean circular maximal function based on local smoothing estimates. The difficulties encountered in our particular situation are also described. In Sections \ref{sec:reformulation} -- \ref{Lp section} we prove bounds for a local variant of $M$, where the supremum is restricted to $1 \leq t \leq 2$. In particular, Section~\ref{sec:reformulation} reduces Theorem~\ref{circular maximal theorem} to a bound for a maximal function in two dimensions. Section~\ref{curvature considerations section} describes notions of curvature which feature in the analysis of $M$. In Section~\ref{decomposition section} the maximal function is decomposed into different pieces according to curvature considerations. 
In Section~\ref{two-parameter-sect} we consider classes of oscillatory integral operators depending on two parameters which are crucial for the relevant $L^2$-theory, mainly based on a `fixed-time' analysis. 
In Section~\ref{L2 section} we apply these $L^2$ estimates to the problem on the Heisenberg group. In Section~\ref{Lp section} we discuss the $L^p$ theory, based on $L^p$ space-time (`local smoothing') estimates. Finally, in Section~\ref{sec:global} the bounds for the local maximal function are extended to bounds for $M$. Two appendices are included for the reader's convenience, providing  useful integration-by-parts lemmata and many explicit computations helpful to the analysis.




\subsection*{\it Notational conventions} Given a (possibly empty) list of objects $L$, for real numbers $A_p, B_p \geq 0$ depending on some Lebesgue exponent $p$ the notation $A_p \lesssim_L B_p$, $A_p = O_L(B_p)$ or $B_p \gtrsim_L A_p$ signifies that $A_p \leq CB_p$ for some constant $C = C_{L,p} \geq 0$ depending on the objects in the list and $p$. In addition, $A_p \sim_L B_p$ is used to signify that both $A_p \lesssim_L B_p$ and $A_p \gtrsim_L B_p$ hold. Given $a$, $b \in \R$ we write $a \wedge b:= \min \{a, b\}$ and  $a \vee b:=\max \{a,b\}$. Given $x = (x_1, x_2, x_3) \in \R^3$ we will often write $x = (x_1, x'') \in \R \times \R^2$ or $x = (x', x_3) \in \R^2 \times \R$. Given $x \in \R^2$ and $t \in \R$ we will also often write $\vec{x} = (x,t) \in \R^2 \times \R$. 
Throughout the article $N$ denotes some fixed large integer, chosen so as to satisfy the forthcoming requirements of the proofs.
The choice of $N=10^{1000}$ is permissible (and in the $d$-dimensional  version of estimates in  Sections 
\ref{two-parameter-sect} and \ref{L2 section}, it  never needs to exceed   $10^{100d}$).
For a phase function $\varphi(x;z)$ the notation $\partial^2_{xz}\varphi$ 
refers to the matrix $A$ with entries $A_{ij}=\partial_{x_i z_j}^2 \varphi$ while
the notation $\partial^2_{zx}\varphi$ refers to its transpose.
The length of a multiindex $\alpha\in \bbN_0^d$ is given by $|\alpha|=\sum_{i=1}^d{\alpha_i}$.
The $C^N$ norm of $(x;z)\mapsto a(x;z)$ is given by $\max_{|\alpha|+|\beta|\le N}\|\partial_x^\alpha\partial_z^\beta a\|_\infty$. We also use the notation $\|a\|_{C^N_z}$ for 
$\sup_x \|a(x;\cdot) \|_{C^N}$.
For a linear operator $T$ bounded from $L^p$ to $L^q$ we use both $\|T\|_{L^p\to L^q}$, $\|T\|_{p\to q}$ as a notation for the operator norm. For a one-parameter family of linear operators $\{T_t\}_{t \in E}$, $\|\sup_{t\in E}  |T_t|\|_{p \to q}$ denotes the $L^p\to L^q$  operator norm of  the sublinear operator
$f\mapsto \sup_{t\in E} |T_t f|$.




\section*{Acknowledgements}
{The authors thank the American Institute of Mathematics for funding their collaboration through the SQuaRE program, also  supported in part  by the National Science Foundation. D.B. was partially supported by ERCEA Advanced Grant 2014 669689 - HADE, by the MINECO project MTM2014-53850-P, by Basque Government project IT-641-13 and also by the Basque Government through the BERC 2018-2021 program and by Spanish Ministry of Economy and Competitiveness MINECO: BCAM Severo Ochoa excellence accreditation SEV-2017-0718. S.G. was partially supported by  NSF grant DMS-1800274. A.S. was partially supported by  NSF grant DMS-1764295 and by a Simons fellowship. This material is partly based upon work supported by the National Science Foundation under Grant No. DMS-1440140 while the authors were in residence at the Mathematical Sciences Research Institute in Berkeley, California, during the Spring 2017 semester. The authors would also like to thank the anonymous referee for a careful reading and valuable suggestions.}

\section{Proof strategy}\label{sec:proof strategy}

Theorem~\ref{circular maximal theorem} easily reduces to bounding a maximal function $\sup_{t>0} |A_tf|$ where the $A_t$ are averaging operators on the Euclidean plane. We aim to follow the broad strategy introduced in~\cite{Mockenhaupt1992}
 to study the Euclidean circular maximal function, which we now recall.
Define $A_t^{\mathrm{eucl}}f$ by taking $A_t^{\mathrm{eucl}}f(x)$ to be the average of $f$ over the circle $\Sigma_{x,t}^{\mathrm{eucl}}$ in the plane centred at $x$ with radius $t$. 
Note that the associated curve distribution is described by the defining function
\begin{equation*}
 \Phi^{\mathrm{eucl}}(x,t;y) := |x-y|^2 - t^2 \qquad \textrm{for $(x,t;y) \in \R^2 \times \R \times \R^2$};   
\end{equation*} 
in particular, $\Sigma^{\mathrm{eucl}}_{x,t} = \{y \in \R^2 : \Phi(x,t;y) = 0\}$. The associated maximal function
\begin{equation*}
    M^{\mathrm{eucl}}f(x) := \sup_{t > 0} |A_t^{\mathrm{eucl}}f(x)| 
\end{equation*}
is the classical circular maximal function studied by Bourgain \cite{Bourgain1986} and also in \cite{Mockenhaupt1992}. A Littlewood--Paley argument reduces the problem of bounding $M^{\mathrm{eucl}}f$ to bounding the local maximal function 
\begin{equation*}
    \sup_{1 \leq t \leq 2} |A_t^{\mathrm{eucl}}f(x)|.
\end{equation*} 

Decompose the averaging operator $A_t^{\mathrm{eucl}}f$ as a sum of pieces $A_t^{\mathrm{eucl},j}f$ localised at frequency scale $2^j$. The sum of the low frequency pieces $(j \leq 0)$ can be bounded in one go via comparison with the Hardy--Littlewood maximal operator and it remains to bound the high frequency pieces. There are two steps in the argument:

\begin{enumerate}[i)]
\item The first step is to show that the inequality 
\begin{equation}\label{introduction L2}
    \big\|\sup_{1 \leq t \leq 2} |A_t^{\mathrm{eucl},j} f| \big\|_{L^2(\R^2)} \leq C \|f\|_{L^2(\R^2)}
\end{equation}
holds uniformly in $j$. An elementary Sobolev embedding reduces \eqref{introduction L2} to proving $L^2$ estimates for certain oscillatory integral operators. A $T^*T$ argument further reduces \eqref{introduction L2} to bounding the corresponding kernels, which are then amenable to stationary phase analysis. 

\item Interpolating \eqref{introduction L2} with the trivial $L^{\infty}$ estimate, 
\begin{equation}\label{introduction Lp}
    \big\|\sup_{1 \leq t \leq 2} |A_t^{\mathrm{eucl},j}f| \big\|_{L^p(\R^2)} \leq C \|f\|_{L^p(\R^2)} \qquad \textrm{for all $2 \leq p \leq \infty$.}
\end{equation}
The problem here is that \eqref{introduction Lp} does not sum in $j$. If, however, there exists some $2 < p_{\circ} < \infty$ and $\varepsilon(p_{\circ}) > 0$ such that 
\begin{equation}\label{introduction local smoothing}
    \big\|\sup_{1 \leq t \leq 2} | A_t^{\mathrm{eucl},j} f| \big\|_{L^{p_{\circ}}(\R^2)} \leq C2^{-j\varepsilon(p_{\circ})} \|f\|_{L^{p_{\circ}}(\R^2)},
\end{equation}
then one may interpolate \eqref{introduction Lp} and \eqref{introduction local smoothing} to obtain favourable $j$ dependence for all $2 < p < \infty$, concluding the proof. The strategy in~\cite{Mockenhaupt1992} is to prove a bound of the form \eqref{introduction local smoothing} via local in time $L^p$ space-time bounds (so-called \textit{local smoothing estimates}) for the wave equation. 
\end{enumerate}
 
There are two key properties of the circular maximal function which allow the above analysis to be carried out, both of which can be expressed in terms of the defining function $\Phi^{\mathrm{eucl}}$. The first is the standard decay properties of the Fourier transform of surface carried measure which correspond to nonvanishing of the  \textit{Phong--Stein rotational curvature} (see, for instance,~\cite[Chapter IX, \S3.1]{Stein1993},\footnote{The definitions of the rotational curvature and other concepts featured in this discussion are also reviewed in \S\ref{curvature considerations section} below.}). This is used to prove the oscillatory integral estimates i). The second is that the \textit{cinematic curvature} (see,~\cite{Sogge1991})
is non-vanishing, which features in the proof of the local smoothing estimates used in ii). The analysis can be generalised to variable coefficient maximal functions formed by averaging operators on the plane associated to defining functions $\Phi$ which satisfy these two conditions \cite{Sogge1991}. 
 
 Now suppose $A_tf$ denote the averaging operators on $\R^2$ which arises in the study of our  maximal operator acting on $\mathbb{H}$-radial functions. This family of operators has an associated defining function $\Phi$, which is described in \eqref{defining function} below. As before, one may decompose $A_tf$ as a sum of pieces $A_t^jf$ localised at a frequency scale $2^j$. Significant issues arise, however, when it comes to implementing either of the above steps to analyse the $A_t^jf$ in this case: 

\begin{enumerate}[i$'$)]
    \item The defining function $\Phi$ has vanishing  rotational curvature. Indeed, the oscillatory integral estimates in the above proof sketch of \eqref{introduction L2} do not hold in this case. 
    \item The defining function $\Phi$ also has vanishing cinematic curvature. This precludes direct application of local smoothing estimates in the proof of \eqref{introduction local smoothing}. 
\end{enumerate}

In order to deal with these issues it is necessary decompose the operator $A_t$ with respect to the various curvatures and to prove bounds of the form \eqref{introduction L2}, \eqref{introduction Lp} and \eqref{introduction local smoothing} for each of the localised pieces.

In bounding the localised pieces of $A_t$, the main difficulty is caused by the vanishing of the rotational curvature. In particular, here the $L^2$ theory relies on certain two parameter variants of estimates for oscillatory integral operators with two-sided fold singularities. Our arguments build on the techniques in~\cite{Cuccagna, GS1999}. This is in contrast with the analysis of the Euclidean maximal function, where the classical estimates for non-degenerate oscillatory integral operators of H\"ormander~\cite{Hormander1973} suffice. The presence of a two-sided fold incurs a (necessary) loss in the oscillatory integral estimates (compared with the non-vanishing rotational curvature case), but special properties of the Heisenberg maximal function allow one to compensate for this. A similar phenomenon was previously observed in the analysis of the spherical maximal function in $\mathbb{H}^n$ for $n > 1$ in~\cite{MS2004}.  

The vanishing of the cinematic curvature presents less of a problem, essentially because the desired bound \eqref{introduction local smoothing} is non-quantitative: all that is required is for \eqref{introduction local smoothing} to hold for \emph{some} $p_{\circ}$ and \emph{some} $\varepsilon(p_{\circ}) > 0$. Roughly speaking, the strategy is to decompose the operator into two parts: one piece supported on the $\delta$-neighbourhood of the variety where the cinematic curvature vanishes and a complementary piece. The former is dealt using a variant of \eqref{introduction Lp} which includes a gain in $\delta$ arising from the additional localisation. The latter piece has non-vanishing cinematic curvature and can be dealt with using local smoothing estimates. Choosing $\delta$ appropriately, one obtains the desired bound. Similar ideas were used by Kung \cite{Kung} to treat a family of Fourier integral operators  with vanishing cinematic curvature.




\section{Reduction to a maximal operator in the plane}\label{sec:reformulation}

\subsection{ Singular support of the Schwartz kernel and implicit definition} A  computation shows that $f \ast \mu_t(u, x)$ corresponds to an average of $f$ over the ellipse in $\R^3$ given by 
\begin{equation*}
  S_{u,x,t} := \big\{(v,z) \in \R \times \R^2 :  v - u + b(x_1z_2 - x_2z_1) = 0, \quad
        |x - z|^2 - t^2 = 0 \big\}.
\end{equation*}
Furthermore, using the identity $(x_1z_1 + x_2z_2)^2 + (x_1z_2 - x_2z_1)^2 = |x|^2|z|^2$, 
 one checks that  $(v,z)\in S_{u,x,t}$ satisfies 
\begin{equation}\label{implicit formulation}
    \Phi_t(u,|x|;v, |z|) = 0
\end{equation}
where $\Phi_t(u,r;v,\rho):=\Phi(u, r, t;v, \rho)$ and
\begin{equation}\label{defining function}
  \Phi(u, r, t;v, \rho) := (u-v)^2 - \Big(\frac{b}{2}\Big)^2\big( 4r^2\rho^2 - (r^2 + \rho^2 - t^2)^2 \big).
\end{equation}
Below we relate explicitly $f*\mu_t$ to an operator acting on functions of the two variables $(v,\rho)$, with a Schwartz kernel $\delta\circ\Phi$ which will define this  integral operator as a weakly singular Radon transform.

In the forthcoming sections it will be necessary to carry out many computations involving $\Phi$. For the reader's convenience, a dictionary of derivatives of this function is provided in Appendix~\ref{derivative appendix}.

\subsection{ Properties of $\mathbb{H}$-radial functions} A function $f \colon \bbH^1 \to \C$ is $\mathbb{H}$-radial if and only if there exists some function $f_0 \colon \R \times  [0,\infty) \to \C$ such that
\begin{equation}\label{radial}
 f(u,x)= f_0(u,|x|).
\end{equation}
Using the fact that $R^{\top}BR=B$ for $R\in \mathrm{SO}(2)$, if $f$ and $g$ are $\mathbb{H}$-radial, then $f*g$ is $\bbH$-radial, and we have
\[
 (f*g)_0(u,r)= \!\int_0^{2\pi} \!\!\! \int_\R \! \int_0^\infty \!\! f_0(v,\rho)g_0(
u-v-br\rho\sin\vth,\sqrt{r^2+\rho^2-2r\rho\cos\vth})\,
\rho \, \ud  \rho  \ud v \ud\vth.
\]
This observation extends to $\mathbb{H}$-radial measures and, in particular, if $f$ is $\mathbb{H}$-radial, then  $f*\mu_t$ is $\bbH$-radial, and we get  
\begin{align} (f*\mu_t)_0(u,r)&= \frac{1}{2\pi}\int_{-\pi}^{\pi}
f_0(u-btr\sin \vth ,\sqrt{r^2+t^2-2rt\cos\vth
})\, \ud\vth
\notag
\\
&=
\sum_{\pm} \frac{1}{2\pi}\int_0^{\pi}
f_0(u\pm btr\sin \vth, \sqrt{r^2+t^2-2rt\cos\vth}) \,
\ud\vth.
\label{sym}
\end{align}

Applying polar coordinates in the planar slices $\{u\} \times \R^2$, given $p>2$ and $f$ as in \eqref{radial}, the goal is to establish the inequality 
\begin{equation}\label{radial maximal}
\Big(\int_{-\infty}^\infty\int_0^\infty\big|(Mf)_0(u,r)\big| ^p r\, \ud r \ud u\Big)^{1/p}\lesssim
\Big(\int_{-\infty}^\infty\int_0^\infty\big|f_0(v,\rho)|^p \rho \, \ud\rho \, \ud v\Big)^{1/p}.
\end{equation}

\subsection{ A weakly singular Radon-type  operator on $\R^2$} By the implicit definition of the circle $S_{u,x,t}$ from \eqref{implicit formulation}, the function $(f*\mu_t)_0$ corresponds to an integral operator associated with the  curve 
\begin{equation*}
    \Sigma_{u,r,t} := \big\{ (v,\rho) \in \R \times  [0,\infty) : \Phi_t(u,r;v,\rho) = 0 \big\}.
\end{equation*}
It is easy to see that $\Sigma_{u,r,t}$ is smooth whenever $r \neq t > 0$. If $r = t > 0$, then there is a unique singular point on the curve at the point where it touches the $v$ axis. See Figure~\ref{singular curve figure}. Furthermore, any $(v,\rho) \in \Sigma_{u,r,t}$ satisfies
\begin{equation}\label{orthogonality relation}
    |r-t| \leq \rho \leq r + t \quad \textrm{and} \quad |u - v| \leq |b| \min\{r\rho, rt, t\rho\};
\end{equation}
these bounds follow since for $(v,\rho) \in \Sigma_{u,r,t}$
\begin{align}
\notag   0 \leq (b/2)^{-2}(u-v)^2 &= 4r^2 \rho^2 - (r^2 + \rho^2 - t^2)^2\\ \notag &= 4r^2 t^2 - (r^2 + t^2 - \rho^2)^2\\& = 4t^2 \rho^2 - (t^2 + \rho^2 - r^2)^2. \label{identities Phi=0}
\end{align}


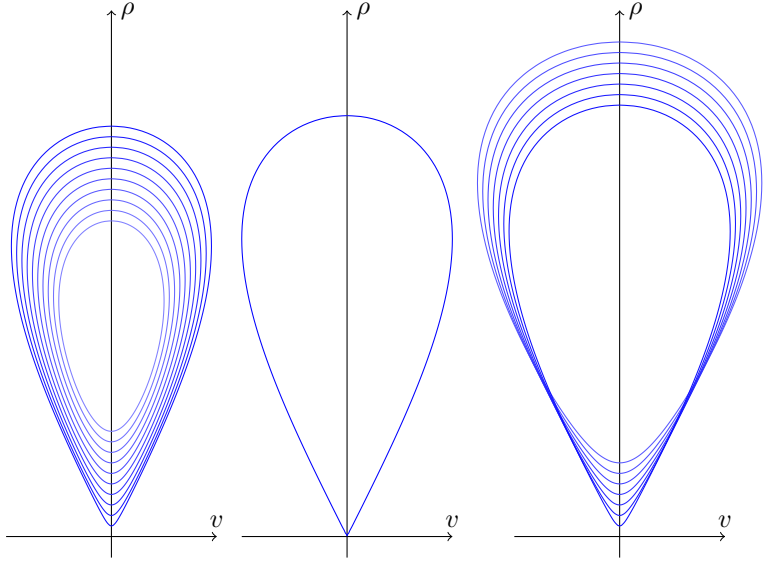
\begin{figure}
\begin{center}
\begin{tikzpicture}[scale=2.8]

      \draw[->] (0,-0.1) -- (0,2.5) node[right] {$\rho$};
      \draw[->] (-0.5,0) -- (0.5,0) node[above] {$v$};

\foreach \y in {50,55, ..., 95}
    {     
\def \a{\y/100};     
      \draw[domain=-1.772:1.772, samples=100, smooth,variable=\x,blue!\y!white] plot ({\a*(1/2)*sin(deg(\x^2))},{sqrt(\a*\a + 1 - \a*2*cos(deg(\x^2)))});
      
}      
    \end{tikzpicture}
    \begin{tikzpicture}[scale=2.8]
\def \a {1}
      \draw[->] (0,-0.1) -- (0,2.5) node[right] {$\rho$};
      \draw[->] (-0.5,0) -- (0.5,0) node[above] {$v$};
      \draw[domain=-1.772:1.772, samples=100, smooth,variable=\x,blue] plot ({\a*(1/2)*sin(deg(\x^2))},{sqrt(\a*\a + 1 - \a*2*cos(deg(\x^2)))});
    \end{tikzpicture}
    \begin{tikzpicture}[scale=2.8]
\def \a {6/5}
      \draw[->] (0,-0.1) -- (0,2.5) node[right] {$\rho$};
      \draw[->] (-0.5,0) -- (0.5,0) node[above] {$v$};
      
    \foreach \y in {65, 70, ..., 95}
   {     
\def \a{(2 - \y/100)};     
      \draw[domain=-1.772:1.772, samples=100, smooth,variable=\x,blue!\y!white] plot ({\a*(1/2)*sin(deg(\x^2))},{sqrt(\a*\a + 1 - \a*2*cos(deg(\x^2)))});
      
}      
    \end{tikzpicture}

\caption{The curves $\Sigma_{0,r,t}$ for $t$ fixed and $r < t$ (left), $r = t$ (centre) and $r > t$ (right). When $r = t$ the curve has a unique singular point on the $v$ axis.}
\label{singular curve figure}
\end{center}
\end{figure}

Consider the integral operator in two dimensions 
defined on functions of the variables $(v,\rho)$ by
\begin{equation}\label{averaging operator}
    A_t f(u,r) \equiv A_{p,t} f(u,r):= \int_{-\infty}^{\infty}\int_0^{\infty} f(v,\rho) r^{1/p}\rho^{1-1/p} \delta\big(\Phi_t(u,r;v,\rho)\big)\ud v\ud \rho.
\end{equation}
In view of 
\eqref{radial maximal}, Theorem~\ref{circular maximal theorem}
 will be  a consequence of the following maximal estimate in the Euclidean plane.

\begin{theorem}\label{main theorem} For all $p > 2$,
\begin{equation*}
\Big(\int_0^\infty\int_{-\infty}^\infty \big(\sup_{t>0} | A_{t}f(u,r)|
\big) ^p \,\ud u\,\ud r \Big)^{1/p} \lesssim
\Big(\int_0^\infty\int_{-\infty}^\infty\big|f(v,\rho)|^p  \,\ud v\,\ud \rho \Big)^{1/p}.
\end{equation*}
\end{theorem}

Note that the $r^{1/p}\rho^{-1/p}$ factor  featured in the averaging operator in \eqref{averaging operator} arises from the weights induced by the polar coordinates in \eqref{radial maximal}. In order to relate
Theorem~\ref{circular maximal theorem} to  Theorem \ref{main theorem} we have to 
write  for $\bbH$-radial test functions the expression   $(f*\mu_t)_0(u,r)$ in terms of 
the distribution $\delta\circ\Phi_t$ which is understood as a weak limit of $\chi_\eps\circ\Phi_t$ as $\eps\to 0$.
The calculation, which is given in the proof of Lemma \ref{deltaPhi-lemma} below, is standard, (\textit{cf.}  \cite[p.498]{Stein1993} which provides a proof for a local version). For the sake of convenience we include below a direct proof for our example.

In what follows we shall use, for a continuous compactly supported  function $g$   the integral notation
$g(c)=\int g(v) \delta(c-v)\, \ud v$ for the pairing of $g$ with the Dirac measure at $c$.
We also let $\chi_\eps(s)=\eps^{-1}\chi(\eps^{-1} s)$ with $\chi$ even and supported in $(-1/2,1/2)$ such that $\int \chi \ud s=1$. We shall prove the following.

\begin{lemma} \label{deltaPhi-lemma}
Let $f\in C^1(\bbH^1)$ be $\bbH$-radial and compactly supported in 
$\{(v,\rho) \in \R^2 :\rho>0\}$. Then, for any $r>0$,
\begin{align*}(f*\mu_t)_0(u,r)
&= \frac{|b|}{\pi}
\lim_{\eps\to 0} 
 \int_0^\infty\int_\bbR \chi_\eps(\Phi_t(u,r;v,\rho)) f_0(v,\rho)\rho \,\ud v \, \ud\rho 
\\&=: \frac{|b|}{\pi}\int_0^\infty\int_\bbR \delta(\Phi_t(u,r;v,\rho)) f_0(v,\rho)\rho \,
\ud v \,\ud\rho .
\end{align*}
\end{lemma}

With the above lemma in hand, Theorem \ref{main theorem} immediately implies Theorem \ref{circular maximal theorem}.

\begin{proof} [Proof that Theorem \ref{main theorem} implies
Theorem~\ref{circular maximal theorem}]
We prove the a priori inequality for smooth
$\bbH$-radial functions which are compactly supported in $\{(u,y)\in \R^3:|y|\neq 0\}$. 
By Lemma \ref{deltaPhi-lemma}
\[ r^{1/p} (Mf)_0(u,r)= \frac{|b|}{\pi} \,\sup_{t >0} \, A_t [\rho^{1/p}f_0](u,r),\] and the assertion follows.
\end{proof}

\begin{proof}[Proof of Lemma \ref{deltaPhi-lemma}]
We use \eqref{sym} and  make a change of variable by setting
$$
\rho=\rho(\vth)=\sqrt{r^2+t^2-2rt\cos\vth}, \quad 0<\vth<\pi.
$$
Observe that the condition
$ 0<\vth<\pi$ is equivalent with $|r-t|<\rho<r+t$.
Then
\begin{equation*}
u\pm btr\sin \vth = u\pm btr\sqrt{1-\Big(\frac{r^2+t^2-\rho^2}{2rt}\Big)^2} = u\pm \frac{b}{2} G(r,t,\rho)
\end{equation*}
where
$$ 
G(r,t,\rho) := \sqrt{4r^2t^2-(r^2+t^2-\rho^2)^2}.
$$
For  the
relevant range $|r-t|<\rho<r+t$ the root is well defined (as $\sin
\vth>0$), and  we have the factorisation
\Be  \label{factorization}  G(r,t,\rho)
=\Big(
(r+t+\rho)(r+t-\rho)(r-t+\rho)(t-r+\rho)
\Big)^{1/2}.
\Ee
We calculate   $$\frac{d\rho}{d\vth}= \rho^{-1} rt \sin(
\vth)  =(2\rho)^{-1}G(r,t,\rho) $$ and thus
\begin{align*} \pi (f*\mu_t)_0(u,r)&=
\sum_{\pm} 
\int_{|r-t|}^{r+t}
f_0(u\pm \tfrac b2 G(r,t,\rho),\rho) \frac{\rho}{G(r,t,\rho)} \, \ud\rho
\\
&=\sum_{\pm} 
\lim_{\eps\to 0} 
\int_{|r-t|}^{r+t}\int_\R \rho f_0(v,\rho) \chi_\eps(u\pm \tfrac b2 G(r,t,\rho)-v)
 \, \ud v \, \frac 1{G(r,t,\rho)} \, \ud \rho.
\end{align*}

Let $U$ be an open interval with  compact closure contained in  $(0,\infty)$
such that $\supp (f_0(u,\cdot)) \subset U$ for all $u\in \bbR$. Let
$U(r,t)= \{\rho\in U: |r-t|<\rho<r+t\}$.
We observe from \eqref{factorization} 
that for fixed $r,t$ with $r\neq t$, the function  $\rho \mapsto |G(r,t,\rho)|^{-1}$ satisfies 
\Be\label{p-integrability} \int_{U(r,t)} |G(r,t,\rho)|^{-p} \ud\rho \le C(r,t)<\infty \quad \text{ for  $1\le p<2$, }
\Ee which we use for $p>1$. Let $E_\eps(r,t)= \{ \rho\in U(r,t): G(r,t,\rho)\le \eps^{1/2}\}$ and $F_\eps(r,t)=U(r,t)\setminus E_\eps(r,t)$. We use H\"older's inequality 
to bound 
\begin{align*}\int_{E_\eps(r,t)}\int_\R \rho \, &  |f_0(v,\rho) |  |\chi_\eps(u\pm \tfrac b2 G(r,t,\rho)-v)| \,  
 \ud v \, \frac 1{G(r,t,\rho)} \, \ud \rho
 \\
 &\lc _{r,t,f} |E_\eps(r,t)|^{1/p'} C(r,t)^{1/p} = O(\eps^{(p-1)/2}),
 \end{align*}
 noting that \eqref{p-integrability} implies $|E_{\varepsilon}|\lesssim_{r,t} \varepsilon^{p/2}$.  For $\rho\in F_\eps(r,t,\rho)$ we use the change of variable
 \[w\to v_\pm (w)=u \pm \frac b2 G(r,t,\rho)-(u-w)^2+\frac{b^2}{4} G(r,t,\rho)^2\]
 which is one-to-one on $(u,\infty)$ and on $(-\infty, u)$
 and satisfies 
 \[u-v_\pm(w)\pm \frac b2 G(r,t,\rho)= (u-w)^2-\frac{b^2}{4}G(r,t,\rho)^2 .\]
 We have $|v'(w)|=2|u-w|$, 
 and $|v(w)-w|=O(\eps)$ on the support of the integrand, and therefore also  
 $|u-w|= G(r,t,\rho)|b|/2+O(\eps)$.
 Hence, by Taylor expansion of $f(v,\rho)$ around $(w,\rho)$,
 \begin{align*}\int_{F_\eps(r,t)}& \int_\R   \rho  f_0(v,\rho) \chi_\eps(u\pm \tfrac b2 G(r,t,\rho)-v)
  \, \ud v \, \frac 1{G(r,t,\rho)} \, \ud \rho
  \\
 &= \frac 12
 \int_{F_\eps(r,t)}\int_\R \rho f_0(v(w),\rho) \chi_\eps((u-w)^2-(\tfrac b2 G(r,t,\rho))^2)
  \frac{2|u-w|}{G(r,t,\rho)} \,\ud w\, \ud \rho
  \\
 &= \frac{|b|}{2} 
 \int_{F_\eps(r,t)}\int_\R \rho f_0(w,\rho) \chi_\eps((u-w)^2-(\tfrac b2 G(r,t,\rho))^2)
  \,\ud w\, \ud \rho + O(\eps^{1/2})
 \end{align*} 
 and by using the estimate $|E_\eps(r,t)| \lc_{r,t} \eps^{p/2}$  the last displayed expression is equal to 
\[ \frac{|b|}{2} 
 \int_{|r-t|}^{r+t}\int_\R \rho f_0(w,\rho) \chi_\eps((u-w)^2-(\tfrac b2 G(r,t,\rho))^2)
  \,\ud w\, \ud \rho + O(\eps^{1/2}),\] 
 for both choices of $\pm$. We sum in $\pm$ and, using \eqref{identities Phi=0}, obtain,  for $r\neq t$,
 \[ (f*\mu_t)_0(u,r)= \frac{|b|}{\pi}
 \int_{|r-t|}^{r+t}\int_\R \rho f_0(w,\rho) \chi_\eps(\Phi_t(u,r;w,\rho))  \ud w\, \ud\rho
 +O(\eps^{(p-1)/2}).
 \]
Letting $\eps\to 0$ concludes the proof.
\end{proof}

\subsection{ A local variant of the maximal operator} The main work in proving Theorem \ref{main theorem} will be to establish the following local variant.

\begin{theorem}\label{thm local maximal function}
For all $p>2$,
$$
\big\| \sup_{1 \leq t \leq 2} |A_t f| \big\|_{L^p(\R \times (0,\infty))} \lesssim \|  f\|_{L^p(\R \times (0,\infty))}.
$$
\end{theorem}

This will be established in \S\ref{curvature considerations section} -- \S\ref{Lp section}.
The passage from Theorem~\ref{thm local maximal function} to the global result in Theorem~\ref{main theorem} is postponed until \S\ref{sec:global}.




\section{Curvature considerations}\label{curvature considerations section} As indicated in the introduction and Section~\ref{sec:proof strategy}, various `curvatures', which feature extensively in the analysis of generalised Radon transforms, are fundamental to the proof of Theorem~\ref{thm local maximal function}. In this section these concepts are reviewed and some calculations are carried out in relation to the operator $A_t$ introduced above. 

\begin{definition} A \emph{smooth family of defining pairs} $[\Phi;\fa]$ consists of a pair of functions $\fa \in C^{\infty}(\R^2 \times \R \times \R^2)$ and $\Phi \in C^\infty$ defined on a neighbourhood of $\supp \fa$ satisfying
\begin{equation*}
    \nabla_{(x,z)} \Phi(x,t;z) \neq 0 \qquad \textrm{for $(x,t;z) \in \mathrm{supp}\,\fa$}.
\end{equation*}
The $t$ variable will play a preferred r\^ole in the forthcoming analysis. For any fixed $t \in \R$ let $\Phi_t(x;z) := \Phi(x,t;z)$ and $\fa_t(x;z) := \fa(x,t;z)$; then $[\Phi_t;\fa_t]$ is referred to as a \emph{defining pair}. The Schwartz kernel 
$\fa \, \delta\circ\Phi$ is then well defined, and the corresponding integral  
 operator $A[\Phi_t;\fa_t]f(x)$ mapping test functions to distributions is 
given by
the pairing
\begin{equation}\label{general average}
    \biginn{A[\Phi_t;\fa_t]f}{g}
     := \iint_{\R^2\times \R^2} g(x)f(z) \fa_t(x;z) \delta\big(\Phi_t(x;z)\big) \,\ud z \,  \ud x.
\end{equation}
\end{definition}

\begin{key example}
For the defining function $\Phi_t$ in  \eqref{defining function}, where $t \sim 1$, with the identification of coordinates $(u,r)=(x_1,x_2)$, $(v,\rho)=(z_1,z_2)$,  the distribution $\delta\circ\Phi$ is defined when paired with $g(u,r)f(v,\rho)$ where $g$ and $f$ are compactly supported $C^\infty$ functions with support away from $\{r=0\}$ and $\{\rho=0\}$ respectively.
The calculations in Lemma \ref{deltaPhi-lemma} show that in this case 
$A[\Phi_t;\fa_t]f(x)$ is pointwise defined for $x_2\neq 0$, as long as $f\in C^\infty_0(\R^2)$ with $\supp\,f\subset \{y \in \R^2:y_2\neq 0\}$.
\end{key example}

\subsection{ Rotational curvature}\label{rotational curvature section} Given a defining pair $[\Phi_t; \fa_t]$ the \emph{rotational curvature} $\mathrm{Rot}(\Phi_t)$ is defined to be the function of $(x;z) \in \R^2 \times \R^2$ given by the determinant of the Monge--Amp\`ere matrix
\begin{equation*}
 \mathfrak{M}(\Phi_t) := \begin{bmatrix}
        \Phi_t & (\partial_z \Phi_t)^{\top} \\
        \partial_x \Phi_t & \partial_{x z}^2 \Phi_t
    \end{bmatrix}.
    \end{equation*}
    Note that $\mathfrak M(\Phi_t)$ is the mixed Hessian $D^2_{(\theta, x), (s,z)} \Psi_t|_{\theta = s = 1}$ of the function
    \begin{equation*}
        (\theta,x, s, z)\mapsto \Psi_t(\theta, x; s, z) := \theta s \Phi_t(x;z)
    \end{equation*}
    and, more generally, 
\begin{equation*}
    D_{(\theta,x), (s,z)}^2 \Psi_t = \begin{bmatrix} \Phi_t& s\partial_z\Phi_t^\top
\\
\theta\partial_x\Phi_t &\theta s\partial_{xz}\Phi_t. 
 \end{bmatrix}.
\end{equation*}    

   It is well-known (see, for instance,~\cite[Chapter XI, \S 3]{Stein1993}) that the behaviour of $\mathrm{Rot}(\Phi_t)$ on the incidence relation $\{\Phi=0\}$ plays an important r\^ole in determining the mapping properties of averaging operators $A[\Phi_t;\fa_t]$ on $L^2$-Sobolev  spaces as well as the $L^p$ theory of their  maximal variants. It is of particular interest to identify points where the rotational curvature vanishes together with the defining function.
   
\begin{key example}   For the defining function $\Phi_t$ in question, as introduced in \eqref{defining function}, 
 we now have $(x_1,x_2)\equiv (u,r)$ and $(z_1,z_2)= (v,\rho)$ and
\begin{align*}
\fM(\Phi_t) & = \begin{bmatrix}
\Phi_t& \partial_v\Phi_t& \partial_\rho\Phi_t
\\ \partial_u\Phi_t& \partial^2_{uv}\Phi_t &\partial_{u\rho}^2\Phi_t
\\
\partial_r\Phi_t& \partial^2_{rv } \Phi_t&\partial^2_{r\rho } \Phi_t
\end{bmatrix} 
\\&=
\begin{bmatrix} 
\Phi_t
&-2(u-v)& -b^2\rho(r^2-\rho^2+t^2)
\\
2(u-v)&-2&0
\\
-b^2r(\rho^2-r^2+t^2)&0& -2b^2r\rho
\end{bmatrix}.
\end{align*}
Then, 
one computes that
\begin{multline*}
    \det \fM(\Phi_t)= 
    2b^2r\rho\big( b^2(\rho^2-r^2+t^2)(r^2-\rho^2+t^2) -4(u-v)^2 +2\Phi \big) 
    \\=2b^4r\rho\big((\rho^2-r^2+t^2)(r^2-\rho^2+t^2)-4r^2\rho^2+(r^2+\rho^2-t^2)^2\big)  +4b^2r\rho\Phi_t .
\end{multline*} 
Setting $\Phi_t=0$ one obtains after further computation 
\begin{equation}\label{fixed t rotational curvature}
   \mathrm{Rot}(\Phi_t) (u,r;v,\rho) = 4b^4rt^2\rho (t^2 - r^2 - \rho^2)\qquad \textrm{for $(v,\rho) \in \Sigma_{u,r,t}$}.
 \end{equation}
Thus, $\mathrm{Rot}(\Phi_t)$ vanishes along the co-ordinate hyperplanes $r = 0$, $t = 0$ and $\rho = 0$ and also, more significantly, along the hypersurface $t^2 = r^2 + \rho^2$.

Continuing with $\Phi_t$ as in \eqref{defining function}, the rotational curvature and $t$-derivative of the defining function are related via the identity
\begin{equation}\label{key identity}
    \mathrm{Rot}(\Phi_t)(u,r;v,\rho) = 4b^2 r t \rho (\partial_t \Phi_t)(u,r;v,\rho).
\end{equation}
A relationship of this kind was previously noted in~\cite{MS2004} in the context of the spherical maximal operator on $\mathbb{H}^n$ for $n \geq 2$. Here, in close analogy with~\cite{MS2004}, the identity \eqref{key identity} will be important in the analysis near the singular hypersurface $t^2 = r^2 + \rho^2$.

Rather than freezing $t$ for the computation of the rotational curvature, it is sometimes useful to freeze $r$ and set 
$$\Phi_r^{\star }(u,t;v,\rho) := \Phi_t(u,r;v,\rho).
$$
 A similar computation to the one above yields in this case
 \begin{equation}
   \label{fixed r rotational curvature}
   \mathrm{Rot}(\Phi_r^{\star }) (u,t;v,\rho) = 4b^4 r^2t\rho(r^2 - t^2 - \rho^2) \qquad \textrm{for $(v,\rho) \in \Sigma_{u,r,t}$.}
\end{equation}
\end{key example}

\subsection{  The fold conditions}  For the defining function from \eqref{defining function}, the vanishing of the rotational curvature along the hypersurface $t^2=r^2+\rho^2$ corresponds to a \textit{two-sided fold singularity}.

\begin{definition} A defining function $\Phi_{t_0}$ satisfies the  \emph{two-sided fold condition} at $(x_0;z_0) \in \R^2 \times \R^2$ if the following hold:
\begin{enumerate}[i)]
\item $\Phi_{t_0}(x_0;z_0) = 0$ and $\mathrm{Rank}\, \mathfrak{M}(\Phi_{t_0})(x_0;z_0) =2$. 
\item If $U=(u_1, u_2, u_3)$ and $V=(v_1, v_2, v_3) \in \R^3$ span the cokernel and kernel of $\mathfrak{M}(\Phi_{t_0})(x_0;z_0)$, respectively, then
\begin{equation*}
\begin{split}
&     \bigg\langle \partial_{zz}^2 \Big\langle U,  \begin{bmatrix}
    \Phi_{t_0} \\
    \partial_x \Phi_{t_0}
    \end{bmatrix}
    \Big\rangle\Big|_{(x_0; z_0)} V''\,,\, V'' \bigg \rangle\neq 0, \\
&    \bigg\langle \partial_{xx}^2 \Big\langle V,  \begin{bmatrix}
    \Phi_{t_0} \\
    \partial_z \Phi_{t_0}
    \end{bmatrix}
    \Big\rangle\Big|_{(x_0; z_0)} U''\,,\, U'' \bigg \rangle \neq 0,
    \end{split}
\end{equation*}
where $U''=(u_2, u_3)$ and $V''=(v_2, v_3)$. 
\end{enumerate}
\end{definition}

As a consequence of the fold condition, $\mathfrak{M}(\Phi_{t_0})(x_0;z_0)$ may be transformed into a `normal form'. In particular, there exist $\mathbf{X}, \mathbf{Z} \in \mathrm{GL}(3, \R)$ satisfying
\begin{itemize}
    \item $\mathbf{X}e_3 = U$ and $\mathbf{X}e_1$, $\mathbf{X}e_2$ are orthogonal to 
    \begin{equation*}
   \Big(0,  \partial_{xx}^2 \Big\langle V,  \begin{bmatrix}
    \Phi_{t_0} \\
    \partial_z \Phi_{t_0}
    \end{bmatrix}
    \Big\rangle\Big|_{(x_0; z_0)} U''\Big),
    \end{equation*}
    \item $\mathbf{Z}e_3 = V$ and $\mathbf{Z}e_1$, $\mathbf{Z}e_2$ are orthogonal to 
    \begin{equation*}
    \Big(0, \partial_{zz}^2 \Big\langle U,  \begin{bmatrix}
    \Phi_{t_0} \\
    \partial_x \Phi_{t_0}
    \end{bmatrix}
    \Big\rangle\Big|_{(x_0; z_0)} V''\Big),
    \end{equation*}
\end{itemize}
where $e_j$ denote the standard basis vectors in $\R^3$, and therefore
\begin{equation*}
\mathbf{X}^{\top} \circ \mathfrak{M}(\Phi_{t_0})(x_0;z_0) \circ \mathbf{Z} = \begin{bmatrix}
\mathbf{M}(x_0,t_0;z_0) & 0 \\
0 & 0
\end{bmatrix}
\end{equation*}
for $\mathbf{M}(x_0, t_0;z_0)$ a non-singular $2 \times 2$ matrix. 

\begin{key example} For the defining function $\Phi_t$ from \eqref{defining function}, if $\Phi_{t_0}$ and $\mathrm{Rot}(\Phi_{t_0})$ both vanish at $(x_0; z_0) = (u_0,r_0;v_0,\rho_0)$ and $r_0t_0\rho_0 \neq 0$, then
\begin{equation}\label{fold condition 1}
   U := \begin{bmatrix}
           1  \\
           -(u_0 - v_0)\\
           -r_0
         \end{bmatrix} \qquad \textrm{and} \qquad    V := \begin{bmatrix}
           1  \\
           u_0 - v_0 \\
           -\rho_0
         \end{bmatrix}
\end{equation}
span the cokernel and kernel of $\mathfrak{M}(\Phi_{t_0})(x_0;z_0)$, respectively. Moreover,
\begin{align*}
    \bigg\langle \partial_{zz}^2 \Big \langle U,  \begin{bmatrix}
    \Phi_{t_0} \\
    \partial_x \Phi_{t_0}
    \end{bmatrix}
    \Big\rangle\Big|_{(x_0; z_0)} V''\,,\, V'' \bigg \rangle  &= 2b^2\rho_0^2(r_0^2 + \rho_0^2) >0, \\
     \bigg\langle \partial_{xx}^2 \Big\langle V,  \begin{bmatrix}
    \Phi_{t_0} \\
    \partial_z \Phi_{t_0}
    \end{bmatrix}
    \Big\rangle\Big|_{(x_0; z_0)} U''\,,\, U'' \bigg \rangle  &= 2b^2r_0^2(r_0^2 + \rho_0^2) > 0
\end{align*}
and the matrices $\mathbf{X}$ and $\mathbf{Z}$ can be taken to be
\begin{equation}\label{fold condition 2}
\mathbf{X} := \begin{bmatrix} 
    1 & 0             & 1  \\
    0 &  -b^2 r_0^3 & -(u_0 - v_0)\\
    0 &  u_0 - v_0       & -r_0
    \end{bmatrix},   \qquad 
    \mathbf{Z} := \begin{bmatrix} 
    1 & 0         & 1  \\
    0 & b^2 \rho_0^3 &  u_0 - v_0\\
    0 &  u_0 - v_0  & -\rho_0 
    \end{bmatrix}.
\end{equation}
\end{key example}
 
\noindent{\it Remark.} For standard incidence relations $\cM\subset \bbR^2_L\times \bbR^2_R$, where $\R^2_L \equiv \R^2_R \equiv \R^2$ and $\cM=\{\Phi=0\}$ with $\nabla \Phi$ bounded below,  
 the two sided fold condition is equivalent to the more common assumption
 (\cite{Melrose-Taylor}, \cite{phong-stein91}) that the projections $\pi_L$, $\pi_R$ mapping the 
 conormal bundle $N^*\cM$ to $T^*\bbR^2_L$, $T^*\bbR^2_R$ have fold singularities. 
 



 \subsection{ Individual curves} It is also useful to consider the curvatures of the individual curves in the curve distribution induced by a defining family $\Phi$. In particular, for fixed $(x,t)$ the non-vanishing of the curvature of $\Sigma_{x,t}:=\{ z \in \R^2 : \Phi_t(x; z) =0\}$ is equivalent to the non-vanishing (on $\Sigma_{x,t}$) of 
 \begin{equation}\label{plane curvature definition}
     \kappa(\Phi_t)(x;z) := \det 
     \begin{bmatrix}
     0 & (\partial_z \Phi_t)^{\top}(x;z) \\
     \partial_z \Phi_t(x;z) & \partial_{zz}^2 \Phi_t(x;z)
     \end{bmatrix}.
 \end{equation}

\begin{key example} For the defining family $\Phi$ as introduced in \eqref{defining function}, the curves have non-vanishing curvature whenever $r \neq t$. To see this, note that
\begin{align*}
\kappa(\Phi_t) & = \det \begin{bmatrix}
\Phi_t & \partial_v\Phi_t& \partial_\rho\Phi_t
\\ \partial_v\Phi_t & \partial^2_{vv}\Phi_t &\partial_{v\rho}^2\Phi_t
\\
\partial_\rho\Phi_t& \partial^2_{\rho v } \Phi_t&\partial^2_{\rho \rho } \Phi_t
\end{bmatrix} 
\\&= \det
\begin{bmatrix} 
\Phi_t
&-2(u-v)& -b^2\rho(t^2 + r^2-\rho^2)
\\
-2(u-v)&2&0
\\
-b^2\rho(t^2 + r^2 - \rho^2)&0& -b^2 (t^2+r^2-3\rho^2) 
\end{bmatrix},
\end{align*}
which after a computation reduces, for $(v,\rho) \in \Sigma_{u, r, t}$, i.e. $\Phi_t=0$,   to
$$
\kappa(\Phi_t)(u,r,t;v,\rho)= b^4 \big( \rho^6 - 3(r^2+t^2)\rho^4 +3(r^2-t^2)^2 \rho^2 - (r^2-t^2)^2(r^2+t^2)  \big).
$$
 Thus, $\kappa(\Phi_t)(u,r,t;v,\rho) = \wp_{r,t}(\rho^2)$, where  $\wp_{r,t}$ is a cubic polynomial with coefficients depending on $r,t$. We first calculate
  $\wp((r-t)^2)=-8b^4 r^2t^2(r-t)^2 $.
  One may  verify that $\wp_{r,t}$ is a decreasing function on the interval $[(r-t)^2, (r+t)^2]$, 
 leading to the lower bound
\begin{equation}\label{plane curve curvature bound} 
    |\kappa(\Phi)(u,r,t;v,\rho)| \geq 8b^4r^2t^2 (r-t)^2 \qquad \textrm{for all $(v,\rho) \in \Sigma_{u,r,t}$}
\end{equation}
Thus, the curves have non-vanishing curvature if $r \neq t$, as claimed.

\end{key example}




 \subsection{ Cinematic curvature} It is also necessary to analyse the average operator from the perspective of the \textit{cinematic curvature condition} of~\cite{Sogge1991}. 
 
 \begin{definition} A smooth family of defining pairs $[\Phi;\fa]$ is said to satisfy the \textit{projection condition} if
 \begin{equation*}
     \mathrm{Proj}(\Phi) := \det \begin{bmatrix}
    \partial_{\vec{x}}\,\Phi & \partial_{\vec{x} z}^2\Phi
    \end{bmatrix}
 \end{equation*}
 is non-vanishing on an open neighbourhood $U$ of $\mathrm{supp}\, \fa$. Here $\vec{x} = (x,t) \in \R^2\times \R$.
\end{definition}

Fixing $\vec{x} \in \R^2\times \R$, the projection condition implies that the map 
\begin{equation*}
 (U \cap \Sigma_{\vec{x}}) \times \R \to \R^3; \quad (z;\theta) \mapsto \theta\partial_{\vec{x}} \Phi(\vec{x};z)
\end{equation*}
is a diffeomorphism and therefore its image $\Gamma_{\vec{x}}$ is an immersed submanifold of $\R^3$. If $\zeta := \theta\partial_{\vec{x}} \Phi(\vec{x};z) \in \Gamma_{\vec{x}}$, then a basis for $T_{\zeta} \Gamma_{\vec{x}}$ is given by the vector fields
\begin{equation}\label{tangent vector fields}
    \mathbf{T}^1 := \partial_{\vec{x}\,}\Phi, \qquad \mathbf{T}^2 := (\mathbf{T}^2_1, \mathbf{T}^2_2, \mathbf{T}^2_3) \quad \textrm{where $\mathbf{T}^2_j := 
    \det \begin{bmatrix}
    \partial_z \Phi & \partial_{z}\partial_{x_j} \Phi
    \end{bmatrix}$}
\end{equation}
evaluated at $(\vec{x};z)$; this may be seen computing the tangent vectors of the parametrisation $\sigma_{\vec{x}}$ below. Note that $\Gamma_{\vec{x}}$ is clearly a cone and therefore has everywhere vanishing Gaussian curvature. If at every point on $\Gamma_{\vec{x}}$  there is a non-zero principal curvature, then $[\Phi;\fa]$ is said to satisfy the \textit{cinematic curvature condition} (see~\cite{Sogge1991} or~\cite{Mockenhaupt1993} for further details).

\begin{definition}\label{cinematic curvature definition} For any defining family $\Phi$ let
\begin{equation*}
    \mathrm{Cin}(\Phi) :=  \det \begin{bmatrix} \mathbf{S} & \mathbf{T}^1 & \mathbf{T}^2
    \end{bmatrix}
\end{equation*}
where $\mathbf{S} = \mathbf{S}^1 - \mathbf{S}^2$ where $\mathbf{S}^i = (\mathbf{S}_1^i, \mathbf{S}_2^i, \mathbf{S}_3^i)$ for
\begin{equation*}
    \mathbf{S}^1_j := \det
    \begin{bmatrix}
    0 & (\partial_{z} \Phi)^{\top}\\
    \partial_{z} \Phi & \partial_{z z}^2\partial_{x_j} \Phi 
    \end{bmatrix}, \quad 
    \mathbf{S}^2_j := \det\begin{bmatrix}
    0 & (\partial_{z} \partial_{x_j} \Phi)^{\top}  \\
    \partial_{z} \Phi & \partial_{z z }^2 \Phi 
    \end{bmatrix}.
\end{equation*} 

\end{definition}

 If $[\Phi;\fa]$ satisfies the projection condition, then  the cinematic curvature condition is equivalent to the non-vanishing of $\mathrm{Cin}(\Phi)(\vec{x};z)$ whenever $z \in \Sigma_{\vec{x}}$. Indeed, fix $\vec{x}$ and let $\gamma_{\vec{x}} \colon [0,1] \to \Sigma_{\vec{x}}$ denote a unit speed parametrisation of $\Sigma_{\vec{x}}$; this induces a parametrisation $\sigma_{\vec{x}} \colon (\theta, s) \mapsto \theta\partial_{\vec{x}} \Phi(\vec{x};\gamma_{\vec{x}}(s))$ of the cone $\Gamma_{\vec{x}}$. The cinematic curvature condition is then  equivalent to the non-vanishing of
\begin{equation}\label{cone curvature}
    \det\begin{bmatrix}
    \partial_{ss} \sigma_{\vec{x}}(\theta, s)  &
    \partial_{\theta} \sigma_{\vec{x}}(\theta, s)  &
    \partial_s \sigma_{\vec{x}}(\theta, s) 
    \end{bmatrix}
\end{equation}
and a  computation shows that \eqref{cone curvature} is equal to
$-|\theta|^2|\partial_z\Phi|^{-3}\mathrm{Cin}(\Phi)$.

\begin{key example} For the defining family $\Phi$ as introduced in \eqref{defining function} one has
\begin{align}
\label{projection formula}
\mathrm{Proj}(\Phi)(u,r,t;v,\rho) &= -8b^4rt\rho(r^2 - t^2), \\
\label{cinematic formula}
 \mathrm{Cin}(\Phi)(u,r,t;v,\rho) &=  64b^8r^3t^3\rho^3(r^2 - t^2).
 \end{align}
Thus, $[\Phi;\fa]$ satisfies the cinematic curvature condition whenever $\mathrm{supp}\, \fa$ avoids the hyperplanes $r =0 $, $t = 0$ and $r = t$.\footnote{In this case, one may further deduce that $\Gamma_{u,r,t}$ is the cone defined implicitly by the equation
\begin{equation*}
    \zeta_1^2 - \frac{\zeta_2^2}{b^2(t^2 - r^2)} + \frac{\zeta_3^2}{b^2(t^2 - r^2)} = 0. 
\end{equation*}} For reference, Appendix~\ref{derivative appendix} contains the formul\ae\, for the various derivatives featured in these computations.  
\end{key example}





\section{The initial decomposition}\label{decomposition section}

For $\Phi$ as defined in \eqref{defining function} both the rotational and cinematic curvature conditions fail. In this section, the operator $A_t$ is decomposed in order to isolate the singularities corresponding to the failure of these curvature conditions.

\subsection{ Spatial decomposition}\label{spatial decomposition subsection}

The operator $A_t$ is first decomposed dyadically with respect to the $r$ variable. To this end, fix a nonnegative $\eta \in C^\infty_c(\R)$ such that
\begin{subequations}
\begin{equation}\label{eta function}
\eta(r) = 1 \quad \textrm{if $r \in [-1,1]$} \quad \textrm{and} \quad \supp \eta \subseteq [-2,2]   
\end{equation}
and define $\beta\in C^\infty_c(\R)$ and $\eta^m, \beta^m \in C^\infty_c(\R)$ by
\begin{equation}\label{beta-definition}\beta(r):=\bbone_{(0,\infty)}(r) (\eta(r) - \eta(2r))\end{equation}
\end{subequations}
and, for each $m \in \Z$,
\begin{equation}\label{beta function}
     \eta^m(r):=\eta (2^{-m} r) \quad \textrm{and} \quad \beta^m(r):=\beta (2^{-m} r) .
\end{equation} 
One may then decompose 
\begin{equation*}
   A_tf(u,r) = \sum_{m \in \Z} \beta^m(r) A_tf(u,r) \qquad \textrm{for $(u,r) \in \R \times (0,\infty)$.}
\end{equation*}

The $r$-localisation induces various spatial orthogonality relations via \eqref{orthogonality relation}. In particular, if $r \in \supp \beta^m$, then $r \sim 2^m$ and it follows from \eqref{orthogonality relation} that 
    \begin{equation}\label{r orthogonality relation}
    |u - v| \lesssim 2^m,\quad |r - \rho| \lesssim 1 \quad \textrm{and} \quad |t-\rho| \lesssim 2^m \quad \textrm{for $(v,\rho) \in \Sigma_{u,r,t}$.}
\end{equation}
To exploit this, given $m, \sigma \in \Z$ define 
\begin{equation*}
    \eta^{m,\sigma}(u,v) := \eta (2^{-m} u - \sigma) \eta (C^{-1} (2^{-m} v - \sigma)),
\end{equation*} 
where $C\geq 1$ is an absolute constant which is chosen to be sufficiently large for the purposes of the forthcoming arguments. We define
\begin{align}
\nonumber
\fa^{0}(u,r,t; v,\rho) &:= \beta(r) r^{1/p}\rho^{1-1/p}, \\
        \fa^{m,\vec{\sigma}}(u,r,t; v,\rho)  &:= \beta^m(r)\eta^{m,\sigma_1}(u,v) \eta^{0,\sigma_2}(r,\rho), \quad \: \textrm{if} \:\: m > 0, \label{m>0 def}\\
        \nonumber
        \fa^{m,\vec{\sigma}}(u,r,t; v,\rho)  &:= \beta^m(r)\eta^{m,\sigma_1}(u,v) \eta^{m,\sigma_2}(t,\rho), \quad \textrm{if} \:\: m < 0,
\end{align}
so that for $m>0$,  $\fa^{m,\vec \sigma} $ is supported where 
\begin{equation*}
    r\sim 2^m, \quad  |u-2^m\sigma_1|\lc 2^m, \quad |v-2^m\sigma_1|\lc 2^m,\quad  |r-\sigma_2|\lc 1, \quad |\rho-\sigma_2|\lc 1.
\end{equation*}
Moreover for $m<0$,   
$\fa^{m,\vec \sigma} $ is supported where 
\begin{equation*}
    r\sim 2^m,\quad  |u-2^m\sigma_1|\lc 2^m,\quad  |v-2^m\sigma_1|\lc 2^m,\quad |t-\sigma_2|\lc 2^m,\quad  |\rho-\sigma_2|\lc 2^m.
\end{equation*} In view of \eqref{r orthogonality relation}, one may bound 
(using the notation in \eqref{general average})
\begin{equation}\label{190715e5.1}
    A_tf \lesssim A[\Phi_t; \fa_t^0]f + 
    \sum_{\vec{\sigma} \in \Z^2}\sum_{m > 0}  2^m
        A[\Phi_t; \fa_t^{m,\vec{\sigma}}]f
        +\sum_{\vec{\sigma} \in \Z^2}\sum_{m < 0}  2^{m/p } A[\Phi_t; \fa_t^{m,\vec{\sigma}}]f,
    \end{equation}
whenever $f$ is a (say) continuous, non-negative function. 

The unit scale piece $\fa_t^0$ is supported where $r\sim 1$ and it is now further dyadically decomposed with respect to both the $\rho$ variable and $|r-t|$. The rationale behind this decomposition is to quantify the value of $\mathrm{Rot}(\Phi_t)$: in view of \eqref{fixed t rotational curvature}, the function $\mathrm{Rot}(\Phi_t)$ can vanish on $\supp \fa_t^0$. If $r \sim 1$ and $\rho \sim 2^{-k}$ , then it follows from \eqref{orthogonality relation} that $|u - v| \lesssim 2^{-k}$ for $(v,\rho) \in \Sigma_{u,r,t}$. Thus, given a function $k \mapsto \ell(k)$ on $\Z$ to be defined momentarily we set
\begin{align*}
    \fa^{k,\ell,\vec{\sigma}}(u,r,t; v,\rho) &:= \beta(r)\beta^{-k}(\rho)\eta^{-k,\sigma_1}(u,v)\beta^{-\ell}(|r-t|) \eta^{-\ell,\sigma_2}(r,t), \qquad  \ell < \ell(k), \\
    \fc^{k,\vec{\sigma}}(u,r,t;v,\rho) &:= \beta(r)\beta^{-k}(\rho)\eta^{-k,\sigma_1}(u,v)\eta^{-\ell(k)}(|r-t|) \eta^{-\ell(k),\sigma_2}(r,t),
 \end{align*}
 so that  on the support of $\fa^{k,\ell,\vec\sigma} $, $\fc^{k,\vec\sigma}$ we have 
 \begin{equation*}
     r\sim 1, \quad  \rho\sim 2^{-k},\quad 
 |u-2^{-k}\sigma_1|\lc 2^{-k} \quad  \text{ and } \quad |v-2^{-k}\sigma_1|\lc 2^{-k};
 \end{equation*}  
 moreover 
 \begin{equation*}
 |r-t|\sim 2^{-\ell}, \quad  |r-2^{-\ell}\sigma_2|\lc 2^{-\ell}, \quad |t-2^{-\ell} \sigma_2|\lc 2^{-\ell} 
 \end{equation*}
 on $\supp \fa^{k,\ell,\vec\sigma}$,
 and
 \begin{equation*}
 |r-t|\lc 2^{-\ell(k)}, \quad  |r-2^{-\ell(k)}\sigma_2|\lc 2^{-\ell(k)}, \quad  |t-2^{-\ell(k)} \sigma_2|\lc 2^{-\ell(k)} 
 \end{equation*}
 on $\supp \fc^{k,\vec\sigma}$.
One may bound
 \begin{align}
\label{190715e5.5}   
A[\Phi_t; \fa_t^0]f &\lesssim \sum_{\vec{\sigma} \in \Z^2}\sum_{\substack{(k,\ell) \in \Z^2 \\ \ell < \ell(k)}}  2^{-k(1-1/p)} A[\Phi_t; \fa_t^{k,\ell,\vec{\sigma}}]f \\
\label{190715e5.6}  
& \qquad  + \sum_{\vec{\sigma}\in \Z^2}\sum_{k \in \Z} 2^{-k(1-1/p)} A[\Phi_t; \fc_t^{k, \vec{\sigma}}]f.
\end{align}
For the purposes of our  proof, we let
$$
\ell(k) := 2k + C_{\mathrm{rot}}
$$
 for some (absolute) constant $C_{\mathrm{rot}} \geq 1$, suitably chosen so as to satisfy the forthcoming requirements. Furthermore, by the first inequality in \eqref{orthogonality relation}, one may in fact restrict the range of the $k$ summation in the above expression to $k \geq - 4$ and of the $(k,\ell)$ summation to the parameter set
\begin{equation*}
    \fP := \big\{(k,\ell) \in \Z \times \Z : k \geq -4 \textrm{ and } k-3 \leq \ell < \ell(k) \big\}. 
\end{equation*}

We show presently that the following bounds imply Theorem~\ref{thm local maximal function}. 
 \begin{theorem}\label{initial decomposition lemma}
 For all $2 < p < \infty$ there exists some $\varepsilon_p > 0$ such that
\begin{flushleft}
\begin{tabular}{rll}
\emph{i)}  & $\displaystyle\big\|\sup_{1 \leq t \leq 2} |A[\Phi_t;  \fa_t^{k,\ell,\vec{\sigma}}]f|\big\|_{p} \lesssim 
    2^{-\ell/p-k\varepsilon_p} 2^{k(1-1/p)} \|f\|_p$ & for $(k, \ell) \in \fP$ \\
 \emph{ii)}  & $\displaystyle\big\|\sup_{1 \leq t \leq 2} |A[\Phi_t;  \fc_t^{k,\vec{\sigma}}]f|\big\|_{p} \lesssim 
    2^{-\ell(k)/p-k\varepsilon_p} 2^{(1-1/p)k}\|f\|_p$ & for all $k \geq -4$ \\
\emph{iii)}   & $\displaystyle\big\|\sup_{1 \leq t \leq 2} |A[\Phi_t; \fa_t^{m, \vec{\sigma}}]f|\big\|_{p} \lesssim 2^{-m}\|f
\|_p$ & for $m > 0$ \\
\emph{iv)}   & $\displaystyle\big\|\sup_{1 \leq t \leq 2} |A[\Phi_t; \fa_t^{m, \vec{\sigma}}]f|\big\|_{p} \lesssim 2^{m \varepsilon_p} 2^{-m/p}\|f\|_p$ & for $m < 0$
  
\end{tabular}
\end{flushleft}
uniformly in $\vec{\sigma} \in \Z^2$.
The above a priori estimates hold for all $f\in C^\infty_0(\bbR^2)$ with support in $\{y \in \R^2:y_2\neq 0\}$.
\end{theorem}

\begin{proof}[Proof of Theorem~\ref{thm local maximal function} assuming Theorem 
\ref{initial decomposition lemma} holds] 
Consider the second and third terms on the right-hand side of \eqref{190715e5.1}. 

When $m>0$ there is spatial orthogonality among the pieces of the decomposition in both $\vec{\sigma}$ and $m$. This observation combined with Theorem~\ref{initial decomposition lemma} iii) above yields
\begin{align*}
\Big\|\sum_{\vec{\sigma} \in \Z^2}\sum_{m > 0} & 2^{m} \sup_{1\le t\le 2} |A[\Phi_t; \fa_t^{m,\vec{\sigma}}]f|\Big\|_p \\
& \lesssim \Big(\sum_{\vec{\sigma} \in \Z^2}\sum_{m > 0} 2^{mp} \big\|\sup_{1 \leq t \leq 2} |A[\Phi_t; \fa_t^{m,\vec{\sigma}}]f|\big\|_p^p\Big)^{1/p}
\lesim \|f\|_p,
\end{align*}
as desired. 

When $m < 0$, note that by the support properties of $\fa_t^{m,\vec{\sigma}}$,
\begin{align*}
\sup_{1 \leq t \leq 2} \sum_{\sigma_2 \in \Z} |A[\Phi_t; \fa_t^{m,\vec{\sigma}}]f| & \lesssim  \sup_{\sigma_2 \in \Z } \sup_{\substack{1 \leq t \leq 2 \\ |t-2^m \sigma_2| \lesssim 2^m}} |A[\Phi_t; \fa_t^{m,\vec{\sigma}}]f| \\
 & \leq \Big( \sum_{\sigma_2 \in \Z } \sup_{\substack{1 \leq t \leq 2 \\ |t-2^m \sigma_2| \lesssim 2^m}} |A[\Phi_t; \fa_t^{m,\vec{\sigma}}]f|^p \Big)^{1/p}.
\end{align*}
Furthermore, applying spatial orthogonality in the $\vec{\sigma}$ parameter, the triangle inequality to the sum in $m$ and Theorem~\ref{initial decomposition lemma} iv), one deduces that
\begin{align*}
\Big\|\sum_{m < 0} &  2^{m/p} \sup_{1\le t\le 2} \sum_{\vec{\sigma} \in \Z^2} |A[\Phi_t; \fa_t^{m,\vec{\sigma}}]f|\Big\|_p \\
& \lesssim \sum_{m < 0}2^{m/p} \Big(\sum_{\vec{\sigma} \in \Z^2} \big\|\sup_{1 \leq t \leq 2} |A[\Phi_t; \fa_t^{m,\vec{\sigma}}]f|\big\|_p^p\Big)^{1/p} \lesssim_p \|f\|_{p},
\end{align*}
where the last step uses the exponential decay $2^{m \varepsilon_p}$ to sum in $m$.

Next, consider the sums in \eqref{190715e5.5}. Again, there is spatial orthogonality in the $\sigma_1$ parameter. This fact and Theorem~\ref{initial decomposition lemma} i) yield
\begin{equation*}
  \Big\| \sum_{\sigma_1 \in \Z} 2^{-k(1-1/p)} \sup_{1\le t\le 2} |A[\Phi_t; \fa_t^{k,\ell,\vec{\sigma}}]f| \Big\|_p \lesim 2^{-\ell / p} 2^{-k \varepsilon_p}\|f\|_p  
\end{equation*}
uniformly in $\sigma_2$. As the parameter $\sigma_2$ corresponds to a decomposition of the $r$ spatial variable,
\begin{align*}
 \Big\|\sum_{\vec{\sigma} \in \Z^2} & 2^{-k(1-1/p)} \sup_{1\le t\le 2} |A[\Phi_t; \fa_t^{k,\ell,\vec{\sigma}}]f|\Big\|_p \\
& \lesim \Big( \sum_{\substack{\sigma_2 \in \Z \\ |\sigma_2| \lesssim 2^\ell}}  \,  \Big\| \sum_{\sigma_1 \in \Z} 2^{-k(1-1/p)} \sup_{1\le t\le 2} |A[\Phi_t; \fa_t^{k,\ell,\vec{\sigma}}]f | \Big\|_p^p \, \Big)^{1/p}\\
& \lesim 2^{-k \varepsilon_p} \big(\sum_{\substack{\sigma_2 \in \Z \\ |\sigma_2| \lesssim 2^\ell}} 2^{-\ell} \|f\|_p^p\big)^{1/p} \lesim 2^{-k \varepsilon_p} \|f\|_p.
\end{align*}
The desired result then follows from the triangle inequality in $(k, \ell)$, using the exponential decay $2^{-k \varepsilon_p}$ to sum over $k$ and $\ell\le \ell(k)$. The sum in \eqref{190715e5.6} is bounded in a similar manner.
\end{proof}




\subsection{ Rescaling} Each piece of the decomposition is appropriately rescaled in order to obtain, wherever possible, favourable bounds on the various curvatures. For the reader's convenience, Appendix~\ref{rescaling appendix} describes the behaviour of the functions $\Phi$, $\mathrm{Rot} (\Phi)$, $\mathrm{Cin} (\Phi)$, etc under general rescalings. 
These rescalings lead to phase functions satisfying certain nonisotropic conditions
which will require  extensions of some classical results on oscillatory integral operators (see \S\ref{two-parameter-sect} below).

\medskip




\noindent\subsubsection{The case ${m = 0}$}  For $(k, 
\ell) \in \fP$ we define the dilations
\[ D^{k,\ell}(u,r,t;v,\rho) := (2^{-k} u, 2^{-\ell}r, 2^{-\ell} t; 2^{-k} v, 2^{-k}\rho).\]
Let \[e(k,\ell) := \ell - 2k + \ell \wedge 2k =
\begin{cases} 
\ell, &\text{ if } \ell\ge 2k,
\\
2\ell-2k, &\text{ if } \ell\le 2k,
\end{cases}
\]
 and  define
\Be
\label{k-rescaling}
\begin{gathered}
    \Phi^{k,\ell} :=  2^{2k + e(k,\ell)/3}\Phi \circ D^{k,\ell}, \quad \tilde{\fa}^{k, \ell, \vec{\sigma}} := \fa^{k,\ell, \vec{\sigma}} \circ D^{k,\ell},\\ \Phi^k := \Phi^{k, \ell(k)}, \quad \tilde{\fc}^{k, \vec{\sigma}} := \fc^{k,\vec{\sigma}} \circ D^{k,\ell(k)}.  
\end{gathered}
\Ee
Note that $\tilde \fa^{k,\ell,\vec\sigma} $ is supported where 
\begin{equation*}
    \rho \sim 1,\quad r\sim 2^\ell,\quad |r-t|\sim 1,\quad |u-\sigma_1|\lc 1,\quad |v-\sigma_1|\lc 1,\quad |r-\sigma_2|\lc 1 \quad |t-\sigma_2|\lc 1.
\end{equation*}
The support of $\tilde \fc^{k,\vec\sigma}$  has similar properties, with $\ell(k)$ in place of $\ell$ and $|r-t|\lesssim 1$.

The appearance of the factor $2^{2k + e(k,\ell)/3}$ is motivated by the fact that
\begin{align}\label{RotCurv1 k ell}
    \mathrm{Rot} (\Phi_t^{k, \ell}) &\sim 1  \quad \textrm{on}\:\: \mathrm{supp}\,\tilde{\fa}^{k, \ell, \vec{\sigma}} \quad \textrm{if} \:\: |\ell-2k| \geq C_{\mathrm{rot}}, \\
  \label{RotCurv ck 1}  
 \mathrm{Rot} (\Phi_t^{k}) &\sim 1  \quad \textrm{on}\:\: \mathrm{supp}\,\tilde{\fc}^{k,  \vec{\sigma}}, \\
    \label{RotCurv ck 2}
     \mathrm{Rot} ((\Phi^{k})^{\star}_r) &\sim 1  \quad \textrm{on}\:\: \mathrm{supp}\,\tilde{\fc}^{k, \vec{\sigma}}
\end{align}
where $(\Phi^{k})^{\star}_r(u,t;v,\rho) := \Phi^{k}_t(u,r;v,\rho)$. Note, however, that $\mathrm{Rot}(\Phi_t^{k,\ell})$ may vanish on $\supp \tilde{\fa}^{k,\ell,\vec{\sigma}}$ if  $|\ell-2k| \leq C_{\mathrm{rot}}$.

Setting $f_k(v,\rho)= f(2^{-k} v, 2^{-k}\rho)$, and using that $\delta$ is homogeneous of degree $-1$ one has
\begin{align*}
&A[\Phi_{2^{-\ell} t} ; \fa^{k,\ell, \vec \sigma}_{2^{-\ell}t}] f(2^{-k} u, 2^{-\ell } r) = 2^{ e(k,\ell)/3} A [\Phi^{k,\ell}_t, \tilde \fa_t^{k,\ell,\vec\sigma} ] f_k(u,r) ,
\\
&A[\Phi_{2^{-\ell(k)} t} ; \fc^{k, \vec \sigma}_{2^{-\ell(k)}t}] f(2^{-k} u, 2^{-\ell } r) = 2^{ \ell(k)/3} A [\Phi^{k,\ell}_t, \tilde \fc_t^{k,\vec\sigma} ] f_k(u,r) .
\end{align*}
Thus, by rescaling 
to prove  Theorem~\ref{initial decomposition lemma} i) and ii) it suffices to show that 
\begin{align}
\label{m 0 bound}
\big\|\sup_{t} |A[\Phi_t^{k,\ell}; \tilde{\fa}_t^{k,\ell, \vec{\sigma}}]|\big\|_{L^p \to L^p} &\lesssim 2^{-e(k,\ell)/3 +(1-2/p)k -k\varepsilon_p},\\
\label{m 0 bound 2}
\big\|\sup_{t} |A[\Phi_{t}^{k}; \tilde{\fc}_{t}^{k , \vec{\sigma}}]|\big\|_{L^p \to L^p} &\lesssim 2^{-\ell(k)/3 +(1-2/p)k -k\varepsilon_p};
\end{align}
where (by a slight abuse of notation) we indicate  the operator norms of the maximal operators on the left-hand side.  We note that in view of the support properties of $\tilde \fa^{k,\ell,\vec\sigma}$, $\tilde \fc^{k,\sigma}$ the global supremum in the definition of the maximal operator reduces to a supremum over an interval $I$ of length $|I|\sim 1$ centered at $\sigma_2$.

It is helpful to isolate the key features of the rescaled averaging operators used to prove the above inequality. As a first step in this direction, note that each $[\Phi_t^{k,\ell}; \tilde{\fa}_t^{k,\ell, \vec{\sigma}}]$  belongs to the class in the following definition. We use  coordinates $(x;z)$ for the rescaled phase functions 
where $(x_1,x_2)$ corresponds to a scaled version of $(u,r)$ and $(z_1,z_2)$ to a scaled version of $(v,\rho)$.

We define collections $\fA^{k,\ell}$ of defining pairs $[\Phi;\fa]$ involving inequalities and support assumptions that are uniform in $k,\ell$. 

\begin{definition}\label{190723def5.2} Let $\mathfrak{A}^{k,\ell}$ denote the set of all smooth families of defining pairs $[\Phi;\fa]$ for which the following conditions hold: 
\begin{flushleft}
\begin{tabular}{rl}
 a)$_{k,\ell}$    &    $|\partial_x^\alpha\partial_z^\beta\partial_t^\gamma \fa(x,t,z)|\lc 1$ and $\mathrm{diam} \,\mathrm{supp}\,\fa \lesssim 1$ \\
 $\Phi_1$)$_{k,\ell}$    &  $\displaystyle |\partial_x^{\alpha}\partial_z^{\beta} \partial_t^{\gamma} \Phi_t(x;z)| \lesssim \left\{\begin{array}{ll}
    2^{-2e(k,\ell)/3} &  \!\!\!\textrm{if $\alpha_2$ or $\gamma \neq 0$} \\
    2^{e(k,\ell)/3} & \!\!\!\textrm{otherwise}
\end{array} \right.$\!\!\!,  $|\partial_z \Phi_t(x;z)| \sim 2^{e(k,\ell)/3}$,\\[10pt]
 $\Phi_2$)$_{k,\ell}$   & $\partial_t\Phi_t(x;z)= 2^{-2e(k,\ell)/3 } c (x,t;z) \mathrm{Rot}(\Phi_t) (x;z)$ for some $c \in C^\infty$  depending \\&  on $[\Phi; \fa]$ and with uniform $C^\infty$ bounds on $\supp \fa$. \\
\end{tabular}
\end{flushleft}
These estimates are understood to hold on $\text{supp}\, \fa$,
with the constants only depending on the multiindices
$\alpha,\beta,\gamma \in \N_0^2$ . That is, if we fix a large $N$ then we get uniform estimates for   $|\alpha|$, $|\beta|$, $|\gamma| \leq N$.
\end{definition}


For $[\Phi^{k,\ell}; \tilde{\fa}^{k,\ell, \vec{\sigma}}]$ it is easy to see that  a)$_{k,\ell}$ and $\Phi_1$)$_{k,\ell}$ hold via a direct computation (the lower bound in $\Phi_1$)$_{k,\ell}$ is a little trickier and uses \eqref{identities Phi=0}). The remaining condition $\Phi_2$)$_{k,\ell}$  follows from an appropriately rescaled variant of the key identity \eqref{key identity}.  Indeed, note that
 \begin{equation*}
     \partial_t \Phi^{k,\ell}_t(u,r;v,\rho) = 2^{-2e(k,\ell)/3} \frac{1}{4b^2 (2^{-\ell}r)(2^{-\ell} t) \rho }    \mathrm{Rot}(\Phi_t^{k,\ell})(u,r;v,\rho) 
 \end{equation*}
 where $r \sim t \sim 2^\ell$ and  $\rho \sim 1$ on $\supp \tilde{a}_t^{k,\ell,\vec{\sigma}}$.
Similarly, each $[\Phi^{k}; \tilde{\fc}^{k, \vec{\sigma}}]$ belongs to $\mathfrak{A}^{k, \ell(k)} =: \mathfrak{C}^{k}$.\medskip




\subsubsection{The cases {{$m \neq 0$}}} For $m \in \Z \setminus \{0\}$, define\footnote{The $\Phi^m$ notation in \eqref{definition phi m}  conflicts with the $\Phi^k$ notation introduced in \eqref{k-rescaling}. Nevertheless, it shall always be clear from the context which definition is intended.} (recalling $m\wedge 0=\min\{m,0\}$)
\begin{gather}
   D^m(u,r, t;v,\rho) := (2^m u, r, 2^{m \wedge 0} t; 2^m v, 2^{m \wedge 0} \rho), \notag \\
    \Phi^{m} := 2^{-2m}\Phi \circ D^m, \quad \tilde{\fa}^{m,\vec{\sigma}} := \fa^{m,\vec{\sigma}} \circ D^m, \label{definition phi m} 
\end{gather}
and let $(\Phi^m)_r^{\star}(u,t;v,\rho) := \Phi^m(u,r,t;v,\rho)$. It follows from \eqref{fixed t rotational curvature} and \eqref{fixed r rotational curvature} that
\begin{equation}\label{rotational m not 0}
   \mathrm{Rot}(\Phi_t^m)  \sim 1 \quad \textrm{if $m > 0$} \quad \textrm{and} \quad \mathrm{Rot}\big((\Phi^m)_r^{\star}\big)  \sim 1 \quad \textrm{if $m < 0$} \quad \textrm{on  $\mathrm{supp}\, \tilde{\fa}^{m,\vec{\sigma}}$;} 
\end{equation}
this observation motivates the choice of normalising factor $2^{-2m}$.

Note that for $m>0$ the new amplitude  $\tilde \fa^{m,\vec\sigma}$ is supported where  
\begin{equation*}
r\sim 2^m,\quad |u-\sigma_1|\lc 1,\quad |v-\sigma_1|\lc 1,\quad |r-\sigma_2|\lc 1,\quad |\rho-\sigma_2|\lc 1,  
\end{equation*}
and if  $m<0$ then $\tilde \fa^{m,\vec\sigma}$ is supported where
\begin{equation*}
r\sim 2^m,\quad |u-\sigma_1|\lc 1,\quad |v-\sigma_1|\lc 1,\quad   |t-\sigma_2|\lc 1,\quad |\rho-\sigma_2|\lc 1.
\end{equation*}

Setting $f^m(v,\rho)=f(2^m v, 2^{m\wedge 0}\rho)$ a computation shows
\[
A[\Phi_{2^{m\wedge 0}t};\fa_{2^{m\wedge 0}t}^{m,\vec\sigma}]f(2^m u,r)=2^{-m}2^{m \wedge 0} 
A[\Phi^m_{t} , \tilde\fa_t ^{m,\vec \sigma} ]  f^m (u,r).
\]
Thus by  rescaling,
to prove  Theorem~\ref{initial decomposition lemma} iii) and iv) {it suffices to} show that 
\begin{equation}\label{m not 0 bound}
\big\|\sup_{t} |A[\Phi_t^{m}; \tilde{\fa}_t^{m, \vec{\sigma}}]|\big\|_{L^p \to L^p} \lesssim 2^{(m \wedge 0)\varepsilon_p}.
\end{equation}
Note that in view of the support properties of $\tilde{\fa}^{m,\vec{\sigma}}_t$, the global supremum in the definition of the maximal operator reduces to a supremum over an interval $I$ which equals $[1,2]$ if $m>0$ and has length $|I|\sim 1$ and it is centered at $\sigma_2$ if $m<0$; in the case $m>0$ we abuse of notation and assume that $\tilde{\fa}_t^{m,\vec{\sigma}}$ is supported on $t\sim 1$, adding a cut-off function if necessary.

If $m > 0$, then a simple computation shows that $[\Phi^{m}; \tilde{\fa}^{m, \vec{\sigma}}] \in \mathfrak{A}^{0,0}=:\mathfrak{A}^0$. On the other hand, if $m < 0$, then  $[\Phi^{m}; \tilde{\fa}^{m, \vec{\sigma}}]$ belongs to the following class classes $\fA_m$ in the following definition where the implicit constants are uniform in $m$.

\begin{definition} \label{fA-m-def} 
For $m < 0$ let $\mathfrak{A}^{m}$ denote the set of all smooth families of defining pairs $[\Phi;\fa]$ satisfying:
\begin{flushleft}
\begin{tabular}{rl}
 a)$_{m}$    & $ \displaystyle |\partial_{x}^{\alpha} \partial_z^\beta \partial_t^\gamma  \fa (x,t;z)| \lesssim  \begin{cases}
    2^{-m\alpha_2} & \textrm{if $\alpha_2 \neq 0$} \\
    1 & \text{otherwise}
 \end{cases} $, $\quad  \mathrm{diam}\,\mathrm{supp}\,\mathfrak{a} \lesssim 1$, and the \\ & projection of   $\supp \fa$  in the $x_2$-variable lies in
  an interval of length $\lesssim 2^{m}$; \\[5pt]
 $\Phi_1$)$_{m}$    &  $\displaystyle |\partial_x^{\alpha}\partial_z^{\beta} \partial_t^{\gamma} \Phi_t(x;z)| \lesssim \left\{\begin{array}{ll}
    2^{-2m} &  \textrm{if $\alpha_2 \neq 0$} \\
    1 & \textrm{otherwise}
\end{array} \right.$,  $\quad |\partial_z \Phi_t(x;z)| \sim 1$ 
\end{tabular}
\end{flushleft}
on $\mathrm{supp}\,\fa$ for all $\alpha,\beta,\gamma \in \N_0^2$ with $|\alpha|$, $|\beta|$, $|\gamma| \leq N$.
\end{definition}

The derivative bounds on the amplitude for $\alpha_2=0$, which are uniformly bounded, are used for the $L^2$-estimates in Section \ref{L2 section}. The bounds for $\alpha_2 \neq 0$ are used for the $L^p$-estimates in Section \ref{Lp section}, although they do not introduce any loss for the purposes of the desired inequality \eqref{Lp inequality}.





\subsection{ Cinematic curvature decomposition}

The decomposition described in \S\ref{spatial decomposition subsection} automatically isolates the region where the cinematic curvature vanishes.\medskip

\subsubsection{The case $m = 0$}  By \eqref{plane curve curvature bound}, \eqref{projection formula} and \eqref{cinematic formula}, each $[\Phi^{k,\ell}; \tilde{\fa}^{k,\ell, \vec{\sigma}}]$ belongs to the following class.

\begin{definition} Let $\mathfrak{A}_{\mathrm{Cin}}^{k,\ell}$ denote the set of all $[\Phi; \fa] \in \mathfrak{A}^{k,\ell}$ satisfying:
\begin{flushleft}
\begin{tabular}{rl}
  \textbf{C})$_{k}$   & $|\kappa(\Phi)(\vec{x};z)|,\:\: |\mathrm{Proj}(\Phi)(\vec{x};z)|,\:\: |\mathrm{Cin}(\Phi)(\vec{x};z)|  \gtrsim 2^{-M k}$ \quad \textrm{for $(\vec{x};z) \in \mathrm{supp}\,\fa$.}
\end{tabular}
  \end{flushleft}
  Here $M \geq 1$ is an appropriate chosen absolute constant.
\end{definition}

Observe, however, that the $[\Phi^{k}; \tilde{\fc}^{k, \vec{\sigma}}]$ lie in $\mathfrak{A}^{k,\ell(k)}$ but do \textbf{not} belong to $\mathfrak{A}_{\mathrm{Cin}}^{k, \ell(k)}$; it is for this reason that this part of the operator is isolated in the analysis. Indeed, the amplitude $\tilde{\fc}^{k,\vec{\sigma}}$ is supported on the region $|r-t|\lesssim 2^{-\ell(k)}$ and therefore $\kappa(\Phi)$, $\mathrm{Proj}(\Phi)$ and $\mathrm{Cin}(\Phi)$ can vanish on $\supp\tilde{\fc}^{k,\vec{\sigma}}$. Nevertheless, these quantities only vanish on a small set and, in particular, $[\Phi^{k}; \tilde{\fc}^{k, \vec{\sigma}}]$ belongs to the following class. 

\begin{definition} Let $\mathfrak{C}_{\mathrm{Cin}}^{k}$ denote the set of all $[\Phi; \fc] \in \mathfrak{A}^{k,\ell(k)}$ such that, for all $\delta > 0$, if $(x,t;z) \in \mathrm{supp}\,\fc$ with $|t - x_2| > \delta$, then
 \begin{flushleft}
\begin{tabular}{rl}
  \textbf{C}$_\delta$)$_{k}$   & $|\kappa(\Phi)(x,t;z)|,\:\: |\mathrm{Proj}(\Phi)(x,t;z)|,\:\: |\mathrm{Cin}(\Phi)(x,t;z)|  \gtrsim \delta 2^{-M k}$. 
\end{tabular}
  \end{flushleft} 
As before, $M \geq 1$ is an appropriately chosen absolute constant.
\end{definition}

\subsubsection{The cases $m \neq 0$} If $m > 0$,  then \eqref{plane curve curvature bound}, \eqref{projection formula} and \eqref{cinematic formula} show that $[\Phi^{m}; \tilde{\fa}^{m, \vec{\sigma}}] $ belongs to   $\mathfrak{A}_{\mathrm{Cin}}^{0,0} =: \mathfrak{A}_{\mathrm{Cin}}^{0}$. On the other hand, if $m < 0$, then $[\Phi^{m}; \tilde{\fa}^{m, \vec{\sigma}}]$ belongs to the following class.

\begin{definition} For $m <0$ let $\mathfrak{A}_{\mathrm{Cin}}^{m}$ denote the set of $[\Phi; \fc] \in \mathfrak{A}^{m}$ satisfying \textbf{C})$_{-m}$.
\end{definition}




\subsection{ Rotational curvature decomposition} Further decomposition is required in order to isolate the regions where the rotational curvature vanishes.\medskip

\subsubsection{The case $m = 0$} Let $\varepsilon_{\circ} > 0$ be a fixed constant, chosen small enough to satisfy the requirements of the forthcoming proof, and define
\begin{equation*}
    \fb^{k,\ell, \vec{\sigma}}(u,r,t;v,\rho) := \tilde{\fa}^{k,\ell, \vec{\sigma}}(u,r,t;v,\rho)\eta\big(\varepsilon_{\circ}^{-1}\mathrm{Rot}(\Phi_t^{k,\ell})(u,r;v,\rho) \big)
    \big).
\end{equation*}
In view of \eqref{RotCurv1 k ell}, one may readily verify that $\fb^{k,\ell,\vec{\sigma}}$ is identically zero unless $|\ell - 2k| \lesssim 1$, in which case $[\Phi^{k,\ell};\fb^{k,\ell,\vec{\sigma}}] \in \mathfrak{A}_{\mathrm{Cin}}^{k,2k} =: \mathfrak{B}_{\mathrm{Cin}}^k$.\medskip

\subsubsection*{\underline{Vanishing rotational curvature}}
 To analyse the operators $A[\Phi_t^{k,\ell}; \fb_t^{k,\ell, \vec{\sigma}}]$ it is necessary to exploit the fold conditions discussed in \S\ref{rotational curvature section}. The observations of \S\ref{rotational curvature section} imply that $[\Phi^{k,\ell}; \fb^{k,\ell,\vec{\sigma}}]$ belongs to the following class.

\begin{definition}\label{190723def5.7}
 Let $\mathfrak{B}^{k}_{\mathrm{Rot}}$ denote the set of all smooth families of defining pairs $[\Phi; \fb] \in \mathfrak{A}^{k,2k}$ that, in addition to a)$_{k,2k}$, $\Phi_1$)$_{k,2k}$, $\Phi_2$)$_{k,2k}$, satisfy:\medskip

\noindent The \textit{support condition}:

\begin{itemize}
\item[b)$_{k}$] $\mathrm{supp}\,\fb_t$ is contained in an $O(\varepsilon_{\circ})$-neighbourhood of $\mathrm{supp}\,\fb_t \cap  \mathcal{Z}_t$ where $ \mathcal{Z}_t$ denotes the \textit{fold surface}
\begin{equation}\label{fold surface}
    \mathcal{Z}_t := \big\{(x;z) \in \R^2  \times \R^2 : \Phi_t(x;z) = \mathrm{Rot}(\Phi_t)(x;z) = 0 \big\}.
\end{equation} 
\end{itemize}

\noindent The \textit{fold conditions}: For every $(x_0;z_0) \in \mathrm{supp}\,\fb_{t_0} \cap  \mathcal{Z}_{t_0}$ there exist:

\begin{itemize}
\item[F$_1$)$_{k}$] Vectors $U=(u_1, u_2, u_3), V=(v_1, v_2, v_3) \in \R^3$ satisfying
\begin{equation*}
\begin{split}
   &  \bigg|\bigg\langle \partial_{zz}^2 \Big\langle U,  \begin{bmatrix}
    \Phi_{t_0} \\
    \partial_x \Phi_{t_0}
    \end{bmatrix}
    \Big\rangle\Big|_{(x_0; z_0)} V''\,,\, V'' \bigg \rangle\bigg|\sim 2^{-4k/3}, \\
   &  \bigg|\bigg\langle \partial_{xx}^2 \Big\langle V,  \begin{bmatrix}
    \Phi_{t_0} \\
    \partial_z \Phi_{t_0}
    \end{bmatrix}
    \Big\rangle\Big|_{(x_0; z_0)} U''\,,\, U'' \bigg \rangle\bigg| \sim 2^{-4k/3},
    \end{split}
\end{equation*}
where $U''=(u_2, u_3)$ and $V''=(v_2, v_3)$.

\item[F$_2$)$_{k}$]  $3 \times 3$ real matrices $\mathbf{X}$ and $\mathbf{Z}$ such that:
\begin{enumerate}[i)]
\item If $\mathbf{X}^{ij}$ and $\mathbf{Z}^{ij}$ denote the $(i,j)$ entry of $\mathbf{X}$ and $\mathbf{Z}$, respectively, then 
\begin{equation*}
   |\mathbf{X}^{ij}| \lesssim \left\{\begin{array}{ll}
        2^{-2k} & \textrm{if $(i,j) \in \{(1,3), (2,3)\}$}  \\
        1 & \textrm{otherwise}
    \end{array} \right., \quad  |\mathbf{Z}^{ij}| \lesssim 1 .
\end{equation*}
\item $\mathbf{X}e_3 = U$, $\mathbf{Z}e_3 = V$ and $|\det \mathbf{X}| \sim |\det \mathbf{Z}| \sim 1$.
\item The matrices $\mathbf{X}, \mathbf{Z}$ transform $\mathfrak{M}(\Phi_{t_0})(x_0;z_0)$ into the normal form
\begin{equation*}
    \mathbf{X}^{\top} \circ \mathfrak{M}(\Phi_{t_0})(x_0;z_0)\circ \mathbf{Z} = \begin{bmatrix}
\mathbf{M}_{t_0}(x_0;z_0) & 0 \\
0 & 0
\end{bmatrix},
\end{equation*}
where the $2 \times 2$ principal minor satisfies $|\det \mathbf{M}_{t_0}(x_0;z_0)| \sim 2^{4k/3}$.

\end{enumerate}
\end{itemize}
\end{definition}

For $[\Phi^{k,\ell}; \fb^{k,\ell,\vec{\sigma}}]$ the support condition is satisfied owing to the choice of localisation whilst, for the fold conditions, $U$, $V$ and $\mathbf{X}$, $\mathbf{Z}$ can be taken to be suitably rescaled versions of the vectors in \eqref{fold condition 1} and the matrices in \eqref{fold condition 2}, respectively. \medskip

 \subsubsection*{\underline{Nonvanishing rotational curvature}}
 By \eqref{RotCurv1 k ell}, each $[\Phi^{k,\ell}; \tilde{\fa}^{k,\ell, \vec{\sigma}} - \fb^{k,\ell, \vec{\sigma}}]$ belongs to the following class.

\begin{definition}\label{fA-k-ell-Rot}
Let $\mathfrak{A}_{\mathrm{Rot}}^{k,\ell}$ denote the set of all $[\Phi; \fa] \in \mathfrak{A}^{k,\ell}$ that satisfy
\begin{flushleft}
\begin{tabular}{rl}
   \textbf{R})$_{k,\ell}$  &  $\mathrm{Rot}(\Phi_t) \sim 1$ on $\mathrm{supp}\,\fa_t$. 
\end{tabular}
\end{flushleft}
\end{definition}

Recalling \eqref{m 0 bound}, to prove Lemma~\ref{initial decomposition lemma} i) it therefore suffices to show: 
\begin{flushleft}
\begin{tabular}{ll}
    $\displaystyle \big\|\sup_{t \in I } |A[\Phi_t; \fb_t]|\big\|_{L^p \to L^p} \lesssim 2^{-\frac{2k}{3}
+(1-\frac{2}{p})k -k\varepsilon_p}\|\fb\|_{C^N}$
&\textrm{if  $[\Phi;\fb] \in \mathfrak{B}^{k}_{\mathrm{Cin}} \cap \mathfrak{B}^{k}_{\mathrm{Rot}}$}, \\
$\displaystyle \big\|\sup_{t \in I} |A[\Phi_t; \fa_t]|\big\|_{L^p \to L^p} \lesssim 2^{-\frac{e(k,\ell)}{3} +(1-\frac{2}{p})k -k\varepsilon_p}\|\fa\|_{C^N} \qquad$ &\textrm{if $[\Phi;\fa] \in \mathfrak{A}_{\mathrm{Cin}}^{k, \ell} \cap \mathfrak{A}_{\mathrm{Rot}}^{k,\ell}$,} 
\end{tabular}
\end{flushleft}
where $I$ is an interval of length $|I|\sim 1$ containing the $t$-support of $\fa$ or $\fb$.

Similarly, by \eqref{RotCurv ck 1} and \eqref{RotCurv ck 2}, each $[\Phi^{k}; \tilde{\fc}^{k, \vec{\sigma}}]$ belongs to the following class.

\begin{definition} Let $\mathfrak{C}^{k}_{\mathrm{Rot}}$ denote the set of all $[\Phi; \fc] \in \mathfrak{C}^k$ that satisfy $\textbf{R})_{k, \ell(k)}$ and
\begin{flushleft}
\begin{tabular}{rl}
   \textbf{R}$^{\star}$)$_{k}$  & $\mathrm{Rot}(\Phi_{x_2}^{\star}) \sim 1$ on $\mathrm{supp}\,\fc_{x_2}^{\star}$ 
\end{tabular}
\end{flushleft}
where $\Phi_{x_2}^{\star}(x_1,t;z) := \Phi_t(x_1,x_2;z)$ and $\fc_{x_2}^{\star}(x_1,t;z) := \fc_t(x_1,x_2;z)$.
\end{definition}
Thus, recalling \eqref{m 0 bound 2}, to prove Lemma~\ref{initial decomposition lemma} ii) it suffices to show: 
\begin{align*}
   \displaystyle\big\|\sup_{t \in I } |A[\Phi_t; \fc_t]|\big\|_{L^p \to L^p} &\lesssim 2^{-\frac{2k}{3}
+(1-\frac{2}{p})k -k\varepsilon_p}\|\fc\|_{C^N}\quad\:\, & \textrm{if $[\Phi;\fc] \in \mathfrak{C}^{k}_{\mathrm{Cin}} \cap \mathfrak{C}^{k}_{\mathrm{Rot}}$,}
\end{align*}
where $I$ is an interval of length $|I|\sim 1$ containing the $t$-support of $\fc$.




 \subsubsection{The cases $m \neq 0$} If $m > 0$, then it follows from \eqref{rotational m not 0} that $[\Phi^{m}; \tilde{\fa}^{m, \vec{\sigma}}] \in \mathfrak{A}_{\mathrm{Rot}}^{0,0} =: \mathfrak{A}_{\mathrm{Rot}}^0$. On the other hand, if $m < 0$, then \eqref{rotational m not 0} implies that $[\Phi^{m}; \tilde{\fa}^{m, \vec{\sigma}}]$ has favourable rotational curvature properties once the r\^oles of the $r$ and $t$ variables are interchanged. In particular, in this case $[\Phi^{m}; \tilde{\fa}^{m, \vec{\sigma}}]$ belongs to the following class.

\begin{definition}\label{fA-m-rot-def}  For $m < 0$ let $\mathfrak{A}^{m}_{\mathrm{Rot}}$ denote the set of all $[\Phi; \fa] \in \mathfrak{A}^{m}$ that satisfy
\begin{flushleft}
\begin{tabular}{rl}
   \textbf{R}$^{\star}$)$_{m}$  & $\mathrm{Rot}(\Phi_{x_2}^{\star}) \sim 1$ on $\mathrm{supp}\,\fa_{x_2}^{\star}$
\end{tabular}
\end{flushleft}
where $\Phi_{x_2}^{\star}(x_1,t;z) := \Phi_t(x_1,x_2;z)$ and $\fa_{x_2}^{\star}(x_1,t;z) := \fa_t(x_1,x_2;z)$.
\end{definition}

Thus, recalling \eqref{m not 0 bound}, to prove Lemma~\ref{initial decomposition lemma} iii) and iv) it suffices to show that 
\begin{flushleft}
\begin{tabular}{ll}
$\displaystyle \big\|\sup_{t \in I} |A[\Phi_t; \fa_t]|\big\|_{L^p \to L^p} \lesssim \|\fa\|_{C^N}$
 &\textrm{if $[\Phi;\fa] \in \mathfrak{A}_{\mathrm{Cin}}^{0} \cap \mathfrak{A}_{\mathrm{Rot}}^{0}$,}\\
 $\displaystyle\big\|\sup_{t \in I} |A[\Phi_t; \fa_t]|\big\|_{L^p \to L^p} \lesssim 2^{m \varepsilon_p}\sup_{x_2}\|\fa\|_{C^N_{x_1,z,t}} \qquad \qquad $  & \textrm{if $[\Phi;\fa] \in \mathfrak{A}_{\mathrm{Cin}}^{m} \cap \mathfrak{A}_{\mathrm{Rot}}^{m}$,  $m < 0$,}  
\end{tabular}
\end{flushleft}
where $I$ is an interval of length $|I|\sim 1$ containing the $t$-support of $\fa$.




\subsection{ Frequency decomposition} Given a smooth family of defining pairs $[\Phi; \fa]$ note that, since the inverse Fourier transform $\check{\eta}$ of the cutoff $\eta$ from \eqref{eta function} has unit mean,
 \begin{equation*}
    A[\Phi_t;\fa_t]f(x) = \lim_{j \to \infty} 2^j \int_{\R^2} \check{\eta}(2^j \Phi_t(x;z))\fa_t(x;z) f(z)\,\ud z,
\end{equation*}
where $\eta$ is a bump function as in \eqref{eta function}. The integral formula for $\check{\eta}$ then yields 
\begin{equation*}
A[\Phi_t;\fa_t] = A_{\leq J}[\Phi_t;\fa_t] + \sum_{j=J}^{\infty} A_j[\Phi_t;\fa_t]
\end{equation*}
for any $J \in \Z$ where
  \begin{align}
  \nonumber
  A_{\leq J}[\Phi_t;\fa_t]f(x) &:= \frac{1}{2\pi} \int_{\R^2} \int_{\R} e^{i \theta \Phi_t(x;z)}  \fa_t(x;z)\eta^J(\theta) \,\ud\theta\, f(z)\,\ud z, \\
  \label{frequence localised definition}
      A_j[\Phi_t;\fa_t]f(x) &:=  \frac{1}{2\pi} \int_{\R^2} \int_{\R} e^{i \theta \Phi_t(x;z)}  \fa_t(x;z)\beta^j(|\theta|) \,\ud\theta\, f(z)\,\ud z .
  \end{align} 
This provides a frequency decomposition of \eqref{general average}. The low frequency part of the operator (corresponding to $A_{\leq J}[\Phi_t;\fa_t]$ for a suitable choice of $J$) can be dealt with via pointwise comparison with the Hardy--Littlewood maximal operator, and so the remainder of the article will focus on the high frequency parts. In view of this and the observations of the preceding subsection, Theorem~\ref{initial decomposition lemma} is a consequence of the following proposition, which will be proved in \S\ref{L2 section} and \S\ref{Lp section} using the theory developed in \S\ref{two-parameter-sect}.

\begin{proposition}\label{frequency decomposition prop} There exists $N\in \bbN$, $\varepsilon_p>0$ 
such that 
for all $k \geq -4$, $(k,\ell) \in \fP$, $j \geq -e(k,\ell)/3$ and $2 < p < \infty$,  the following bounds hold, with the implicit constants depending on $p$. In each inequality, $I$ denotes an interval of length $|I|\sim 1$ containing the $t$-support of the amplitude.

\smallskip

(i) For $[\Phi;\fb] \in \mathfrak{B}^{k}_{\mathrm{Cin}} \cap \mathfrak{B}^{k}_{\mathrm{Rot}}$,
\[
 \big\|\sup_{t \in I } |A_j[\Phi_t; \fb_t]f|\big\|_{p} \lesssim 2^{-(j\vee 0) \varepsilon_p} 2^{-\frac{2k}{3}
+(1-\frac{2}{p})k -k\varepsilon_p}\|\fb\|_{C^N} \|f\|_p.
\]

(ii) For
$[\Phi;\fa] \in \mathfrak{A}_{\mathrm{Cin}}^{k, \ell} \cap \mathfrak{A}_{\mathrm{Rot}}^{k,\ell}$,
\[\big\|\sup_{t \in I } |A_j[\Phi_t; \fa_t]f|\big\|_{p} \lesssim 2^{-(j\vee 0) \varepsilon_p}  2^{-\frac{e(k,\ell)}{3} +(1-\frac{2}{p})k -k\varepsilon_p}\|\fa\|_{C^N} \|f\|_p.
\]

(iii) For
 $[\Phi;\fc] \in \mathfrak{C}^{k}_{\mathrm{Cin}} \cap \mathfrak{C}^{k}_{\mathrm{Rot}}$,
 \[\big\|\sup_{t \in I } |A_j[\Phi_t; \fc_t]f|\big\|_{p} \lesssim 2^{-(j\vee 0) \varepsilon_p}  2^{-\frac{2k}{3}
+k(1-\frac{2}{p}) -k\varepsilon_p}\|\fc\|_{C^N}\|f\|_p.
\]

(iv) For $[\Phi;\fa] \in \mathfrak{A}_{\mathrm{Cin}}^{0} \cap \mathfrak{A}_{\mathrm{Rot}}^{0}$,
\[ 
\big\|\sup_{t \in I } |A_j[\Phi_t; \fa_t]f|\big\|_{p} \lesssim 2^{-j \varepsilon_p}  \|\fa\|_{C^N} \|f\|_p.
\]

(v) For $m<0$ and $[\Phi;\fa] \in \mathfrak{A}_{\mathrm{Cin}}^{m} \cap \mathfrak{A}_{\mathrm{Rot}}^{m}$,
\[ \big\|\sup_{t \in I} |A_j[\Phi_t; \fa_t]f|\big\|_{p } \lesssim 2^{-j \varepsilon_p}  2^{m \varepsilon_p}\sup_{x_2}\|\fa\|_{C^N_{x_1,z,t}}\|f\|_p.
\]
\end{proposition}


\noindent{\it Remark.}
Here cases i), iii), iv) and v) are understood to hold for $\ell = 2k$ so that $j$ ranges over values $j \geq -2k/3$, with $k=0$ in the cases iv) and v). In each case, similar estimates hold for $A_{\leq -e(k,\ell)/3}[\Phi_t; \fa_t]$ (corresponding to the low frequency part), which can be proved by elementary means.



\section{$L^2$ bounds for two parameter oscillatory integral
operators}\label{two-parameter-sect}

The first step towards establishing Proposition~\ref{frequency decomposition prop} is to obtain $L^2$ bounds for the frequency localised pieces with favourable dependence in the parameters $k$ and $\ell$. This will follow from certain estimates for maximal functions associated to two parameter oscillatory integrals, which will be proven in this section.

To this end, let $U \subset \R^d \times \R^d$ be an open set, $\Psi:U \to \R$ be a smooth phase function and $a \in C_0^\infty(U)$. Consider, for $\lambda>1$,  the oscillatory integral operator associated to the phase/amplitude pair $[\Psi;a]$, 
\Be \label{TLa}T^\lambda f(x)\equiv T^\lambda[\Psi;a]f(x):=\int_{\R^d} e^{i\lambda \Psi(x;z)} a(x;z) f(z)\, \ud z.\Ee



We now let $0<\delta_\circ\le 1$ and we shall assume that the  following nonisotropic derivative estimates
\Be\label{upperboundsder}
|\partial_{x} ^\alpha \partial_z^\beta \Psi (x;z) |
+ \delta_\circ^{-1} |\partial_{x} ^\alpha \partial_z^\beta \partial_{x_d} \Psi (x;z) |
\le C_{\alpha, \beta}
\Ee
hold for all $(x;z) \in U$ and all $\alpha,\beta  \in \N_0^d$. We shall then derive estimates in terms of the two parameters $\lambda > 1$ and $\delta_\circ\le 1$.
Our results  could be rewritten as a two parameter oscillatory integral estimates 
with phase $\lambda (\varphi(x';z)+\delta_\circ \psi(x;z))$, where $x=(x',x_d)$, and  uniform upper bound derivative estimates on $\varphi$ and $\psi$.

\subsection{ The nondegenerate case}\label{nondegenerate-sect} We first formulate a variant of the classical $L^2$ result of H\"ormander in \cite{Hormander1973}
under the assumption \eqref{upperboundsder}.
\begin{proposition}\label{2par-Hoerprop}
Let $\lambda \geq 1$, $0 < \delta_\circ \leq 1$, $\Psi$ be as in \eqref{upperboundsder}  and suppose that there is $c>0$ such  that $|\det \partial^2_{zx}\Psi (x_0;z_0)| \ge c\delta_\circ$ for some $(x_0; z_0) \in U$.
 Then there exist $\eps_\circ>0$ and $N>0$, independent of  $\lambda$ and $\delta_\circ$,  such that for all smooth $a$ with $\supp a \subset B_{\eps_\circ}(x_0;z_0)$, 
\[\|T^\lambda \|_{L^2(\R^d)\to L^2(\R^d)} \lc \lambda^{-\frac{d-1}{2}} 
\min\{ (\lambda\delta_\circ)^{-1/2}, 1\}\|a\|_{C^N}.\]
\end{proposition}
\begin{proof} After applying translation operators we may assume $(x_0;z_0)=(0;0)$. 
The kernel of $T^\lambda (T^\lambda)^*$ is given by 
$$K^\lambda(x,y) :=\int_{\R^d} e^{i\lambda (\Psi(x;z)-\Psi(y;z) )} 
a(x;z)\overline {a(y;z) } \, \mathrm dz,
$$
and by the Schur's test, the desired estimate follows from the bounds
\begin{equation}\label{eq:basic hormander 1}
\sup_{x \in \R^d} \int_{\R^d} |K^\lambda(x,y)| \, \ud y, \,\,\, \sup_{y \in \R^d} \int_{\R^d} |K^\lambda(x,y)| \, \ud x \,\,  \lesssim \lambda^{-(d-1)} \min \{ (\lambda \delta_\circ)^{-1}, 1\} \| a \|^2_{C^N}.
\end{equation}
We have 
\[
\nabla _{z} ( \Psi(x;z)-\Psi(y;z) )  = A_{\delta_\circ}(x,y;z) \begin{bmatrix} x'-y' \\ \delta_\circ(x_d-y_d)\end{bmatrix} 
\]
where $x=(x',x_d), y=(y',y_d)$ and $$
A_{\delta_\circ}(x,y;z)=\int_0^1\begin{bmatrix} 
\partial_{z'x'}^2\Psi &\delta_\circ^{-1}\partial_{z'x_d}^2\Psi\\ 
\partial_{z_dx'}^2\Psi &\delta_\circ^{-1}\partial_{z_dx_d}^2\Psi \end{bmatrix} \Big|_{(y+s(x-y);z)} \,\ud s. 
$$
By \eqref{upperboundsder}  
we have  $\|A_{\delta_\circ}\|_{C^N}\lc_N 1$. Also clearly $|\det A_{\delta_\circ}(0,0;0)|\ge c$ and thus 
there is an $\eps_\circ>0$ such that 
for $|(x,y;z)|\le \eps_\circ$ the matrix $A_\delta$ is invertible and we obtain the estimate 
$\|\partial_{x,y,z} ^\alpha A_{\delta_\circ}^{-1}(x,y;z)\| \le C_\alpha $ for  all $\alpha \in \N_0^{3d}$ for the matrix norms of the derivatives of $A_{\delta_\circ}^{-1}$.
Hence for $|x|, |y|, |z| \le \eps_\circ$ 
\[|\nabla _{z} ( \Psi(x;z)-\Psi(y;z) )| \ge c(|x'-y'|+\delta_\circ |x_d-y_d|). \] 
By \eqref{upperboundsder}   we have 
\[
|\partial_z^\alpha ( \Psi(x;z)-\Psi(y;z) )| \le C(|x'-y'|+\delta_\circ |x_d-y_d|)
\] 
for all $\alpha \in \N_0^d$. By repeated integration-by-parts in the form of  Corollary \ref{corollaryapp}, 
with the choices of  $\rho(x,y)=|x'-y'|+\delta_\circ|x_d-y_d|$ and $R_2(x,y)=1$,
one  obtains 
\[|K^\lambda(x,y)| \lc_N \|a\|_{C^N}^2(1+\lambda|x'-y'| +\lambda \delta_\circ|x_d-y_d|)^{-N}.\]   In view of the compact support of $a$,  the desired bounds \eqref{eq:basic hormander 1} follow from integrating in $x \in \supp a$ for fixed $y \in \supp a$, and in $y \in \supp a$ for fixed $x \in \supp a$ respectively.
\end{proof}

\subsection{ A two parameter oscillatory integral estimate under two-sided fold conditions}
We shall also formulate a variant of the $L^2$ estimates for oscillatory integral operators with fold singularities of Pan and Sogge \cite{pan-sogge}, which are based on the previous work on Fourier integral operators by Melrose and Taylor \cite{Melrose-Taylor}, under the assumption \eqref{upperboundsder}. We will instead follow the approach in the works of Phong and Stein \cite{phong-stein91}, Cuccagna \cite{Cuccagna} and Greenleaf and the fourth author \cite{GS1999}. 

 \begin{proposition} \label{ps-thm}
Let $\lambda\ge 1$, $0 < \delta_\circ<1$, $\Psi$ be as in  \eqref{upperboundsder}
and suppose  that for some $(x_0;z_0) \in U$ there is $c>0$ such that 
\begin{subequations}
\begin{align}
&|\det \partial_{z'x'}^2\Psi (x_0;z_0)| \geq c,
\label{x'z'}
\\
&\partial_{zx_d}^2\Psi(x_0;z_0)=0,\qquad \qquad\:\:\: \partial_{z_dx}^2\Psi(x_0;z_0)=0,
\label{lastcolumnrow}
\\
&|\partial_{x_dz_dz_d}^3\Psi(x_0;z_0)|\ge c\delta_\circ, \qquad|\partial^3_{z_dx_dx_d}\Psi(x_0;z_0)|\ge c\delta_\circ.
\label{folds}
\end{align} \end{subequations}
Then there exist ${\eps_\circ}>0$ and $N>0$, independent of  $\lambda$ and $\delta_\circ$, such that for all smooth $a$ with $\supp a \subset B_{\eps_\circ}(x_0;z_0)$, 
\[\|T^\lambda \|_{L^2(\R^d)\to L^2(\R^d)} \lc \lambda^{-\frac{d-1}{2}} 
\min\{ (\lambda\delta_\circ)^{-1/3}, 1\}\|a\|_{C^N}.\]
\end{proposition}	

Following \cite{phong-stein91,Cuccagna,GS1999}, we decompose dyadically our operator according to the size of $\det \partial^2_{xz} \Psi$. It is useful to consider the auxiliary quantity 
\begin{equation}\label{sigmadef}
\sigma\equiv \sigma(\Psi)=\partial_{x_dz_d}^2\Psi - {\partial_{x_dz'}^2\Psi[(\partial_{x'z'}^2\Psi)^\top]}^{-1}\partial_{z_dx'}^2\Psi,
\end{equation}
which measures the size of the mixed Hessian. In fact, note that if $A$ is an invertible $(d-1)\times(d-1)$ matrix, $b,c\in \bbR^{d-1}$ and $d\in \bbR$, one has the identity
\[
\begin{bmatrix} I&0\\ -c^\top A^{-1}&1\end{bmatrix}
\begin{bmatrix} A&b\\ c^{\top}&d\end{bmatrix} =\begin{bmatrix} A&b\\0^\top & d-c^\top A^{-1}d\end{bmatrix} \] and therefore
 \Be  \label{rotcurv vs sigma} \det \partial_{xz}^2\Psi(x;z)=\sigma(x;z) \det \partial_{x'z'}^2\Psi(x;z)\Ee for $(x;z)$ near $(x_0;z_0)$. Hence we get, assuming that $\eps_\circ$ is small enough,
$$ 
 |\sigma(x;z)|\sim |\det \partial_{xz}^2\Psi(x;z)|.
 $$
 The fold conditions \eqref{folds} together with 
  \eqref{lastcolumnrow}  imply that
\begin{align*}
 &|\partial_{x_d}  \sigma(x;z)|=
 |\partial_{x_dx_dz_d}^3\Psi(x;z)|+ O(\eps_\circ\delta_\circ),
 \\
 &|\partial_{z_d}  \sigma(x;z)|=
 |\partial_{x_dz_dz_d}^3\Psi(x;z)|+ O(\eps_\circ\delta_\circ), 
 \end{align*}
and using \eqref{folds} we get
\begin{equation}\label{sigma bounds 3}
 |\partial_{x_d}  \sigma(x;z)|\sim \delta_\circ, \qquad 
  |\partial_{z_d}  \sigma(x;z)|\sim \delta_\circ.
\end{equation}
  Finally, note that the assumption \eqref{upperboundsder} implies
 \Be\label{sigma bounds 4}
| \partial_x^\alpha \partial_z^\beta  \sigma(x;z)|\lc_{\alpha,\beta} \delta_\circ
\Ee
for all $\alpha, \beta \in \N_0^d$.  



Let $\eta_0, \eta_1$ be $C^\infty$ functions on the real line with \[ \supp \eta_0 \subset [-2,2], \qquad \supp \eta_1  \subset [-2,-1/2]\cup[1/2,2].\]
For $\lambda \geq 1$, set 
 \Be\label{Mdef}M := \max \{
 \floor{\log_2 (\lambda^{1/2})}, 0 \}
 \Ee 
 and define
\begin{align} 
T^{\lambda, m} f(x)&:=\int_{\R^d} e^{i\lambda \Psi(x;z)} a(x;z)\eta_1(2^m\delta_\circ^{-1} |\sigma(x;z)|) f(z)\,  \mathrm dz, \quad 0 \leq m<M , \label{Tm def}
\\
T^{\lambda, M} f(x)&:=\int_{\R^d} e^{i\lambda \Psi(x;z)} a(x;z)\eta_0(2^M\delta_\circ^{-1} \sigma(x;z) ) f(z) \, \mathrm dz \label{TM def}.
\end{align}
By \eqref{rotcurv vs sigma} and  \eqref{sigma bounds 4} we have $|\det \partial_{zx}^2 \Psi |\sim 2^{-m}\delta_\circ$ on the support of the amplitude in $T^{\lambda, m}$ if $0 \leq m < M$ and 
$|\det \partial^2_{zx}\Psi|\lc 2^{-M}\delta_\circ\lc  \lambda^{-1/2} \delta_\circ$ on the support of the amplitude in $T^{\lambda, M}$.


\begin{proposition}\label{curvdecprop}
Let $\lambda \geq 1$, $\delta_\circ<1$, $[\Psi; a]$ be as in Proposition \ref{ps-thm} and $M$ as in \eqref{Mdef}.
\begin{enumerate}[(i)]
\item If $\lambda\ge \delta_\circ^{-1} $ then,  for $0 \leq m< M$,
\[
\|T^{\lambda, m} \|_{L^2(\R^d)\to L^2(\R^d)} \lc \lambda^{-\frac{d-1}{2}} \min\{
 (2^m/(\lambda\delta_\circ))^{1/2}, 2^{-m}\} \| a \|_{C^N} .\]
 Moreover,
 \[
\|T^{\lambda, M} \|_{L^2(\R^d)\to L^2(\R^d)} \lc \lambda^{-\frac{d}{2}} \| a \|_{C^N}.
\]
\item  If $1\le \lambda\le \delta_\circ^{-1}$ then,
for $0 \leq m<  M $
\[
\|T^{\lambda,m} \|_{L^2(\R^d)\to L^2(\R^d)} \lc 2^{-m} \lambda^{-\frac{d-1}{2}} \| a \|_{C^N}.
\]
Moreover,
 \[
\|T^{\lambda, M} \|_{L^2(\R^d)\to L^2(\R^d)} \lc \lambda^{-\frac{d}{2}} \| a \|_{C^N}.
\]
\end{enumerate}
\end{proposition}

 



We first note  that the bounds in Proposition \ref{curvdecprop} imply Proposition \ref{ps-thm} by summing in the  $m$-parameter.
\begin{proof}[Proof of Proposition \ref{ps-thm}, assuming Proposition \ref{curvdecprop}]
Let  $\eta$, $\beta$ be defined as in \eqref{eta function}, \eqref{beta-definition}.
 Taking $\eta_1=\beta(|\cdot|)$ and $\eta_0=\eta$ in the definitions \eqref{Tm def}, \eqref{TM def}, we have  $T^\lambda= \sum_{m=0}^{M} T^{\lambda, m}$.

If $\lambda \delta_\circ \leq 1$, the bound trivially follows from summing in $m$ the estimates in (ii) in Proposition \ref{curvdecprop}.

If $\lambda \delta_\circ \geq 1$, note that the bounds in (i) in Proposition \ref{curvdecprop} imply
\begin{align*}
    \| T^\lambda \|_{L^2\to L^2} & \lesssim  \, \lambda^{-\frac{d-1}{2}} \Big( \sum_{\substack{1 \leq  2^m \leq (\lambda \delta_\circ)^{1/3}}} 2^{m/2} (\lambda \delta_\circ)^{-1/2} + \sum_{\substack{ (\lambda \delta_\circ)^{1/3}  < 2^m \leq \lambda^{1/2}}} 2^{-m} \Big)  \| a \|_{C^N}      \\
     & \lesssim \, \lambda^{-\frac{d-1}{2}} (\lambda \delta_\circ)^{-1/3} \| a \|_{C^N},
\end{align*}
as desired.
\end{proof}

\subsection{\it Proof of Proposition \ref{curvdecprop}}
We fix $N\ge 100 d$. As the operators depend linearly on $a$ we may assume $\|a\|_{C_N}\le 1$.
 The proof is based on a variant of the arguments in \cite{phong-stein91}, \cite{Cuccagna}, \cite{GS1999}; the latter two are themselves inspired by the Calder\'on--Vaillancourt theorem on the $L^2$ boundedness of pseudo-differential operators \cite{CV1971}.
 Again, by performing translations we may take $(x_0; z_0)=(0;0)$.

 Recall that, by hypothesis, $\sigma(0;0)=0$ and by \eqref{sigma bounds 4} and \eqref{sigma bounds 3}  we have that
$|\partial_{x_d} \sigma |\sim \delta_\circ$,
$| \partial_{z_d}\sigma|\sim \delta_\circ$ and $|\partial_{x}^\alpha \partial_z^\beta \sigma|\lesssim_{\alpha,\beta} \delta_\circ $ in $B_{ \varepsilon_\circ}(0;0)$ for some small $\varepsilon_\circ >0$. By an application of a quantitative
version of the  implicit function theorem  (see for example  \cite[\S8]{Christ1985}) 
there exist  smooth functions   
$$
(x';z)\mapsto \fu(x';z) \qquad \text{ and } \qquad (x;z')\mapsto \fv(x;z'),
$$
defined for $|x'|\le 2\eps_\circ$, $|z|\le 2\eps_\circ$ and 
$|x|\le 2\eps_\circ$, $|z'|\le 2\eps_\circ$  respectively, such that
$$
\sigma(x', \fu(x';z);z)=0\qquad
\text{ and }  \qquad
\sigma(x;z', \fv(x;z'))=0.
$$
Furthermore, by \eqref{sigma bounds 3}
$$
|\fu(x';z)-x_d|, \,\,|\fv(x;z')-z_d| \sim \delta_\circ^{-1} |\sigma(x;z)|.   
$$ 
We may expand 
$|x_d-y_d|\le |x_d-\fu(x';z)|+|\fu(x';z)-\fu(y';z)|+|\fu(y';z)-y_d |$
and obtain the crucial estimate 
\begin{align}
 \label{orthogonalityx}
&|\sigma(x;z)|\sim 2^{-m}\delta_\circ,\, |\sigma(y;z)|\sim 2^{-m}\delta_\circ  \implies 
 |x_d-y_d|
\lc 2^{-m}+|x'-y'|
\\
\intertext{ and similarly (using $\fv$) }
&|\sigma(x;w)|\sim 2^{-m}\delta_\circ,\, |\sigma(x;z)|\sim 2^{-m}\delta_\circ  \implies  |w_d-z_d| 
\lc 2^{-m}+|w'-z'|. \notag
\end{align}

These observations suggest further decomposing the amplitude into functions supported essentially on
 $C\eps_\circ 2^{-m}$ cubes. Let $\zeta\in C^\infty_0(\bbR)$ supported in $(-1,1)$ such that $\sum_{n\in \bbZ}\zeta(\cdot-n)\equiv 1$.
 Set 
 \begin{equation*}
    b_{\mu\nu}^{m}(x;z) := a(x;z)\eta_1(2^m\delta_\circ^{-1}\sigma(x;z))\Big(\prod_{j=1}^d\zeta(\varepsilon_{\circ}^{-1}2^m x_j - \mu_j)\zeta(\varepsilon_{\circ}^{-1}2^m z_j - \nu_j)\Big)
    \end{equation*}
    and write the corresponding decomposition
            $$T^{\lambda, m}=  
    \sum_{(\mu,\nu)\in \bbZ^d\times\bbZ^d} T_{\mu\nu}^{\lambda, m}$$ where
    $ T_{\mu\nu}^{\lambda, m} f := T^\lambda [\Psi; b_{\mu \nu}^m] f.$ Observe that 
    \[ |\partial_{x}^{\alpha} \partial_z^\beta  b_{\mu\nu}^m(x;z)|    \lc 2^{m(|\alpha|+|\beta|)}  \]
for all $\alpha,\beta \in \N_0^d$. 
 By the  Cotlar--Stein lemma (see, for instance,~\cite[Chapter VII, \S 2.1]{Stein1993}), the proof of the proposition reduces to showing the estimates
 \Be\label{Cotlar-Stein}
     \|T^{\lambda,m}_{\mu\nu}(T^{\lambda, m}_{\tilde{\mu}\tilde{\nu}})^*\|_{2\to 2} + \|(T^{\lambda, m}_{\mu\nu})^*T^{\lambda, m}_{\tilde{\mu}\tilde{\nu}}\|_{2 \to 2} \lesssim 
      \frac{\lambda^{-(d-1)} \min \{ 2^m/(\lambda\delta_\circ) , 2^{-2m}\}}
       {(1 + |\mu - \tilde{\mu}| + |\nu - \tilde{\nu}|)^{3d}} 
\Ee
for all $(\mu, \nu), (\tilde{\mu}, \tilde{\nu}) \in \Z^d \times \Z^d$. The proof of \eqref{Cotlar-Stein} is divided in two cases.




\subsubsection*{Off-diagonal estimates} 
The first step is to establish \eqref{Cotlar-Stein} in the \textit{off-diagonal} case where 
\begin{equation}\label{off-diagonal case}
    \max\{|\mu - \tilde{\mu}|, |\nu - \tilde{\nu}|\} \geq C_{\mathrm{diag}}\varepsilon_{\circ}^{-1}
\end{equation}
for a large absolute constant  $C_{\mathrm{diag}} \geq 2$, chosen  independently of $\varepsilon_{\circ}$. To this end, it is convenient to introduce the kernels associated to the operators of the type $TT^*$ and $T^*T$.
The Schwartz kernel of $T^{\lambda, m}_{\mu\nu}(T^{\lambda, m}_{\tilde{\mu}\tilde{\nu}})^*$  is given by 
\Be\label{degenerate kernel}
K^{\lambda, m}_{\mu\nu,\tilde{\mu}\tilde{\nu}}(x,y) 
:= \int_{\R^d} e^{i\lambda(\Psi(x;z) - \Psi(y;z))} b^{m}_{\mu\nu,\tilde{\mu}\tilde{\nu}}(x,y;z)\,\ud z,
\Ee and the Schwartz kernel of 
 $(T^{\lambda, m}_{\mu\nu})^*T^{\lambda, m}_{\tilde{\mu}\tilde{\nu}}$ is given by 
$$
  \tilde{K}^{\lambda, m}_{\mu\nu,\tilde{\mu}\tilde{\nu}}(z,w) := \int_{\R^d} e^{-i\lambda(\Psi(x;z) - \Psi(x;w))} \tilde{b}^{m}_{\mu\nu,\tilde{\mu}\tilde{\nu}}(x;z, w)\,\ud x;
$$
here the symbols are given by
\begin{equation}
      b^{m}_{\mu\nu,\tilde{\mu}\tilde{\nu}}(x,y;z) := b_{\mu\nu}^{m}(x;z)\overline{b_{\tilde{\mu}\tilde{\nu}}^{m}(y;z)}, \qquad  \tilde{b}^{m}_{\mu\nu,\tilde{\mu}\tilde{\nu}}(x;z, w) := b_{\mu\nu}^{m}(x;z)\overline{b_{\tilde{\mu}\tilde{\nu}}^{m}(x;w)}.\nonumber
\end{equation}


\begin{lemma}[Off-diagonal estimate]\label{off-diagonal lemma} 
 Let $1 \leq 2^m\le \lambda^{1/2}$
and suppose that  \eqref{off-diagonal case} holds.
\begin{enumerate}[i)]
\item If $|\mu-\tilde{\mu}| \geq C_{\mathrm{diag}} \varepsilon_{\circ}^{-1}$, then $(T^{\lambda, m}_{\mu\nu})^*T^{\lambda, m}_{\tilde{\mu}\tilde{\nu}} \equiv 0$ and 
\[
\|T^{\lambda,m}_{\mu\nu}(T^{\lambda, m}_{\tilde{\mu}\tilde{\nu}})^*\|_{2 \to 2} \lc_N
2^{-2dm} (\lambda2^{-2m} |\mu-\tilde\mu|)^{-N} .
\]
    \item If $|\nu-\tilde{\nu}| \geq C_{\mathrm{diag}} \varepsilon_{\circ}^{-1}$, then $T^{\lambda, m}_{\mu\nu}(T^{\lambda, m}_{\tilde{\mu}\tilde{\nu}} )^*\equiv 0$ and
    \[
\|(T^{\lambda, m}_{\mu\nu})^*T^{\lambda, m}_{\tilde{\mu}\tilde{\nu}}\|_{2 \to 2} \lc_N
2^{-2dm} (\lambda2^{-2m} |\nu-\tilde \nu|)^{-N} .
\]


\end{enumerate}
\end{lemma}

\begin{proof} Only the proof of i) is given; the same argument can be applied to ii) \textit{mutatis mutandis} (the asymmetry of assumptions regarding the $x_d$ dependence does not make a difference for the current proof).
Furthermore, if $|\mu - \tilde{\mu}| \geq 2$, then it immediately follows from the support properties of the symbols that $(T^{\lambda, m}_{\mu\nu})^*T^{\lambda, m}_{\tilde{\mu}\tilde{\nu}} \equiv 0$ and it only remains to consider the Schwartz kernel $K^{\lambda, m}_{\mu\nu,\tilde{\mu}\tilde{\nu}}(x,y)$  of 
 of $T^{\lambda, m}_{\mu\nu}(T^{\lambda, m}_{\tilde{\mu}\tilde{\nu}})^*$. By Schur's test, the desired estimate follows  from 
\begin{multline}\label{kerneloffdiagbounds}
   \sup_{x \in \R^d} \int_{\R^d} |K^{\lambda,m}_{\mu \nu , \tilde{\mu} \tilde{\nu}}(x, y)| \, \ud y, \:\: \sup_{y \in \R^d} \int_{\R^d} |K^{\lambda,m}_{\mu \nu , \tilde{\mu} \tilde{\nu}}(x, y)| \, \ud x   \lesssim \frac{2^{-2dm}(2^{-2m} \lambda)^{-N}}{|\mu-\tilde{\mu}|^{N}}.
\end{multline}

First note that, provided $C_{\mathrm{diag}}$ is suitably chosen, combining the hypothesis $|\mu - \tilde{\mu}| \geq C_{\mathrm{diag}}\varepsilon_{\circ}^{-1}$ with \eqref{orthogonalityx} yields \begin{equation}\label{off-diagonal 1}    |x_d - y_d| \lesssim |x' - y'| \qquad \textrm{on $\mathrm{supp}\,b^{m}_{\mu\nu,\tilde{\mu}\tilde{\nu}}$.}\end{equation}
Thus, by Taylor's theorem and \eqref{off-diagonal 1}
\begin{subequations}\Be\label{partz'derupper}
\big|\partial_{z'}^\alpha (\Psi (x;z) -  \Psi (y;z))\big|\lc_\alpha |x'-y'|.
\Ee
For the lower bounds we use \eqref{x'z'} and, from  \eqref{lastcolumnrow}, $\partial_{z'x_d}^2 \Psi(0;0)=0$, to  deduce 
\begin{equation*}
\partial_{z'}\Psi (x;z)-\partial_{z'}\Psi (y;z) =\int_0^1\partial_{z'x'}^2 \Psi (y+s(x-y);z) \,\ud s \, (x'-y') + O(\eps_\circ|x_d-y_d|).
\end{equation*} 
Thus, from \eqref{off-diagonal 1} we obtain that, for $(x,y;z)$ near $(0,0;0)$, 
\Be\label{partz'derlower}
\big|\partial_{z'} (\Psi (x;z) -  \Psi (y;z))\big|\ge c|x'-y'|.
\Ee
Finally, $|\partial^\alpha_z b^m_{\mu\nu,\tilde\mu\tilde\nu}|\lc_\alpha  2^{m|\alpha|}$,
 and the $z$-integration is extended over a set of diameter $O(2^{-m})$. 
\end{subequations}
By \eqref{partz'derlower} and \eqref{partz'derupper}, we may  use repeated integration-by-parts in the form of  Corollary \ref{corollaryapp}, 
with the choices of  $\rho(x,y):=|x'-y'|$ and   $R(x,y):=1$, to
obtain 
\begin{align*} |K^{\lambda, m}_{\mu\nu,\tilde\mu\tilde\nu} (x,y)|&\lc  2^{-dm} (2^{-m} \lambda|x'-y'|)^{-N} .
\end{align*}
By \eqref{orthogonalityx}, the kernel is identically zero unless $|\mu_3 - \tilde{\mu}_3| \lesssim \max\{1, |\mu' - \tilde{\mu}'|\}$. Provided $C_{\mathrm{diag}}$ is sufficiently large, $|\mu'-\tilde{\mu}'| \sim |\mu-\tilde{\mu}|$ and, furthermore, $|\mu'-\tilde{\mu}'| \geq 2$. Consequently, $\varepsilon_{\circ}^{-1}2^m|x'-y'| \sim |\mu - \tilde{\mu}|$ and so
$$
|K^{\lambda, m}_{\mu\nu,\tilde\mu\tilde\nu} (x,y)|\lc 2^{-dm} (2^{-2m}\lambda|\mu-\tilde{\mu}|)^{-N}.
$$
For fixed $x$, the support of $y\mapsto K^{\lambda, m}_{\mu\nu,\tilde\mu\tilde\nu}(x,y) $ 
is a set of measure $O(2^{-dm})$ and likewise, for fixed $y$  the support of
$x\mapsto K^{\lambda, m}_{\mu\nu,\tilde\mu\tilde\nu}(x,y) $, and \eqref{kerneloffdiagbounds} follows.
\end{proof}




\subsubsection*{Diagonal estimates} The proof of \eqref{Cotlar-Stein} has now been reduced  to the following two lemmata.

\begin{lemma}
\label{diagonal lemma-M} Suppose that $\lambda\ge 1$ and $1 \leq 2^m\lc \lambda^{1/2}$.
Then, 
for all $(\mu,\nu) \in \Z^d \times \Z^d$,
\begin{equation*}
\|T^{\lambda, m}_{\mu\nu}\|_{2 \to 2} \lesssim 2^{-m} \lambda^{-(d-1)/{2} } .
\end{equation*}
Furthermore,
\begin{equation*}
\|T^{\lambda, M}_{\mu\nu}\|_{2 \to 2} \lesssim  \lambda^{-d/2} .
\end{equation*}

\end{lemma}

\begin{lemma}
\label{diagonal lemma-m} 
Suppose that $\lambda\delta_\circ\ge 1$ and 
$1\le 2^m \le (\lambda\delta_\circ)^{1/3}$. Then for all $(\mu,\nu) \in \Z^d \times \Z^d$, 
\begin{equation*}
\|T^{\lambda, m}_{\mu\nu}\|_{2 \to 2} \lesssim 2^{m/2}\delta_\circ^{-1/2}\lambda^{-d/2}. 
\end{equation*}
\end{lemma}

Note that the estimate in Lemma \ref{diagonal lemma-m} is better than the estimate in Lemma \ref{diagonal lemma-M} in the range $\lambda\delta_\circ\ge 1$, $1 \leq 2^m \le (\lambda\delta_\circ)^{1/3}$.

\begin{proof}[Proof of Lemma \ref{diagonal lemma-M}]
Let  $I_{\mu_d}$, $J_{\nu_d}$ denote the intervals of length $\eps_\circ2^{1-m}$ centered at 
$x_{\mu_d}=\eps_\circ 2^{-m} \mu_d$ ,
$z_{\nu_d}=\eps_\circ 2^{-m} \nu_d$ , respectively. For $g\in L^2(\bbR^{d-1})$ define 
\[T^{\lambda,m,x_d,z_d}_{\mu\nu}g(x')
 = \int_{\R^{d-1}} e^{i\lambda\Psi(x',x_d;z', z_d)} b_{\mu,\nu}^m(x;z) g(z') \, \ud {z'}\] and observe that
$$
T^{\lambda, m}_{\mu\nu} 
f(x)= \bbone_{I_{\mu_d} } (x_d) \int_{J_{\nu_d} } 
T^{\lambda, m,x_d,z_d}_{\mu\nu}[f(\cdot,x_d)] \, \ud{z_d}.
$$
The Schwartz kernel $K_{\mu\nu}^{\lambda, m,x_d,z_d} (x',y')$ of  $T^{\lambda, m,x_d,z_d}_{\mu\nu}(T^{\lambda, M,x_d,z_d}_{\mu\nu})^*$ is equal to
\[
\int_{\R^{d-1}} e^{i\lambda (\Psi (x', x_d;z', z_d)- \Psi (y', x_d;z', z_d))} 
b_{\mu\nu}^m(x',x_d;z',z_d) \overline{ b_{\mu\nu}^m(y', x_d;z',z_d)} \, \ud {z'}.
\] 
We use 
integration-by-parts based on \eqref{x'z'}; that is, we use the $(d-1)$-dimensional case of Corollary \ref{corollaryapp}  with  the choices     $\rho(x',y'):=|x'-y'|$, $R(x,y):=1$ and  the fact that $\partial_{z'}^\alpha$ applied  to the amplitude yields a term which is $O(2^{m|\alpha|})$.
This implies
\[|K_{\mu\nu}^{\lambda, m,x_d,z_d} (x',y')|\lc_N  2^{-m(d-1)}(1+\lambda 2^{-m}|x'-y'|)^{-N}\]
 uniformly in $x_d, z_d$, and
by the  Schur's test one has 
$$\|T^{\lambda,m,x_d,z_d}_{\mu,\nu}
\|_{L^2(\R^{d-1})\to L^2(\R^{d-1})} \lesssim \lambda^{-(d-1)/2} . $$
Consequently,
\begin{align*}
\|T^{\lambda, m}_{\mu\nu}f\|_{L^2(\R^d)} & \lc \int_{J_{\nu_d} } 
\Big(\int_{I_{\mu_d}} 
\|T^{\lambda, m,x_d,z_d}_{\mu\nu}
 [f(\cdot, z_d)]\|_{L^2(\bbR^{d-1} ) }^2 \ud{x_d} \Big)^{1/2} \ud{z_d} 
\\ &\lc 2^{-m/2}  \lambda^{-(d-1)/2} 
\int_{J_{\nu_d} } 
\|[f(\cdot, z_d)]\|_{L^2(\bbR^{d-1} ) }  \, \ud{z_d} \\
& \lesssim 2^{-m}  \lambda^{-(d-1)/2}  \|f\|_{L^2(\R^d)}
\end{align*} 
and hence 
$\|T^{\lambda, m}_{\mu\nu} \|_{2 \to 2} \lc  2^{-m} \lambda^{-(d-1)/2}$  as desired. 
The arguments for $T^{\lambda, M}_{\mu\nu}$ is analogous.
\end{proof}

\begin{proof}[Proof of Lemma \ref{diagonal lemma-m}]
Let $K_{ \mu\nu}^{\lambda, m} := K_{\mu\nu,\mu\nu}^{\lambda, m}$ denote the kernel of $T^{\lambda,m}_{\mu\nu}(T^{\lambda,m}_{\mu\nu})^*$, as given by the formula in 
\eqref{degenerate kernel}. It will also be useful to write $b_{\mu\nu}^m$ for the symbol $b_{\mu\nu,\mu\nu}^{m}$. By the Schur test, the problem is reduced to showing
\begin{subequations}
\begin{align}\label{diagonal estimate 1}
  &\sup_{y \in \R^d} \int_{\R^d} |K_{\mu\nu}^{\lambda, m}(x, y)| \, \ud x 
\lesssim  2^m\delta_\circ^{-1}\lambda^{-d},
\\ 
&\sup_{x \in \R^d} \int_{\R^d} |K_{\mu\nu}^{\lambda, m}(x, y)| \, \ud y 
\lesssim  2^m\delta_\circ^{-1}\lambda^{-d} .
\label{diagonal-partb}
\end{align}
\end{subequations}
 Since 
$T^{\lambda, m}_{\mu\nu}(T^{\lambda, m}_{\mu\nu})^*$ is self-adjoint \eqref{diagonal-partb} follows from 
\eqref{diagonal estimate 1}. We proceed to show \eqref{diagonal estimate 1}.

Since the partial mixed Hessian $\partial^2_{z'x'}\Psi$ is non-singular, there exist local solutions in $x'$ to the implicit equation $\nabla_{z'} \Psi (x;z)=\nabla_{z'} \Psi (y;z)$. In particular, by applying a quantitative version of the implicit function theorem (see, for instance,~\cite[\S 8]{Christ1985}), provided $\varepsilon_{\circ}$ is 
chosen suitably small, there exists a smooth $\bbR^{d-1}$-valued function $(x_d,y,z)\mapsto X(x_d;y;z)$ defined by 
\begin{align}
 \partial_{z'} \Psi (X(x_d;y;z), x_d;z) &= \partial_{z'} \Psi (y;z),    \label{Xdef1}  \\
     X(y_d;y;z) &= y' . \label{Xdef2}  
\end{align}
Implicit differentiation yields
\begin{equation}\label{X derivative}
  \partial_{x_d}X(x_d;y;z) 
  = - (\partial^2_{x'z'}\Psi)^{-1}  
    \partial_{z'x_d}^2\Psi \Big|_{(X(x_d;y;z),x_d; z)}.
\end{equation}
From this formula, the chain rule and the definition of $\sigma$ one deduces that
\begin{equation}\label{sigma derivative formula}
    \partial_{x_d}  \big[ \partial_{z_d}\Psi(X(x_d;y;z), x_d;z))\big] = \sigma(X(x_d;y;z), x_d;z) .
\end{equation}
Notice that the right hand side of \eqref{X derivative} vanishes at $(x_d;y;z)=(0;0;0)$ and that $\partial_{x_d}^\alpha X(x_d;y;z) =O(\delta_\circ)$. Hence 
we get
\begin{equation}\label{Xxd bounds}
     |\partial_{x_d}X(x_d;y;z)| \lesssim \varepsilon_{\circ} \delta_\circ.
\end{equation}
Moreover, implicit differentiation of \eqref{Xdef1} with respect to $z$ yields
\begin{align*} 
\partial_{z'x'}^2\Psi (X(x_d;y;z),y_d; z) \partial_z X(x_d;y;z) 
&= 
\partial_{z'z}^2\Psi(y',y_d;z)- \partial_{z'z}^2\Psi(X(x_d;y;z),x_d;z) 
\\
&\lc |y'-X(x_d;y;z)|+\delta_\circ |x_d-y_d|  \\ & =\, O(\delta_\circ|x_d-y_d|),
\end{align*}
where we have used \eqref{Xdef2} and \eqref{Xxd bounds}. This gives  
\Be\label{Xzd bounds}
|\partial_{z} X(x_d;y;z)| \lesssim \delta_\circ|x_d - y_d|.
\Ee

We shall now state the inequalities  for the integration-by-parts argument
which will allow us to prove \eqref{diagonal estimate 1}.
In what follows we write
$X := X(x_d;y;z)$ and $X_{\nu} := X(x_d;y;z_{\nu})$ where $z_{\nu} := \varepsilon_{\circ}2^m\nu$, noting that the $z$-support of $b_{\mu\nu}^m$ lies in a ball of radius $O(\varepsilon_{\circ}2^{-m})$ about this point. We claim that
\Be \label{nablazphaseupper}
\big|\partial_z^\alpha\Psi(x;z)-\partial_z^\alpha\Psi(y;z)\big| \le 
C_\alpha\big( |x'-X_\nu|+ \delta_\circ  |x_d-y_d|\big)
\Ee
and 
\Be\label{nablazphaselower}
|\nabla_z\Psi(x;z)-\nabla_z\Psi(y;z)| \ge 
c\big( |x'-X_\nu|+ \delta_\circ 2^{-m} |x_d-y_d|\big).
\Ee

To see \eqref{nablazphaseupper}, by Taylor expansion the left-hand side is dominated by  a constant times 
$|x'-y'|+\delta_\circ|x_d-y_d|$. We then bound $|x'-y'|\leq |x'-X_\nu| + |y'-X_\nu|$ and, using \eqref{Xdef2}, by the mean value theorem, \eqref{Xxd bounds} and \eqref{Xzd bounds} one has
\begin{align*}|y'-X_\nu|&\le |X(x_d;y;z)-X(y_d;y;z)|
+|X(y_d;y;z)-X(y_d;y;z_\nu)|
\\
&\lc \delta_\circ|x_d-y_d|.
\end{align*}
Now \eqref{nablazphaseupper} easily follows.

We turn to \eqref{nablazphaselower}.
Taking a Taylor expansion in the $x'$ variables,
\begin{align}
\nonumber
    \partial_{z'}\Psi(x;z)-\partial_{z'}\Psi(y;z) &= \partial_{z'}\Psi(x;z) - \partial_{z'}\Psi(X, x_d;z) \\
\label{diagonal estimate 2}
&= \partial_{z'x'}^2\Psi(X, x_d;z) (x'-X) + O(|x' - X|^2)
\end{align}
whilst, by a Taylor expansion in the $z$-variables, the last  expression is equal to
\begin{equation}\label{diagonal estimate 2 bis}
\partial_{z'x'}^2\Psi (X, x_d;z) (x'-X_{\nu}) + O\big(|x' - X_{\nu}|^2 + \varepsilon_{\circ}2^{-m}\delta_\circ|x_d-y_d|\big).
\end{equation}
Here the additional error term arises by applying the mean value theorem to $|X-X_\nu|$ together with \eqref{Xzd bounds}.

On the other hand, one may write $\partial_{z_d}\Psi (x; z)-\partial_{z_d}\Psi  (y; z) = \mathrm{I} + \mathrm{II}$ where
\begin{equation*}
   \mathrm{I} := \partial_{z_d}\Psi (X,x_d;z) - \partial_{z_d}\Psi  (y; z), \quad \mathrm{II} := \partial_{z_d} \Psi (x; z)-\partial_{z_d}\Psi (X,x_d;z).
\end{equation*}
To estimate $\mathrm{I}$, take a Taylor expansion first in the $x_d$ variable and then in the $z$ variable to obtain
\begin{align}
\nonumber
   \mathrm{I} &= \sigma(y;z)(x_d - y_d) + O(\delta_\circ|x_d - y_d|^2) \\
   \label{diagonal estimate 2.5}
   &= \sigma(y;z_{\nu})(x_d - y_d) + O\big(\varepsilon_{\circ}2^{-m}\delta_\circ|x_d - y_d| \big).
\end{align}
Here $\sigma$ appears owing to \eqref{sigma derivative formula} and \eqref{Xdef2}. The second estimate holds due to \eqref{sigma bounds 4} and the localisation of the $(x,y;z)$-support of $b_{\mu\nu}^m$. To estimate the $\mathrm{II}$ term, arguing as in \eqref{diagonal estimate 2}, take a Taylor expansion in the $x'$ variable and then in the $z$ variable to obtain
\begin{align}\nonumber
  \mathrm{II} &=  
  \partial_{z_dx'}^2\Psi(X,x_d;z)(x'-X) 
  + O(|x' - X|^2)  \\
  \label{II bound}
   &=  \partial^2_{z_dx'}\Psi(X,x_d;z)(x'-X_\nu)  + O\big(\varepsilon_{\circ}|x' - X_{\nu}| + \varepsilon_{\circ}2^{-m}\delta_\circ|x_d - y_d| \big). 
\end{align}
In the last step we applied \eqref{Xzd bounds}.
From \eqref{diagonal estimate 2 bis},
\eqref{diagonal estimate 2.5} and \eqref{II bound} 
we get (assuming $\eps_\circ$ is chosen sufficiently small)
that
\begin{multline*}
|\partial_{z'}\Psi
(x;z)-\partial_{z'}\Psi(y;z)| \ge c_1|x'-X_\nu|    \quad\text{ if  }  |x'-X_\nu| \ge C_1\eps_\circ2^{-m} \delta_\circ|x_d-y_d|
\end{multline*} and
\begin{multline*}
|\partial_{z_d}\Psi(x;z)-\partial_{z_d}\Psi(y;z)| \ge (\delta_\circ/2) 2^{-m}|x_d-y_d|   \\ \text{ if  }   |x'-X_\nu| \le C_1\eps_\circ2^{-m} \delta_\circ|x_d-y_d|,
\end{multline*}
and these inequalities imply  \eqref{nablazphaselower}.

We now estimate 
 $K_{\mu\nu}^{\lambda, m}(x, y) $. 
Using just the size and support of the integrand we get
\Be\label{integrand-size}
|K_{\mu\nu}^{\lambda, m}(x, y) | \lc 2^{-md}\Ee
which we use for $|x'-X_\nu|+2^{-m}\delta_\circ|x_d-y_d|\le \lambda^{-1}$.

Now assume  $|x'-X_\nu|+2^{-m}\delta_\circ|x_d-y_d|\ge \lambda^{-1}$; we use integration-by-parts to improve on \eqref{integrand-size}. By 
\eqref{nablazphaseupper}, \eqref{nablazphaselower} 
we can apply  Corollary \ref{corollaryapp}  
with the choices    
 $R(x,y):=2^m$ and  $\rho(x,y):=|x'-X_\nu(x_d,y;z_\nu)|+2^{-m}\delta_\circ|x_d-y_d|$. 
 We also use that for fixed $x,y$ the amplitude is supported  in a set of diameter $2^{-m}$ and the  estimates \[|\partial_z^\alpha[ b^m_{\mu\nu} (x,z) b^m_{\mu\nu}(y,z)]|\lc 2^{m|\alpha|} 
.\] 
Altogether, Corollary \ref{corollaryapp}  yields, for $x\neq y$,
\[
 |K_{\mu\nu}^{\lambda, m}(x, y) | \lesssim  2^{-md} \la^{-N} 2^{mN}
 \big(|x'-X_\nu|+\lambda 2^{-m}\delta_\circ|x_d-y_d|)\big)^{-N} .
\]
 Combining this  with \eqref{integrand-size} we obtain
 $$
 |K_{\mu\nu}^{\lambda, m}(x, y) | \lesssim  2^{-md} 
 \big(1+\lambda 2^{-m} |x'-X_\nu|+\lambda 2^{-2m}\delta_\circ|x_d-y_d|)\big)^{-N} .
$$
 
 Fixing $y$ and integrating in $x$ yields 
\[\int_{\R^d} |K_{\mu\nu}^{\lambda, m}(x, y) | \, \ud x \lc 2^{-md} 
(2^m\lambda^{-1})^{d-1} 2^{2m}\lambda^{-1} \delta_\circ^{-1} \lc
2^{m}\delta_\circ^{-1}\lambda^{-d} ,
\]
which is the desired estimate for the first term in \eqref{diagonal estimate 1}. 
This finishes the proof of \eqref{diagonal estimate 1} and the proof  of the lemma.
\end{proof}




\subsection{Uniform estimates depending on a $t$-variable}\label{subsec:t extension}

The estimates obtained in Propositions \ref{2par-Hoerprop}, \ref{ps-thm} and \ref{curvdecprop} will be used to obtain $L^2$-bounds for the operators $A_j[\Phi_t; \fa_t]$. To this end, we shall allow a $t$-dependence in our operator and obtain uniform estimates in $t$. Consider now an open set $U \subset \R^d \times \R \times \R^d$, a phase function $\Psi:  U \to \R$ and an amplitude $a \in C_0^{\infty}(U)$, and define
\Be\label{Psit-at-def}
\Psi_t(x;z)=\Psi(x;t;z) \quad \text{ and } \quad a_t(x;z)=a(x;t;z).
\Ee
Given $\lambda \geq 1$, let $T_t^\lambda$ denote the  oscillatory integral associated to the  pair $[\Psi_t; a_t]$ as in \eqref{TLa}, given by $T_t^\lambda \equiv T^\lambda[\Psi_t;a_t]$. 
For $0 < \delta_\circ \leq 1$, we assume that the condition \eqref{upperboundsder} continues to hold under $t$-derivatives. That is, the estimates
\begin{equation}\label{upperboundsder t included}
|\partial_x^\alpha \partial_z^\beta \partial_t^\gamma \Psi_t(x;z)| + \delta_\circ^{-1}|\partial_x^\alpha \partial_z^\beta \partial_t^\gamma \partial_{x_d} \Psi_t(x;z)| \leq C_{\alpha,\beta,\gamma}
\end{equation}
hold for all $(x;t;z) \in U$ and all $\alpha, \beta \in \N_0^d$, $\gamma \in \N_0$. Thus, if the condition $|\det \partial_{zx}^2 \Psi_{t_0}(x_0;z_0)| \geq c \delta_\circ$ holds  for some $(x_0;t_0;z_0) \in U$, Proposition \ref{2par-Hoerprop} in conjunction with  \eqref{upperboundsder t included} immediately extends to  a uniform estimate for the operators $T_t^\lambda$ for all $|t-t_0| \leq \varepsilon_\circ$, for  suitable $\eps_\circ$.
Likewise if \eqref{upperboundsder t included} holds and 
 the conditions \eqref{x'z'}, \eqref{lastcolumnrow} and \eqref{folds} are satisfied at a certain $(x_0;t_0;z_0) \in U$, Propositions \ref{ps-thm} and \ref{curvdecprop} also extend to the operators $T_t^\lambda$ for all $|t-t_0|\leq \varepsilon_\circ$, with uniform bounds on $t$; note that \eqref{upperboundsder t included} implies that the quantity $\sigma_t(x;z)\equiv \sigma (x;t;z)$ defined as in \eqref{sigmadef}, also satisfies the derivative bounds \eqref{sigma bounds 4} under $t$-differentiation, that is,
\begin{equation}\label{sigma_t der}
|\partial^\alpha_x \partial_z^\beta \partial^\gamma_t \sigma_t(x;z)| \lesssim_{\alpha,\beta,\gamma} \delta_\circ
\end{equation}
holds for all $(x;t;z) \in U$ and all $\alpha,\beta \in \N_0^d$, $\gamma \in \N_0$.




\subsection{ Estimates for maximal oscillatory integrals}
We now state the version of the estimates in Propositions \ref{2par-Hoerprop} and \ref{ps-thm} for the maximal functions
associated to the oscillatory integral operators $T_t^\lambda$.

To obtain such maximal estimates we will assume that \eqref{upperboundsder t included} holds and that, in addition,  there is $\delta_\circ$-smallness when we differentiate with respect to the $t$-variable; more precisely we assume that
\begin{equation}\label{delta smallness t derivative}
| \partial_x^\alpha\partial_z^\beta\partial_t^\gamma \Psi_t(x;z)| \lesssim_{\gam} \delta_\circ 
\end{equation}
holds for all $(x;t;z) \in U$ and all $\alpha, \beta \in \N_0^d$,  $\gamma>0$.

\begin{proposition}\label{Hoerm-max-prop}
Let  $[\Psi$; $a]$ be as in \eqref{Psit-at-def}.
Suppose 
$\Psi$  satisfies 
  \eqref{upperboundsder t included}, \eqref{delta smallness t derivative} and $|\det \partial_{zx}^2\Psi_{t_0}(x_0;z_0)|\ge c\delta_\circ$ for some $(x_0;t_0;z_0) \in U$.
  Then there is $\eps_\circ>0$ and $N>0$ such that, under the assumption of $a_t$ supported in $B_{\eps_\circ}(x_0,z_0)$,
 \[
\big\| \sup_{|t-t_0|\le \eps_\circ} |T^\lambda[\Psi_t;a_t] | \big\|_{L^2(\R^d) \to L^2(\R^d)} \lc \lambda^{-\frac{d-1}{2}} \| a \|_{C^N}.
\]
\end{proposition}

\begin{proposition}\label{fold-max-prop}
Let $[\Psi$; $a]$ be as in \eqref{Psit-at-def}. Assume that  $\Psi$ satisfies  \eqref{x'z'}, \eqref{lastcolumnrow}
\eqref{folds} at a certain $(x_0;t_0;z_0) \in U$, the estimates \eqref{upperboundsder t included} and \eqref{delta smallness t derivative} and, in addition, assume that
\Be\label{pointwise-est-t-deriv}
\partial_t \Psi_t(x;z)= c_t(x;z)\det \partial^2_{zx}\Psi_t(x;z)\Ee 
for some $c \in C^\infty$ in a neighbourhood of $\supp a$.
Then there is $\eps_\circ >0$ and $N>0$ such that, under the assumption of $a_t$ supported in $B_{\eps_\circ}(x_0,z_0)$
\[
\big\| \sup_{|t-t_0|\le \eps_\circ} |T^\lambda[\Psi_t;a_t] | \big\|_{L^2(\R^d) \to L^2(\R^d)} \lc \lambda^{-\frac{d-1}{2}} 
\log(2+\lambda \delta_\circ) \|a\|_{C^N}.
\]
\end{proposition}

The proofs rely on a standard Sobolev embedding inequality
(see for instance ~\cite[Chapter XI, $\S$3.2]{Stein1993}).
Namely, for a $C^1$ function $t\mapsto g(t)$ supported on an interval $I$, with
$t_0\in I$,  we have, for $1\le p<\infty$,
\Be\label{sobolev inequality pre}
\sup_{t \in I}|g(t)|^p  \leq  |g(t_0)|^p + p  \|g\|_{L^p(I)}^{p-1} \|g'\|_p
\Ee
which follows by the fundamental theorem of calculus applied to $|g|^p$ and H\"older's inequality. We can apply this to  $F(x,t)$ with $F\in L^p(\bbR^d;C^1)$, and after integrating in $x$ and  another application of H\"older's inequality, 
\eqref{sobolev inequality pre}  gives
\begin{equation}\label{sobolev lemma}
  \|\sup_{t \in I}|F(\cdot,t)| \|_{L^p(\R^d)}^p \leq  \inf_{t_0\in I} \|F(\cdot,t_0)\|_{L^p(\R^d)}^p + p  \|F\|_{L^p(\R^d \times I)}^{p-1} \|\partial_t F\|_{L^p(\R^d \times I)}.
\end{equation}


\begin{proof} [Proof of Proposition \ref{Hoerm-max-prop}]
Note that if $T_t^\lambda :=T^\lambda[\Psi_t; a_t] $, then $\partial_t T_t^\lambda = T^\lambda[\Psi_t; d_t]$, where $d_t:=(i \lambda \partial_t \Psi_t) a_t + \partial_t a_t$. By \eqref{delta smallness t derivative} one has $\| d_t \|_{C^N} \lesssim (1+ \lambda \delta_\circ) \| a_t \|_{C^{N+1}}$. Thus, by the hypothesis and Proposition \ref{2par-Hoerprop} applied to $T_t^\lambda$ and $\partial_t T_t^\lambda$ (as discussed in \S\ref{subsec:t extension}), there exist $\varepsilon_\circ$ and $N>0$ such that, if $a_t$ is supported in $B_{\varepsilon_\circ}(x_0;z_0)$, the bounds
 \[
 (1+\lambda\delta_\circ)^{1/2} \|T_t^\lambda f\|_{L^2(\R^d)}+
  (1+\lambda\delta_\circ)^{-1/2} \|\partial_t T_t^\lambda f\|_{L^2(\R^d)}\lc \lambda^{-\frac{d-1}{2}} \| a \|_{C^{N+1}}\|f\|_{L^2(\R^d)}
  \]
  hold uniformly in $|t-t_0| \leq \varepsilon_\circ$. Now the assertion follows immediately by the Sobolev inequality \eqref{sobolev lemma} for the exponent $p=2$. 
  \end{proof}
 
 \begin{proof} [Proof of Proposition \ref{fold-max-prop}]
 Given $0\leq m \leq M$, let 
 \begin{align*} 
T_t^{\lambda, m} f(x)&:=\int_{\R^d} e^{i\lambda \Psi_t(x;z)} a_t(x;z)\beta(2^m\delta_\circ^{-1} |\sigma_t(x;z)|) f(z)\,  \mathrm dz, \quad 0 \leq m<M , 
\\
T_t^{\lambda, M} f(x)&:=\int_{\R^d} e^{i\lambda \Psi_t(x;z)} a_t(x;z)\eta(2^M\delta_\circ^{-1} \sigma_t(x;z) ) f(z) \, \mathrm dz;
\end{align*}
that is, \eqref{Tm def} and \eqref{TM def} with $\beta(|\cdot|)$, $\eta$ in place of $\eta_1$, $\eta_0$ and with the phase-amplitude pair $[\Psi_t; a_t]$.
 
 Using \eqref{pointwise-est-t-deriv} and \eqref{rotcurv vs sigma} we compute
 \begin{align*} 
 \partial_t T_t^{\la,m} f(x) &=
 \la \delta_\circ  2^{-m}
 \int_{\R^d} e^{i\la\Psi_t(x;z) } \widetilde c_t(x;z) a_t(x;z) \eta_{1}(2^m\delta_\circ^{-1}\sigma_t(x;z))  f(z) \, \mathrm d z
 \\
 &+ 2^m \int_{\R^d} e^{i\la\Psi_t(x;z) } 2^{-m} \partial_t a_t(x;z)  \beta(2^m\delta_\circ^{-1}|\sigma_t(x;z)|)  f(z) \, \mathrm d z
 \\&+ 2^m \int_{\R^d} e^{i\la \Psi_t(x;z)} \delta_\circ ^{-1}\partial_t\sigma_t(x;z) a_t(x;z)
 \widetilde\eta_1(2^m\delta_\circ^{-1}\sigma_t(x;z))  f(z)\, \mathrm d z
 \end{align*}
 where $\widetilde c_t=c_t \partial^2_{x'z'} \Psi_t$,  is smooth and $\eta_1(s)=s\beta(|s|)$, $\widetilde \eta_1(s)=\frac{d}{ds}(\beta(|s|))$.
 For  $m=M$ we have a similar formula   with $\beta$ replaced by $\eta$. 
 Note that in view of \eqref{sigma_t der} we have
 $|\partial_x^\alpha\partial_z^\beta [\delta_0^{-1}\partial_t\sigma_t a_t ] |\lc C_{\alpha,\beta}$.
 


 Assume $1 \leq 2^m \le (\lambda\delta_\circ)^{1/3}$. By the hypothesis and Proposition \ref{curvdecprop} applied to $T_t^{\lambda, m}$ and $\partial_t T_t^{\lambda, m}$, 
 there exist $\varepsilon_\circ$ and $N>0$ such that, if $a_t$ is supported in $B_{\varepsilon_\circ}(x_0;z_0)$, one has the bounds
 \begin{align*}\|T^{\lambda, m}_t f\|_{L^2(\R^d)} &\lc \lambda^{-\frac{d-1}{2}} \Big(\frac{2^m}{\lambda\delta_\circ}\Big)^{1/2} \|f\|_{L^2(\R^d)}
     \\
     \intertext{and} 
     \|\partial_t T^{\lambda, m}_t f\|_{L^2(\R^d)} & \lc \lambda^{-\frac{d-1}{2}} \Big(\frac{2^m}{\lambda\delta_\circ}\Big)^{1/2} 
     \big( \lambda\delta_\circ2^{-m}+2^m\big)
     \|f\|_{L^2(\R^d)}
     \\ &\lc  \lambda^{-\frac{d-1}{2}} \Big(\frac{2^m}{\lambda\delta_\circ}\Big)^{-1/2} 
    \|f\|_{L^2(\R^d)}
     \end{align*}
      uniformly in $|t-t_0|\leq \varepsilon_\circ$, where the last inequality follows because we are under the assumption $1 \leq 2^{m} \leq (\lambda \delta_\circ)^{1/3} \leq (\lambda \delta_\circ)^{1/2}$. Therefore, the above estimates combined with  \eqref{sobolev lemma} yield
  \Be\notag
 \sum_{0\le m\le\floor{\log_2 (\lambda\delta_\circ)^{1/3}}} \big\|\sup_{|t-t_0|\le\eps_\circ} |T_{t}^{\lambda, m}f| \big\|_{L^2(\R^d)} \lc
 \log(2+\lambda\delta_\circ) \lambda^{-\frac{d-1}{2}} \|f\|_{L^2(\R^d)}.
   \Ee
 Similarly, if $\lambda^{1/2} \geq 2^m \ge \min\{ (\lambda\delta_\circ)^{1/3},1\}$,  Proposition \ref{curvdecprop} implies 
  \begin{align*}\|T^{\lambda, m}_t f\|_{L^2(\R^d)} &\lc \lambda^{-\frac{d-1}{2}} 2^{-m} \|f\|_{L^2(\R^d)}
     \\
     \intertext{and}
     \|\partial_t T^{\lambda,m}_t f\|_{L^2(\R^d)} &\lc \lambda^{-\frac{d-1}{2}} 2^{-m} 
     \big( \lambda\delta_\circ2^{-m}+2^m\big)
      \|f\|_{L^2(\R^d)} 
     \end{align*}
     uniformly in $|t-t_0|\leq \varepsilon_\circ$. The above bounds imply, by \eqref{sobolev lemma}, that 
     $$
     \big\|\sup_{|t-t_0|\le\eps_\circ} |T_{t}^{\lambda, m}f| \big\|_{L^2(\R^d)} \lesssim 
     \lambda^{- \frac{d-1}{2}} 2^{-m} (\lambda \delta_\circ 2^{-m} + 2^m)^{1/2}  \|f\|_{L^2(\R^d)},
     $$
     and thus
     \begin{align*}
\sum_{\floor{\log_2 (\lambda\delta_\circ)^{1/3}}\wedge 1 \le m \le M} \big\|\sup_{|t-t_0|\le\eps_\circ} |T_{t}^{\lambda, m}f| \big\|_{L^2(\R^d)} \lc  \lambda^{-\frac{d-1}{2}}  \|f\|_{L^2(\R^d)}
 \end{align*}
 follows from summing a geometric series, as $\lambda \delta_\circ 2^{-m} \leq 2^{2m}$ in the range of summation. Combining both sums one obtains the desired bound by the triangle inequality, which concludes the proof of the proposition.
 \end{proof}




\subsection{  Radon-type operators in $d$ dimensions versus  oscillatory integral operators in $d+1$ dimensions}\label{reduction to oscillatory integral estimates section} 
In this section we use  variables $(x;z)\in \bbR^{d+1}\times\bbR^{d+1}$ and split 
$x=(x_1,x'')$, $z=(z_1,z'')$ with $x''\in \bbR^d$, $z''\in \bbR^d$.
Recall that the  frequency localised Radon-type operators in \eqref{frequence localised definition} are of the form (with $d=2$)
\begin{align} 
A_j[\Phi_t;\fa_t] f(x'')&= \int_{\R^d}  \fa_t(x''; z'') \int_\R \beta(2^{-j}|\theta|) e^{i\theta \Phi_t(x''; z'')}f(z'') \, \ud\theta \, \ud z''
\notag
\\
&= 2^j \int_{\R \times \R^d} \fa_{t}(x''; z'')\beta(|\om|)  e^{i2^j\om \Phi_t(x''; z'')}f(z'') \, \ud\om\,\ud z'',\label{Rj-expression}
\end{align}
We rely on an idea in  \cite[Chapter XI, \S 3.2.1]{Stein1993}
to show that a $L^p(\bbR^d)$ estimate for 
$\sup_{t\in I}|A_j[\Phi_t;a_t] f|$ is implied by a $L^p$-estimate for a maximal function associated with  a closely related family of oscillatory integral operators acting on functions on $\bbR^{d+1}$ which we will presently define.

Recall that $\beta$ is supported in $[1/2,2]$. Let $\widetilde \beta$ be supported in $(1/4,4)$ such that also $\widetilde \beta(s)=1$ for $s\in [1/3,3]$. Notice that $\widetilde \beta(s) \beta(us)=\beta(us)$ for $2/3<u<3/2$.
Now let $\chi_1\in C^\infty_0(\bbR)$ so that $\chi_1(r)=1$ on $J:=[2/3, 3/2]$.
Consider the  family of  oscillatory integral operators  $T^{2^j}[\phi_t;a_t]$, as defined in \eqref{TLa} but acting on functions $g$ on $\bbR^{d+1}$, where  
\begin{equation}\label{phases relation Aj Tlambda}
  \phi_t(x;z)=x_1z_1\Phi_t(x'';z''),\quad \text{ and } \quad  a_t(x;z)=\chi_1(x_1)x_1 \fa_t(x'';z'')\beta(x_1|z_1|).  
\end{equation}

\begin{lemma}\label{TvsA-lemma}
Let $E\subset (0,\infty)$, $\Phi$, $\phi$, $\fa$, $a$ as in \eqref{phases relation Aj Tlambda}, and define 
\[M_j [\Phi;\fa]f:=\sup_{t\in E} |A_j[\Phi_t;\fa_t]f|,  \qquad  
\cM_j[\phi;a]g=\sup_{t\in E} |T^{2^j}[\phi_t;a_t]g|.\]
Then
\[
\|M_j[\Phi;\fa]\|_{L^p(\bbR^d)\to L^p(\bbR^d)} 
\le  2^j(6/5)^{1/p} \|\widetilde\beta\|_{L^p(\bbR)} \|\cM_j[\phi;a]\|_{L^p(\bbR^{d+1})\to L^p(\bbR^{d+1})}.
\]
\end{lemma}
\begin{proof}
For fixed  $x_1$ we  change variables $\om=x_1z_1$ in \eqref{Rj-expression}. We  use that  $\chi(x_1)=1$ for $x_1\in J$ and  that 
$\widetilde \beta(|z_1|) \beta(x_1|z_1|)=\beta(x_1|z_1|)$ for $(x_1,z_1)\in J\times\bbR$ to obtain the identity
\begin{align} 
\notag
    A_j[\Phi_t;\fa_t] f (x'')&= 2^j  T^{2^j}[\phi_t;a_t] (\widetilde \beta\otimes f)  (x_1,x'') \quad  \text{ for all }x_1\in J.
\end{align}
This identity  implies 
that 
\begin{align*}2^{-j}\|M_j[\Phi;\fa]f|\|_{L^p(\bbR^d)}&\le  |J|^{-1/p} \|\cM_j[\phi;a] (\widetilde \beta\otimes f)\|_{L^p(J\times \bbR^{d})}
\\
&\le  (3/2-2/3)^{1/p} \|\cM_j[\phi;a] (\widetilde \beta\otimes f)\|_{L^p(\bbR^{d+1})}
\\&
\le (6/5)^{1/p} \|\cM_j[\phi;a]\|_{L^p(\R^{d+1})\to L^p(\R^{d+1})} \|\widetilde \beta\|_{L^p(\bbR)} \|f\|_{L^p(\bbR^d)}
\end{align*}
which implies the assertion.
\end{proof}


\section{Proof of Proposition 
\ref{frequency decomposition prop}:
$L^2$ bounds}\label{L2 section}





In this section we apply the maximal function results in \S\ref{two-parameter-sect} to deduce favourable $L^2$ bounds which will feature in the proof of Proposition \ref{frequency decomposition prop}.


\begin{proposition}\label{L2 proposition} For all $m <0$,  $k \geq -4$, $(k,\ell) \in \fP$ and $j \geq - e(k,\ell)/3$, the following bounds hold, where in each inequality $I$ denotes an interval of length $|I|\sim 1$ containing the $t$-support of the amplitude.
\begin{flushleft}
\begin{tabular}{rll}
\emph{i)}     & $\displaystyle \big\|\sup_{t \in I } |A_j[\Phi_t;  \fb_t]|\big\|_{L^{2}(\R^2) \to L^{2}(\R^2)} \lesssim
    (j\vee1) 2^{-2k/3}\|\fb\|_{C^N}$ &   if $[\Phi;\fb] \in \mathfrak{B}_{\mathrm{Rot}}^{k}$; \\
\emph{ii)}     & $\displaystyle\big\|\sup_{t \in I } |A_j[\Phi_t;  \fa_t]|\big\|_{L^{2}(\R^2) \to L^{2}(\R^2)} \lesssim  2^{-e(k,\ell)/3}\|\fa\|_{C^N}$ &  if $[\Phi;\fa] \in \mathfrak{A}_{\mathrm{Rot}}^{k,\ell}$; \\
\emph{iii)}     & $\displaystyle\big\|\sup_{t \in I } |A_j[\Phi_t;  \fc_t]|\big\|_{L^{2}(\R^2) \to L^{2}(\R^2)} \lesssim  2^{-2k/3}\|\fc\|_{C^N}$ &  if $[\Phi;\fc] \in \mathfrak{C}_{\mathrm{Rot}}^{k}$; \\
\emph{iv)}     & $\displaystyle\big\|\sup_{t \in I } |A_j[\Phi_t;  \fa_t]|\big\|_{L^{2}(\R^2) \to L^{2}(\R^2)} \lesssim  \|\fa\|_{C^N}$ &  if $[\Phi;\fa] \in \mathfrak{A}_{\mathrm{Rot}}^{0}$; \\
\emph{v)}     & $\displaystyle\big\|\sup_{t \in I } |A_j[\Phi_t;  \fa_t]|\big\|_{L^{2}(\R^2) \to L^{2}(\R^2)} \lesssim  2^{m/2}\sup_{x_2}\|\fa\|_{C^N_{x_1,z,t}}$,  &  if $[\Phi;\fa] \in \mathfrak{A}_{\mathrm{Rot}}^{m}$.
\end{tabular}
\end{flushleft}
\end{proposition}

As in Section \ref{decomposition section}, the cases i), iii), iv) and v) are understood to hold for $\ell = 2k$, with $k=0$ in the cases iv) and v).

The proof of Proposition~\ref{L2 proposition} is presented in what follows. Observe that, by the definition of the classes, iii) and iv) are both just special cases of ii). Thus, it will suffice to prove i) , ii) and v) only.

\smallskip

\noindent{\it Remark.} Only rotational curvature considerations are required to establish the above $L^2$ bounds. The cinematic curvature is used in \S\ref{Lp section} to deduce local smoothing estimates in order to obtain summable bounds in the $j$ parameter.

\smallskip

Using Lemma \ref{TvsA-lemma} the estimates in Proposition \ref{L2 proposition} may be deduced from estimates on oscillatory integral operators acting on functions in $\bbR^3$; in particular,
our assumptions on the phase/amplitude pairs allow  direct applications of Propositions \ref{Hoerm-max-prop} and \ref{fold-max-prop} with suitable choices of the parameters $\lambda$ and $\delta_\circ$.

\subsection{\it Proof of Proposition \ref{L2 proposition} (i)}
By Lemma \ref{TvsA-lemma}, it suffices to show that
$$
\| \sup_{t \in I} |T^{2^j}[\phi_t; b_t]| \|_{L^2(\R^3) \to L^2(\R^3)} \lesssim 2^{-j} (j\vee1) 2^{-2k/3} \| \fb \|_{C^N},
$$
where $\phi_t(x;z)=
x_1z_1\Phi(x'',t; z'')$ and $b_t(x;z)= \chi(x_1) x_1 \fb_t (x'';z'')\beta(x_1|z_1|)$.

First we use the fold conditions, inherent in the hypotheses F$_1$)$_k$ and F$_2$)$_k$ in the definition of $\mathfrak{B}_{\mathrm{Rot}}^k$, to place the operator in a normal form. By assumption  b)$_k$, one may assume without loss of generality, decomposing $b_t$ into at most $O(1)$ pieces, that $\mathrm{supp}\,\fb$ is contained in an $\varepsilon_{\circ}$-ball centred at some point $(x_0'';t_0; z_0'')$ with $(x_0'';z_0'') \in$ $\mathcal{Z}_{t_0}$. Here $\mathcal{Z}_{t_0}$ is as defined in \eqref{fold surface}. Fix a pair of $3 \times 3$ matrices $\mathbf{X}$ and $\mathbf{Z}$ satisfying the properties enumerated in property F$_2$)$_k$. Since $|\det\mathbf{X}| \sim |\det \mathbf{Z}| \sim 1$, by a change of variables it suffices to show 
 the $L^2$ bound for the  maximal function $\sup_{|t-t_0| < \varepsilon_\circ} |T^{2^j}[\tilde{\phi}_t; \tilde{b}_t] f(x)|$
in $\R^3$, where 

\begin{equation*}
\tilde{\phi}_t(x;z) :=  \phi_{t}(\mathbf{X}x; \mathbf{Z}z), \quad \tilde{b}_{t}(x;z) :=  b_{t}(\mathbf{X}x; \mathbf{Z}z).
\end{equation*}

Now the assumption $[\Phi;\fb]\in \fB^k_{\text{Rot}}$ implies that
the support of $\tilde b_t$ is contained in a 
$\varepsilon_{\circ}$-ball centred at $(x_{0},t_0;z_{0})=(0,x_0'', t_0; 1, z_0'') \in \R^3 \times \R \times \R^3$;  moreover
we have the following conditions on the derivatives of $\tilde\phi$: 
\begin{subequations}\label{tphi}
\begin{align}
 &|\partial_x^{\alpha}\partial_z^{\beta} \tilde\phi_t(x;z)| \lesssim \begin{cases}
 2^{-4k/3} &  \textrm{if $\alpha_3 \neq 0$,} \\
    2^{2k/3} & \textrm{otherwise},
    \end{cases}
    \label{tphia}
    \\
 &\partial_{x_3 z}^2 \tilde\phi_{t_0}(x_0;z_0) = \partial_{x z_3}^2 \tilde\phi_{t_0}(z_0;z_0) = 0 \text{ and } |\det \partial_{x' z'}^2\tilde \phi_{t_0}(x_0;z_0)| \sim 2^{4k/3},
     \label{tphib}
 \\
  &|\partial_{x_3z_3z_3}^3\tilde\phi_{t_0}(x_0;z_0)|, |\partial_{x_3x_3z_3}^3\tilde\phi_{t_0}(x_0;z_0)| \sim 2^{-4k/3},
      \label{tphic}
\\
&|\partial_x^{\alpha}\partial_z^{\beta} \partial_t^{\gamma} \tilde\phi_t(x;z)|\lesssim 
    2^{-4k/3} \text{  for }\gamma > 0,
        \label{tphid}
\\
&\partial_t \tilde\phi_t(x;z) = \tilde{c}(x,t;z) 2^{-4k/3} \det\partial_{x z}^2 \tilde\phi_t (x;z) \, \text{ for some $\tilde{c} \in C^\infty$ with } \label{tphie}\\
& \text{uniform $C^\infty$ bounds on $\supp \tilde{b}_t$.} \notag
\end{align}
\end{subequations}
The following table shows which conditions for  the class $\fB^k_{\text{rot}}$ of defining functions imply the conditions in 
\eqref{tphi}. 
\medskip

\begin{center}
\begin{tabular}{||r|l||r|l||}
\hline
 \eqref{tphia} & $\Phi_1$)$_{k,2k}$ and F$_2$)$_k$ i) & \eqref{tphid}  & $\Phi_1$)$_{k,2k}$ and F$_2$)$_k$ i)  \\
 \eqref{tphib}  &  F$_2$)$_k$ iii) & \eqref{tphie}  &  $\Phi_2$)$_{k,2k}$ and F$_2$)$_k$ ii) \\
 \eqref{tphic}  &  F$_1$)$_k$ and F$_2$)$_k$ ii) & &  \\
   \hline
\end{tabular}
\end{center}
\medskip

One now  checks that the phase function 
\[ \Psi_t(x;z)=2^{-2k/3}\widetilde \phi_t(x;z)\]
satisfies the assumptions in Proposition \ref{fold-max-prop} with $d=3$ and $\delta_\circ=2^{-2k}$   via (7.1).
If we put $\lambda=2^{j+2k/3}$,  then $\lambda\Psi= 2^j \, \tilde \phi$ and we can apply Proposition \ref{fold-max-prop}  to obtain 
$$
\| \sup_{t \in I} |T^{2^j}[\phi_t; b_t]| \|_{L^2(\R^3) \to L^2(\R^3)} \lesssim \lambda^{-1}  \log (2+\lambda\delta_\circ) \| b \|_{C^N} \lesssim  2^{-j-2k/3} (j\vee1)  \| \fb \|_{C^N},
$$
as desired \qed.

\subsection{\it  Proof of Proposition \ref{L2 proposition} (ii)}
We again use the reduction in \S\ref{reduction to oscillatory integral estimates section} so that it suffices to show
$$
\| \sup_{t \in I} |T^{2^j}[\phi_t; a_t]| \|_{L^2(\R^3) \to L^2(\R^3)} \lesssim 2^{-j}  2^{-e(k,\ell)/3} \| \fa \|_{C^N},
$$
where $\phi_t(x;z)=
x_1z_1\Phi(x'',t; z'')$ and $a_t(x;z)= \chi(x_1) x_1 \fa_t (x'';z'')\beta(x_1|z_1|)$. The condition $[\Phi;\fa]\in \fA^{k,\ell}_{\text{Rot}}$ implies that the phase function
$\phi_t(x;z)=x_1z_1\Phi_t(x'';z'')$ satisfies 
the inequalities
\begin{align}\label{phi-rot-a}
 &|\partial_x^{\alpha}\partial_z^{\beta} \phi_t(x;z)| \lesssim 
 \begin{cases}     2^{-2e(k,\ell)/3} &  \textrm{if $\alpha_3 \neq 0$} \\
    2^{e(k,\ell)/3} & \textrm{otherwise}
    \end{cases}\\
 \label{phi-rot-b} &|\det\partial_{x z}^2 \phi_t (x;z)| \sim 1
\\ \label{phi-rot-c}
&|\partial_x^{\alpha}\partial_z^{\beta} \partial_t^{\gamma} \phi_t(x;z)|\lesssim 
    2^{-2e(k,\ell)/3} \quad  \textrm{  for $\gamma > 0$.}
    \end{align}
These estimates are understood to hold on $\mathrm{supp}\,a_t$ 
(which has diameter $\lc 1$) for all $\alpha, \beta \in \N_0^3$, $\gamma \in \N$ with implicit constants depending on the multiindices.
One checks that \eqref{phi-rot-a} and \eqref{phi-rot-c} are implied by $\Phi_1)_{k,\ell}$ in the definition of $\fA^{k,\ell}$ while \eqref{phi-rot-b} is implied by the additional rotational curvature condition in Definition \ref{fA-k-ell-Rot}.

We can now verify  that the phase function 
\[ \Psi_t(x;z)=2^{-e(k,\ell)/3} \phi_t(x;z)\]
satisfies the assumptions in Proposition \ref{Hoerm-max-prop} with $d=3$ and $\delta_\circ=2^{-e(k,\ell)}$.
If we put $\lambda= 2^{j+e(k,\ell)/3}$,  then $\lambda\Psi= 2^j \phi$ and by  Proposition \ref{Hoerm-max-prop}  we get 
$$
\| \sup_{t \in I} |T^{2^j}[\phi_t; a_t]| \|_{L^2(\R^3) \to L^2(\R^3)} \lesssim \lambda^{-1} \| a \|_{C^N} \lesssim  2^{-j-e(k,\ell)/3} \| \fa \|_{C^N},
$$
as desired. \qed.
%
\subsection{\it Proof of Proposition~\ref{L2 proposition} (v)}
Again, by Lemma \ref{TvsA-lemma}, it suffices to show that
\begin{equation}\label{claim-for-maxTla}
\| \sup_{t \in I} |T^{2^j}[\Psi_t; a_t]| \|_{L^2(\R^3) \to L^2(\R^3)} \lesssim 2^{-j} 2^{m/2} \sup_{x_3} \| a \|_{C^N_{x_1,x_2,z,t}},
\end{equation}
where $\Psi(x;t;z)=x_1z_1 \Phi(x'',t;z'')$ and $a_t(x;z)= \chi(x_1) x_1 \fa_t (x'';z'')\beta(x_1|z_1|)$. By the condition $[\Phi; \fa] \in \fA^{m}_{\text{Rot}}$, the diameter of the support of $a$ is $O(1)$ and moreover the following conditions hold (see
Definitions \ref{fA-m-def} and \ref{fA-m-rot-def}). First,
there exists an interval $I_m$  of length $\lc 2^{m}$ so that
$a(x;t;z)=0$ when $x_3\notin I_m$. 
Next, if $\Psi_{x_3}^{\star}(x_1,x_2,t; z):= x_1z_1\Phi(x'',t;z'') $ then
$\Psi_{x_3}^\star$  satisfies
\begin{subequations}
\begin{align}
&|\partial_x^\alpha\partial_z^\beta\partial_t^\gamma \Psi_{x_3}^\star(x_1,x_2,t;z)|\lc \begin{cases}
&2^{-2m}  \text{ if $\alpha_3\neq 0$}\\&1 \qquad \! \text{ otherwise }
\end{cases}, 
\label{Psi-a}
\\
&|\det \partial^2_{(x_1,x_2,t),(z_1,z_2,z_3)}\Psi_{x_3}^\star | \sim 1.\label{Psi-b}
\end{align}
\end{subequations}
In what follows we will freeze $x_3$, so the derivatives with respect to $x_3$ in \eqref{Psi-a} will be irrelevant for our purposes. 

To establish 
\eqref{claim-for-maxTla}
we show that if 
\begin{equation*}
 S^{2^j}_{x_3} f(x_1,x_2,t) \equiv T^{2^j}[\Psi_{x_3}^\star; a_{x_3}^\star]f(x_1,x_2,t) \equiv T^{2^j} [\Psi_t; a_t] f (x_1,x_2,x_3),   
\end{equation*} 
where $a_{x_3}^\star(x_1,x_2,t;z)=a_t(x;z)$,  then we have, for all $x_3\in I_m$, 
\begin{align} \notag
\Big (\int_{\R^3} |S_{x_3}^{2^j} f(x_1,x_2,t)|^2 \ud x_1\,\ud x_2\,\ud t\Big)^{1/2} + 2^{-j} &\Big(
\int |\partial_t S_{x_3}^{2^j}  f(x_1,x_2,t)|^2 \, \ud x_1 \, \ud x_2 \, \ud t\Big)^{1/2} \\
& \lc 2^{-3j/2 } \sup_{x_3} \| a \|_{C^N_{x_1,x_2,z,t}} \|f\|_2.      \label{fixed-x3-est}
\end{align}
Indeed, note that $\partial_t S_{x_3}^{2^j} f(x_1,x_2,t) = T^{2^j}[\Psi_{x_3}^\star; d_{x_3}^\star] f (x_1,x_2,t)$, where 
\begin{equation*}
    d_{x_3}^\star:=(i 2^j \partial_t  \Psi_{x_3}^\star) \,a + \partial_t a_{x_3}^\star,
\end{equation*} 
and, in view of \eqref{Psi-a} and \eqref{Psi-b}, the estimate \eqref{fixed-x3-est} is now an immediate consequence of the oscillatory integral estimate in Proposition \ref{2par-Hoerprop} with $\delta_\circ=1$, which holds uniformly in $x_3 \in I_m$. Note that, in this case, our application of Proposition \ref{2par-Hoerprop} corresponds to the classical Hörmander $L^2$-estimate for oscillatory integrals \cite{Hormander1973}. Integrating 
the square of the left-hand side of 
\eqref{fixed-x3-est}
over  $x_3\in I_m$ and using   $|I_m|\lc 2^{m}$, we get
\begin{align*}
    \Big (\int_{I_m} \int_{\R^3} |T^{2^j}[\Psi_t;a_t] f(x) |^2+
2^{-2j}  |\partial_t T^{2^j}  [\Psi_t;a_t] & f(x)|^2   \, \ud x_1 \, \ud x_2 \,  \ud t \, \ud x_3\Big)^{1/2} \\
& \lc 2^{-3j/2}  2^{m/2} \sup_{x_3} \| a \|_{C^N_{x_1,x_2,z,t}} \|f\|_2.
\end{align*}
By the Sobolev inequality \eqref{sobolev lemma} and Fubini's theorem, the desired estimate \eqref{claim-for-maxTla} immediately follows. \qed





\section{
Proof of Proposition 
\ref{frequency decomposition prop}:
$L^p$ theory}\label{Lp section}




This section deals with the remainder of the proof of Proposition~\ref{frequency decomposition prop}. Local space-time $L^p$  estimates are used to establish $L^p$ bounds with favourable $j$ dependence when $p > 2$. These bounds can be combined with the $L^2$ estimates from Proposition~\ref{L2 proposition} and  $L^\infty$ estimates to yield the desired results.

\subsection{ $L^p$ bounds}

It is first noted that the $L^2$ bounds of the previous section imply $L^p$ estimates via interpolation with straightforward $L^{\infty}$ bounds. 

\begin{corollary}\label{Lp corollary} For all $m <0$, $(k,\ell) \in \fP$, $j \geq -e(k,\ell)/3$ and $2 \leq p \leq \infty$, there exists $N \in \N$ such that the following bounds hold. In each inequality, $I$ denotes an interval of length $\sim 1$ containing the $t$-support of the amplitude.
\begin{flushleft}
\begin{tabular}{rll}
\emph{i)}     & $\displaystyle \big\|\sup_{t \in I} |A_j[\Phi_t;  \fb_t]|\big\|_{L^{p}(\R^2) \to L^{p}(\R^2)} \lesssim
    (j\vee 1)^{2/p} 2^{-2k/3}\|\fb\|_{C^N}$ &   if $[\Phi;\fb] \in \mathfrak{B}_{\mathrm{Rot}}^{k}$, \\
\emph{ii)}     & $\displaystyle\big\|\sup_{t \in I} |A_j[\Phi_t;  \fa_t]|\big\|_{L^{p}(\R^2) \to L^{p}(\R^2)} \lesssim 
   2^{-e(k,\ell)/3}\|\fa\|_{C^N}$ &  if $[\Phi;\fa] \in \mathfrak{A}_{\mathrm{Rot}}^{k,\ell}$,  \\ 
\emph{iii)}  & $\displaystyle\big\|\sup_{t \in I} |A_j[\Phi_t; \fc_t]|\big\|_{L^{p}(\R^2) \to L^{p}(\R^2)} \lesssim 2^{-2k/3}\|\fc\|_{C^N}$ & if $[\Phi;\fa] \in \mathfrak{C}_{\mathrm{Rot}}^k$, \\
\emph{iv)}  & $\displaystyle\big\|\sup_{t \in I} |A_j[\Phi_t; \fa_t]|\big\|_{L^{p}(\R^2) \to L^{p}(\R^2)} \lesssim \|\fa\|_{C^N}$ & if $[\Phi;\fa] \in \mathfrak{A}^0_{\mathrm{Rot}}$,\\
\emph{v)}  & $\displaystyle\big\|\sup_{t \in I} |A_j[\Phi_t; \fa_t]|\big\|_{L^{p}(\R^2) \to L^{p}(\R^2)} \lesssim 2^{m/p}\sup_{x_2}\|\fa\|_{C^N_{x_1,z,t}}$, & if $[\Phi;\fa] \in \mathfrak{A}_{\mathrm{Rot}}^m$.
\end{tabular}
\end{flushleft}
\end{corollary}

\noindent{\it Remark.} The estimates from Corollary \ref{Lp corollary} are not summable in the $j$ parameter, so alone they do not imply Proposition~\ref{frequency decomposition prop}. However, i), ii) and iii) have better $k$ dependence than what is required in Proposition~\ref{frequency decomposition prop} (by a factor of $2^{(1-\frac{2}{p})k - k \varepsilon_p}$) and, similarly, v) has a better $m$ dependence (by a factor of $2^{m/p-m\varepsilon_p}$). This observation is used below to mitigate losses in $k$ and $m$ in Proposition~\ref{Lp proposition}.

\begin{proof}[Proof of Corollary~\ref{Lp corollary}] We will only consider i) since the proofs of the remaining cases are similar. For $p = 2$ the desired bound is precisely Proposition~\ref{L2 proposition} i). By interpolation, it suffices to verify the bound for $p = \infty$. 

Let $[\Phi;\fb] \in \mathfrak{B}_{\mathrm{Rot}}^{k}$ and recall from \eqref{frequence localised definition} that
\begin{equation*}
    A_j[\Phi_t;\fb_t]f(x) = \int_{\R^2} f(z) \fb_t(x;z) 2^j\check{\beta}\big(2^j\Phi_t(x;z)\big)\ud z.
\end{equation*}
Further recall that $\Phi_t$ satisfies Definition~\ref{190723def5.7} and, in particular, the condition $\Phi_1)_{k, 2k}$ as stated in Definition~\ref{190723def5.2}. Thus, on the support of $\fb_t$ we have
\begin{equation*}
|\partial_z \Phi_t(x; z)| \gtrsim 2^{2k/3}
\end{equation*}
and so the desired $L^{\infty}$ estimate follows. 
\end{proof}

The following proposition provides the crucial $j$ summability for $j>0$.

\begin{proposition}\label{Lp proposition} There exist $N, M \in \N$  and $\varepsilon_{\circ} > 0$ such that for all $(k, \ell ) \in \fP$, the inequality
\begin{equation}\label{Lp inequality}
    \big\|\sup_{t \in I} |A_j[\Phi_t;  \fa_t]|\big\|_{L^6(\R^2) \to L^{6}(\R^2)} \lesssim 
   2^{M k} 2^{-j \varepsilon_{\circ} }\|\fa\|_{C^N}
\end{equation}
holds if $[\Phi;\fa]$ belongs to any one of the following classes:

\begin{flushleft}
\begin{tabular}{rlrl}
    \emph{i)} & $\mathfrak{B}_{\mathrm{Cin}}^{k}$, & \emph{iv)} & $\mathfrak{A}_{\mathrm{Cin}}^{0}$ taking $k = 0$ in \eqref{Lp inequality},\\[2pt]
    \emph{ii)} & $\mathfrak{A}_{\mathrm{Cin}}^{k,\ell}$, & \emph{v)} &   $\mathfrak{A}_{\mathrm{Cin}}^{m}$, $m <0$, taking $k = -m$ in \eqref{Lp inequality}.\\[2pt]
    \emph{iii)} & $\mathfrak{C}_{\mathrm{Cin}}^k \cap \mathfrak{C}_{\mathrm{Rot}}^{k}$. &  & 
\end{tabular}
\end{flushleft}
In \eqref{Lp inequality}, $I$ denotes an interval of length $\sim 1$ containing the $t$-support of $\fa$.
\end{proposition}

\noindent{\it Remark.} The exponent $p=6$ does not play a significant r\^ole and is used merely for convenience (one could equally work with other $p$ values). See the comments after Theorem~\ref{local smoothing theorem} below.

Assuming this result, Proposition~\ref{frequency decomposition prop} easily follows by interpolation with the estimates in Corollary~\ref{Lp corollary}.

\begin{proof}[Proof of Proposition~\ref{frequency decomposition prop} assuming Proposition~\ref{Lp proposition} holds] 
For $-e(k,\ell)/3\le j\le 0$ the asserted bounds are an immediate consequence of Corollary~\ref{Lp corollary}. For $j>0$ it suffices,
by  Corollary~\ref{Lp corollary},  to show each of the five estimates in Proposition~\ref{frequency decomposition prop} hold \textit{for a single value $2 < p_{\ast} < \infty$}: indeed, once this is established, one may interpolate the $p_{
\ast}$ estimates with the $p = 2$ and $p = \infty$ cases of Corollary~\ref{Lp corollary} to obtain Proposition~\ref{frequency decomposition prop} for all $2 < p < \infty$. 

We interpolate the inequalities from Proposition~\ref{Lp proposition}
with the corresponding $L^{\infty}$ estimates of Corollary~\ref{Lp corollary}, or the $L^2$ estimate in case v). In the case v), note that $\| \fa \|_{C^N} \lesssim 2^{-O(N)m}$, which is harmless in view of the $2^{-Mm}$ loss in \eqref{Lp inequality}. Therefore, it follows that Proposition~\ref{frequency decomposition prop} holds for some $p_{\ast}$ in the range $6 < p_{\ast} < \infty$ for the cases i) to iv), or in the range $2 <p_{\ast}<6$ for case v), concluding the proof. 
\end{proof}

It  remains  to  prove  Proposition~\ref{Lp proposition}. By the definition of the classes, Proposition~\ref{Lp proposition} i) and iv) automatically follow from ii). Furthermore, for the purposes of the argument, the cases ii) and v) are essentially simplified variants of case iii). In particular, the main difficulties occur in the proof of iii).




\subsection{Reduction to Fourier integral estimates}\label{sec:reduction to 2 Fourier variables} Following the strategy of~\cite{Mockenhaupt1992, Mockenhaupt1993}, Proposition~\ref{Lp proposition} is derived from local smoothing estimates for Fourier integral operators. In order to invoke the local smoothing inequalities, it desirable to express $A_j[\Phi_t;\fa_t]$ as a Fourier integral operator  with two Fourier variables. That such a representation is possible is a standard result, referred to as the \textit{equivalence of phase theorem} (see, for instance,~\cite{Hormander1971} or~\cite{Duistermaat2011}). Since here, however, the estimates are required to be quantitative, at least in some weak sense, basic stationary phase techniques are instead applied to obtain an explicit two Fourier variable representation of the frequency localised averaging operators.

\subsubsection*{Fourier integral representation} Fix a smooth family of defining pairs $[\Phi; \fa]$ and, for the purposes of this subsection, assume that 
\begin{equation*}
|\kappa(\Phi)(\vec{x};z)|, \:\: |\mathrm{Proj}(\Phi)(\vec{x};z)|,\:\: |\mathrm{Cin}(\Phi)(\vec{x};z)| \geq \varepsilon_{\mathrm{Cin}} > 0 \qquad \textrm{for all $(\vec{x};z) \in \supp \fa$;}
\end{equation*}
moreover, assume that upper bounds for the derivatives of $\Phi$ depend polynomially on $2^{k}$, where $k$ is as in Proposition \ref{Lp proposition}. 
Here $\vec{x} = (x,t) \in \R^2 \times \R$. Owing to the nature of the estimates in Proposition~\ref{Lp proposition}, here one does not need to be very precise about dependencies involving various derivatives of $\Phi$ and $\fa$ and the bounds on the curvatures (as opposed to the situation in \S\ref{L2 section}). For instance, the constant $\varepsilon_{\mathrm{Cin}}$ may depend on the parameters $k, \ell$ and $m$. In what follows, we will not determine the precise dependence of our estimates on these parameters but will only be concerned with showing that it is not worse than $2^{Mk}$ for some large constant $M \geq 1$.

Given a phase/amplitude pair $[\Phi_t; \fa_t]$, from \eqref{frequence localised definition} and the Fourier inversion formula,
\begin{equation*}
    A_j[\Phi_t; \fa_t]f(x) = \int_{\hat{\R}^2} \widetilde{K}^{2^j}(\vec{x};\xi) \hat{f}(\xi)\,\ud \xi
\end{equation*}
where
\begin{equation}\label{oscillatory kernel}
    \widetilde{K}^{\lambda}(\vec{x};\xi) := \frac{1}{(2\pi)^3}\int_{\R^2} \int_{\R} e^{i(\theta \Phi(\vec{x};z) + \langle z, \xi \rangle)}\fa(\vec{x};z)\beta(\theta/\lambda)\,\ud \theta \, \ud z.
\end{equation}
This function can be  analysed via stationary phase arguments. The critical points $(z_{cr},\theta_{cr})$ of the phase function 
\begin{equation} \label{eq:criticalPsi}(z,\theta)\mapsto 
\Psi(z,\theta; \vec x;\xi)= \theta\Phi(\vec x; z)+\inn z \xi
\end{equation}
satisfy $\Phi(\vec{x};z_{cr}) = 0$ and $\theta_{cr}\partial_z \Phi(\vec{x};z_{cr}) + \xi = 0$. 
The former condition implies that $z_{cr} \in \Sigma_{\vec{x}} =\{ z \in \R^2 : \Phi_t(x; z) =0\}$ while the latter implies that the normal to $\Sigma_{\vec{x}}$ at $z_{cr}$ is parallel to $\pm\xi$. 
We also have 
$|\det \partial^2_{(z,\theta)} \Psi|= |\theta |^d\kappa(\Phi) $ so that the critical points are nondegenerate.

Let $C_{\circ} \geq 1$ satisfy 
\begin{equation*}
 (C_{\circ}/10)^{-1} \leq |\partial_z \Phi(\vec{x};z)| \leq C_{\circ}/10 \qquad \text{ for all } (\vec{x}; z) \in \mathrm{supp}\,\fa.
\end{equation*}
There are no critical points for the phase if $|\xi| \geq 4 C_{\circ} \lambda$ or $|\xi| \leq \lambda/4C_{\circ}$. Thus, by repeated integration-by-parts
\begin{equation}\label{kernel localisation}
    \widetilde{K}^{\lambda}(\vec{x}; \xi) = \widetilde{K}^{\lambda}(\vec{x};\xi)\tilde{\beta}(|\xi|/\lambda) + E^{\lambda}(\vec{x};\xi)
\end{equation}
where $\tilde{\beta}(r) := \eta(C_{\circ}^{-1}r) - \eta(C_{\circ}r)$ and the error $E^{\lambda}$ satisfies 
\begin{equation}\label{bounds error 1}
    |\partial_{\xi}^{\alpha} [e^{- i \langle x , \xi \rangle} E^{\lambda}(\vec{x};\xi)]| \lesssim C_\circ^N \lambda^{-N/2}(1 + |\xi|)^{-N/2} \quad \textrm{for all $|\alpha| \leq N$,}
\end{equation}
with implicit bounds depending on $\| \fa \|_{C^N}$. Note that the value of $C_\circ$ will generally depend on $k$ or $m$ for the classes considered in Proposition~\ref{Lp proposition}, but this dependence is admissible in our forthcoming analysis.
\begin{key example} 
Let $[\Phi; \mathfrak{a}] \in \mathfrak{C}^k_{\mathrm{Cin}} \cap \mathfrak{C}^k_{\mathrm{Rot}}$. The condition $\Phi_1$)$_{k,2k}$ ensures that $|\partial_z \Phi (\vec{x}; z)| \sim 2^{2k/3}$ and so $C_\circ \sim 2^{2k/3}$ in this case. 
\end{key example}

We further analyse $\widetilde{K}^\la(\vec x;\xi)$ for $C_\circ^{-1}/4<|\xi|< 4C_\circ |\xi|$. Decompose 
\[\widetilde{K}^{\lambda}(\vec{x};\xi)\tilde{\beta}(|\xi|/\lambda) = \sum_{i\in \cJ} \widetilde{K}^\la_i(\vec x;\xi) \] where the cardinality of the index set $\cJ$ is 
polynomial in $\eps_{\mathrm Cin}$ and $C_\circ$ and each  $\widetilde{K}^{\lambda}_i$ is of the form 
\begin{equation}\label{oscillatory kernel 2}
    \mathcal {K}^{\lambda}(\vec{x};\xi) := \chi(\vec x; \xi) \frac{1}{(2\pi)^3}\int_{\R^2} \int_{\R} e^{i(\theta \Phi(\vec{x};z) + \langle z, \xi \rangle)}\fa(\vec{x};z)\beta(\theta/\lambda)\,\ud \theta \, \ud z,
\end{equation}
with $\chi$, $\beta$, $\fa$  supported on sets of diameter $\eps$ with $  \eps_{\mathrm {Cin}}^C\lc \eps\lc\eps_{\mathrm {Cin}}$.  

It suffices to consider the kernel \eqref{oscillatory kernel 2}.
 If $\Phi(\vec x;z) $ does not vanish in the neighborhood of the support then the integral represents a smooth function with derivative bounds polynomial in  $\eps_{\mathrm {Cin}}^{-1} $, $C_\circ$. Otherwise 
 we may use the method of stationary phase, using that the critical points of 
 \eqref{eq:criticalPsi} are nondegenerate.
 In a neighborhood of the support of the symbol, we can then solve the equation  $\nabla_{\theta,z} \Psi( z,\theta; \vec x,\xi) =0$ in $(z,\theta)$ with  $z=\nu(\vec x;\xi)$, $\theta=\Theta(\vec x;\xi) $  denoting the solutions; moreover  $\nu$ is homogeneous of degree $0$ in $\xi$ and $\Theta$ is homogeneous of degree $1$ in $\xi$. Hence
\begin{equation}\label{critical functions}
    \left\{ \begin{array}{rl}
         \Phi(\vec{x};\nu(\vec{x};\xi)) & = 0  \\
         \Theta(\vec{x};\xi)\partial_z \Phi(\vec{x};\nu(\vec{x};\xi)) + \xi & = 0  
    \end{array} \right. .
\end{equation}
Furthermore, if 
\begin{equation*}
    \varphi(\vec{x}; \xi) := \Psi(\nu(\vec{x},\xi), \Theta(\vec{x}; \xi); \vec{x}; \xi),
\end{equation*}
 then \eqref{critical functions} implies that
\begin{equation}\label{n Fourier phase}
   \varphi(\vec{x}; \xi)  = \langle \nu(\vec{x};\xi), \xi \rangle.
\end{equation}
By rescaling and applying the method of stationary phase~\cite[Theorem 7.7.5]{HormanderBook1}, one deduces that 

\begin{equation}\label{FT kernel 2}
    \mathcal{K}^{\lambda}(\vec{x};\xi) = e^{i\lambda \varphi(\vec{x};\xi/\lambda)} \frac{\ba(\vec{x};\xi/\lambda)}{(1 + |\xi|^2)^{1/4}} + E(\vec{x};\xi/\lambda)
\end{equation}
where, for some $M_N>0$:
\begin{itemize}
    \item The symbol $\ba$ is supported in $\{C_\circ^{-1}\lc |\xi|  \lc C_\circ \}$ and satisfies 
\begin{equation*}
 |\partial_{\vec{x}}^\alpha \partial_\xi^\beta \ba (\vec{x}; \xi) | \lesssim ( \varepsilon_{\mathrm{Cin}}^{-1} +  C_\circ + \| \Phi \|_{C^{3N+2}} + \| \fa \|_{C^{3N}} )^{M_N} 
\end{equation*}
 and all $(\alpha,\beta) \in \mathbb{N}_0^3 \times \mathbb{N}_0^2$ with $|\alpha|, |\beta| \leq N$.
 \item The error term $E$ is rapidly decaying in the sense that
\begin{equation}\label{error bounds 2}
| \partial_\xi^{\alpha} [ e^{- i \langle x, \xi \rangle} E(\vec{x}; \xi/\lambda)]| \lesssim ( \varepsilon_{\mathrm{Cin}}^{-1} + \| \Phi \|_{C^{3N+2}} + \| \fa \|_{C^{3N}} )^{M_N} \lambda^{-N}
\end{equation}
for any $\alpha \in \N_0^2$ with $|\alpha| \leq N$. 
\end{itemize}
 One is therefore led to consider operators belonging to the following class.

\begin{definition} An FIO pair $[\varphi; \ba]$ consists of a pair of functions $\varphi, \ba \in C^{\infty}(\R^3 \times \hat{\R}^2)$ with $\ba$ supported in a compact set of diameter 1. For any such pair $[\varphi; \ba]$ and $\mu \in \R$ define  Fourier integral operators of order $\mu$ by
\begin{equation}\label{FIO}
    \mathcal{F}_{\mu}^{\lambda}[\varphi; \ba]f(\vec{x}\,) := \int_{\hat{\R}^2} e^{i \lambda\varphi(\vec{x};\xi/\lambda)} \frac{\ba(\vec{x};\xi/\lambda)}{(1 + |\xi|^2)^{-\mu/2}} \beta(|\xi|/\lambda) \hat{f}(\xi)\,\ud \xi \qquad \textrm{for $\lambda \geq 1$}. 
\end{equation}
\end{definition}

\subsubsection*{Local smoothing estimates} Under certain `geometric' hypotheses on the phase, $L^p_x \to L^p_{x,t}$ estimates are known for the operators \eqref{FIO} with good $\lambda$ decay (indeed, the best possible decay (up to $\varepsilon$ losses) for $6 \leq p \leq \infty$). Here the relevant hypotheses are stated in a weakly quantitative form. In what follows we use the notation
$\bigwedge_{k=1}^2\vec v_k$ for the standard vector product $\vec v_1\times \vec v_2$ for vectors in $\bbR^3$.

\begin{definition}\label{FIO class} For $R \geq 1$ let $\mathbb{A}(R)$ denote the class of all $[\varphi; \ba]$ satisfying 
\begin{enumerate}
    \item[H0)] $|\partial^{\alpha}_{\vec{x}} \partial^{\beta}_{\xi} \varphi(\vec{x};\xi)| \lesssim R \quad$ for $|\alpha| \leq N$ and $0< |\beta| \leq N$, 
    \item[H1)] $\displaystyle \Big|\bigwedge_{k=1}^2 \partial_{\xi_k} \partial_{\vec{x}}\varphi(\vec{x};\xi)\Big| \geq R^{-1}$,
    \item[H2)] $\displaystyle \max_{1 \leq i, j \leq 2}\Big|\Big\langle \partial_{\xi_i \xi_j}^2 \partial_{\vec{x}}\,\varphi(\vec{x};\xi)\,,\, \bigwedge_{k=1}^2 \partial_{\xi_k} \partial_{\vec{x}}\varphi(\vec{x};\xi)  \Big\rangle \Big| \geq R^{-1}$
\end{enumerate}
for all $(\vec{x};\xi) \in \mathrm{supp}\,\ba$. 
\end{definition}

The following theorem is the key ingredient in the proof of Proposition~\ref{Lp proposition}.

\begin{theorem}[\cite{BHS}] 
\label{local smoothing theorem} There exist $N, M \in \N$ such that
\begin{equation*}
    \|\mathcal{F}_{\mu}^{\lambda}[\varphi; \ba]\|_{L^6(\R^2) \to L^6(\R^3)} \lesssim_{\varepsilon} R^M \lambda^{1/6 + \mu + \varepsilon} \|\ba\|_{C^N} \qquad \textrm{for all $[\varphi; \ba] \in \mathbb{A}(R)$.}
\end{equation*}
\end{theorem}

This weakly quantitative statement is not explicit in~\cite{BHS}
or the corresponding survey
\cite{BHS2}
 but it may  be extracted from the proof. It is remarked that Theorem~\ref{local smoothing theorem} is more than enough for the purposes of this article and, indeed, \textit{any} non-trivial local smoothing estimate 
 (that is,  a gain of an epsilon derivative over the fixed term estimate) would suffice.  Thus one  could equally appeal to the older results of~\cite{Mockenhaupt1993} (see also the related work~\cite[Chapter 3]{KungThesis}, or the more recent work \cite{GMY}).

\subsubsection*{Relating the phase functions} In order to apply Theorem~\ref{local smoothing theorem} we analyse the hypotheses H0), H1) and H2) for the specific case of the phase $\varphi$ arising from the averaging operators $A[\Phi_t;\fa_t]f$. 

Let $\varphi$ be of the form \eqref{n Fourier phase}, induced by some defining function $\Phi$. Implicit differentiation of \eqref{critical functions} yields
\begin{align}
\label{lambda nu derivatives}
    \begin{bmatrix}
    \partial_{\xi}\Theta \\[2pt]
    \partial_{\xi}\nu
    \end{bmatrix} &=
    -\begin{bmatrix}
    0 & (\partial_{z}\Phi)^{\top} \\[2pt]
     \partial_{z}\Phi  & \Theta \partial_{z z}^2\Phi
    \end{bmatrix}^{-1}
    \begin{bmatrix}
    0 \\[2pt]
    \mathrm{Id}_2
    \end{bmatrix}, \\
    \label{lambda nu derivatives 2}
    \begin{bmatrix}
    \partial_{\vec{x}}\Theta \\[2pt]
    \partial_{\vec{x}}\nu
    \end{bmatrix} &=
    -\begin{bmatrix}
    0 & (\partial_{z}\Phi)^{\top} \\[2pt]
     \partial_{z}\Phi  & \Theta \partial_{z z}^2\Phi
    \end{bmatrix}^{-1}
    \begin{bmatrix}
    \partial_{\vec{x}}\,\Phi &
    \Theta\partial_{\vec{x} z}^2 \Phi 
    \end{bmatrix}^{\top},
\end{align}
where the right-hand matrices are evaluated at $z = \nu(\vec{x};\xi)$. In particular, \eqref{lambda nu derivatives} implies that $\partial_{\xi_2}\nu_1 = \partial_{\xi_1}\nu_2$ and combining this with Euler's homogeneity relation $\varphi(\vec{x};\xi) = \langle \partial_\xi \varphi(\vec{x};\xi), \xi \rangle$ yields
\begin{equation}\label{xi derivatives}
    \partial_{\xi} \varphi = \nu. 
\end{equation}
Consequently, one can check that if $(\alpha, \beta) \in \N_0^3 \times \N^2_0$ satisfies $|\alpha|$,  $|\beta| \leq N$, then
\begin{equation}\label{H0}
    |\partial_{\vec{x}}^{\alpha}\partial_{\xi}^{\beta} \varphi(\vec{x};\xi) | \lesssim_N (\| \Phi \|_{C^N} + \varepsilon^{-1}_{\mathrm{Cin}} )^{M_N}.
\end{equation}
for a certain $M_N>0$. 

Furthermore, \eqref{lambda nu derivatives 2} and \eqref{xi derivatives} also imply that
\begin{equation}\label{H1}
    \displaystyle \big|\bigwedge_{j=1}^2 \partial_{\xi_j} \partial_{\vec{x}}\varphi(\vec{x};\xi)\big| \gc \frac{1}{|\partial_{\vec{x}} \Theta(\vec{x};\xi)|}\det\begin{bmatrix}
    \partial_{\vec{x}}\Theta(\vec{x};\xi) \\[2pt]
    \partial_{\vec{x}}\nu(\vec{x};\xi)
    \end{bmatrix} \geq \mathrm{Proj}(\Phi)(\vec{x};\xi)\cdot \|\Phi\|_{C^2}^{-3}.
\end{equation}
These inequalities allow one to deduce H0) and H1); the condition H2) requires a slightly more involved analysis. 

Letting $\sigma_1 := 2$ and $\sigma_2 := 1$, the identities in \eqref{lambda nu derivatives} and \eqref{lambda nu derivatives 2} give
\begin{align}\label{1st order nu}
    \partial_{\xi_j} \nu_i &= \frac{(-1)^{i+j+1} \xi_{\sigma_i}\xi_{\sigma_j}}{\Theta \big(\Theta^2\kappa(\Phi)\big)}, \\
    \label{1st order nu 2}
    \partial_{\vec{x}}\nu_i &= \frac{(-1)^i}{\kappa(\Phi)} \Bigg( 
    \det\begin{bmatrix}
    \partial_{z_1}\Phi & \partial^2_{z_1 z_{\sigma_i}} \Phi \\
    \partial_{z_2}\Phi & \partial^2_{z_2 z_{\sigma_i}} \Phi
    \end{bmatrix} \mathbf{T}^1 - (\partial_{z_{\sigma_i}}\Phi)
    \mathbf{T}^2 \Bigg),
\end{align}
where $\kappa(\Phi)$ is as defined in \eqref{plane curvature definition} and the $\mathbf{T}^i$ are the tangent vector fields from \eqref{tangent vector fields}. Recalling \eqref{xi derivatives}, the condition H2) for the phase function \eqref{n Fourier phase} involves mixed second order derivatives of $\nu$; by \eqref{1st order nu}, computing these derivatives boils down to differentiating $ \big(\Theta^2\kappa(\Phi)\big)^{-1}\Theta^{-1}$ with respect to $\vec{x}$. Recalling the definition of $\Theta$ and $\nu$ from \eqref{critical functions} and the identities of \eqref{lambda nu derivatives 2},
\begin{equation*}
    \partial_{\vec{x}} \,\big(\Theta^2\kappa(\Phi)\big) = \partial_{\vec{x}} 
    \det\begin{bmatrix}
    0 & \xi^{\top} \\
    \xi & \partial_{z z}^2 \Phi
    \end{bmatrix} = \Theta^2 \mathbf{S}^1, \quad \partial_{\vec{x}} \Theta = - \frac{\Theta}{\kappa(\Phi)} \Big(\det \big[\partial_{z z}^2 \Phi\big] \mathbf{T}^1 + \mathbf{S}^2 \Big)
\end{equation*}
where the $\mathbf{S}^i$ are as in Definition~\ref{cinematic curvature definition}. The product rule then yields 
\begin{equation}\label{2nd order nu}
    \partial_{\vec{x}}\, \Big(\big(\Theta^2\kappa(\Phi)\big)^{-1}\Theta^{-1}\Big) = -\Theta^{-3} \kappa(\Phi)^{-2} \big(\mathbf{S} - \det \big[\partial_{z z}^2 \Phi\big]\mathbf{T}^1 \big).
\end{equation}
Combining \eqref{xi derivatives}, \eqref{1st order nu}, \eqref{1st order nu 2} and \eqref{2nd order nu}, one deduces that
\begin{equation}\label{H2}
     \det \begin{bmatrix}
     \partial_{\xi_i \xi_j}^2 \partial_{\vec{x}}\,\varphi(\vec{x};\xi) \\
     \partial_{\xi_1} \partial_{\vec{x}}\,\varphi(\vec{x};\xi) \\
     \partial_{\xi_2} \partial_{\vec{x}}\,\varphi(\vec{x};\xi)
     \end{bmatrix}
     = (-1)^{i+j+1}\frac{\mathrm{Cin}(\Phi)(\vec{x};\nu(\vec{x};\xi))}{\Theta(\vec{x};\xi)^3 \kappa(\Phi)(\vec{x};\nu(\vec{x};\xi))^3} \xi_{\sigma_i}\xi_{\sigma_j} .
\end{equation}

The identities \eqref{H0}, \eqref{H1} and \eqref{H2} allow one to relate the conditions H0), H1) and H2) of the phase $\varphi$ to properties of the underlying defining function (and, in particular, bounds on $\| \Phi \|_{C^N}$, $\kappa(\Phi)$, $\mathrm{Proj}(\Phi)$ and $\mathrm{Cin}(\Phi)$).




\subsection{Application of local smoothing} Theorem~\ref{local smoothing theorem} can now be applied to yield Proposition~\ref{Lp proposition}. 

\begin{proof}[Proof of Proposition~\ref{Lp proposition}] The main difficulty is to prove iii). Fix $[\Phi;\fc] \in \mathfrak{C}^{k}_{\mathrm{Cin}} \cap \mathfrak{C}^k_{\mathrm{Rot}}$ and  $\delta>0$; it 
 Let $I$ denote an interval of length $|I|\sim 1$. 
The Sobolev embedding argument used to prove \eqref{sobolev lemma} yields
\begin{multline}\label{Sobolev embedding LS}
\| \sup_{t \in I} |A_j [\Phi_t; \fc_t] f|  \|_{L^6 (\R^2)}^6  \leq |I|^{-1} \| A_j[\Phi_t; \fc_t] f \|_{L^6(\R^2 \times I)}^6 \\
+ 6 \| A_j[\Phi_t; \fc_t] f \|_{L^6(\R^2 \times I)}^5 \| A_j[\Phi_t; \mathfrak{d}_t] f \|_{L^6(\R^2 \times I)},
\end{multline}
where $\mathfrak{d}_t := 2\pi i 2^j (\partial_t \Phi_t) \mathfrak{c}_t +  \partial_t \mathfrak{c}_t$.\footnote{To be more precise, one may write $A_j[\Phi_t;\mathfrak{d}_t] = A_j[\Phi_t;\mathfrak{d}_t^1] + A_j[\Phi_t;\mathfrak{d}_t^2]$ where $\mathfrak{d}_t^1 :=2\pi i 2^j (\partial_t \Phi_t) \mathfrak{c}_t $  and the average corresponding to $\mathfrak{d}_t^2$ is defined with the frequency cut-off $\theta \mapsto \theta \beta(\theta)$, rather than just $\beta$. It is remarked that this ambiguity in the definition has no bearing on the analysis.}  By the definition of the class $\mathfrak{C}^k_{\mathrm{Cin}}$,
\begin{equation}\label{good curvature}
    |\kappa(\Phi)(x,t;z)|,\:\: |\mathrm{Proj}(\Phi)(x,t;z)|,\:\: |\mathrm{Cin}(\Phi)(x,t;z)|  \gtrsim 2^{-M k}\delta 
\end{equation} 
whenever $(x,t;z) \in \mathrm{supp}\,\fc$ and $|t- x_2| \gtrsim \delta$. Decompose $\mathfrak{c} := \mathfrak{c}^{(\delta)} + \mathfrak{c}^{\dagger}$ where 
\begin{equation*}
    \mathfrak{c}^{(\delta)}(x,t;z) := \mathfrak{c}(x,t;z)\eta((t-x_2)/10\delta)
\end{equation*}
so that the estimates \eqref{good curvature} hold on the support of $\mathfrak{c}^{\dagger}$.

The piece corresponding to $\fc^{(\delta)}$ can be bounded using the theory from Sections \ref{nondegenerate-sect} and \ref{reduction to oscillatory integral estimates section}.  Indeed, let $\cG(x_1,x_2,t,z):=(x_1,x_2+t,t;z)$ and define
\begin{equation*}
  \tilde{\Phi}:= \Phi \circ \cG, \qquad \tilde{\fc}^{(\delta)}:= \fc^{(\delta)} \circ \cG.  
\end{equation*}
Note that $|x_2| \leq \delta$ in $\supp \tilde{\fc}^{(\delta)}$. Performing the above change of variables, by Fubini's theorem 
\begin{equation*}
\| A_j [\Phi_t; \fc^{(\delta)}_t] f \|_{L^6(\R^2 \times I)}^6 = \int_{-\delta}^\delta   \|A_j [(\tilde{\Phi})^{\star}_r; (\tilde{\fc}^{(\delta)})^{\star}_r] f \|_{L^6(\R^2)}^6 \, \ud r
\end{equation*}
where 
\begin{equation*}
 (\tilde{\Phi})^{\star}(u,t,r; v, \rho):=\tilde{\Phi} (u,r,t;v,\rho) \quad \textrm{and} \quad (\tilde{\fc}^{(\delta)})^{\star}(u,t,r;v,\rho):= \tilde{\fc}^{(\delta)} (u,r,t;v,\rho).    
\end{equation*}
Since $[\Phi; \fc] \in \mathfrak{C}^{k}_{\mathrm{Rot}}$, it follows that $[(\tilde{\Phi})^{\star}; (\tilde{\fc}^{(\delta)})^{\star}] \in \mathfrak{A}^{k,\ell(k)}_{\mathrm{Rot}}$. 
Combining Proposition \ref{2par-Hoerprop} with Lemma \ref{TvsA-lemma} we get an $L^2(\bbR^2)$ estimate for fixed $r$,
$$
\|A_j [(\tilde{\Phi})^{\star}_r; \tilde{\fc}_r^{(\delta)} \|_{L^2(\R^2) \to L^2(\R^2)} \lesssim (2^{-j/2} \wedge 2^{-2k/3})\| \fc^{\delta}\|_{C^N_z}.
$$
Interpolating this bound  with the $L^{\infty}$ estimate from Corollary~\ref{Lp corollary} iii) one gets
$$
\|A_j [(\tilde{\Phi})^{\star}_r; \tilde{\fc}_r^{(\delta)} \|_{L^6(\R^2) \to L^6(\R^2)} \lesssim (2^{-j/6} \wedge 2^{-2k/3}) \| \fc^{\delta}\|_{C^N_z}
$$
and therefore
\begin{equation} \label{singular part} 
\| A_j [\Phi_t; \fc_t^{(\delta)}] \|_{L^6(\R^2) \to L^6(\R^2 \times I)} \lesssim \delta^{1/6} (2^{-j/6} \wedge 2^{-2k/3})\| \fc \|_{C^N}
\lesssim \delta^{1/6} 2^{-j/6}\|\fc\|_{C^N}.
\end{equation}

On the other hand, Theorem~\ref{local smoothing theorem} can be used to show that \begin{equation}\label{non-singular part}
    \big\|A_j[\Phi_t;  \fc_t^{\dagger}]\big\|_{L^{6}(\R^2) \to L^{6}(\R^2 \times I)} \lesssim_{\varepsilon}
   \delta^{-M}2^{M k} 2^{-j(1/3 - \varepsilon)}\|\fc\|_{C^N}.
\end{equation}
Temporarily assuming \eqref{non-singular part}, by taking $\delta := 2^{-j/(24 M)}$ and $\varepsilon := 1/12$, 
we get 
\begin{equation*}
\| A_j [\Phi_t; \fc_t^\dagger]  \|_{L^6(\R^2) \to L^6(\R^2 \times I)} \lesssim 2^{Mk} 2^{-j(1/3-1/12-1/24)}\|\fc\|_{C^N}
\end{equation*}
and hence
combining this with  \eqref{singular part} 
we obtain 
\begin{align}\notag
\| A_j [\Phi_t; \fc_t]  \|_{L^6(\R^2) \to L^6(\R^2 \times I)} &\lesssim_{\varepsilon} 2^{Mk} 2^{-j/6} (2^{-j/24}+ 2^{-j/(144M)}) \| \fc \|_{C^N}\\&\lc 2^{Mk} 2^{-j/6- j\eps_0} \|\fc\|_{C^N}
\label{estimate for c LS}
\end{align}
for some  $\varepsilon_0>0$ (indeed $\varepsilon_0= (144 M)^{-1}$). This gives a favourable bound for the terms on the right-hand side of \eqref{Sobolev embedding LS} involving $\fc_t$. For the amplitude $\mathfrak{d}_t$ it suffices to note that $\| \mathfrak{d}\| \lesssim 2^j \| \mathfrak{c} \|$ and that $[\Phi; \mathfrak{d}] \in \mathfrak{C}^k_{\mathrm{Cin}} \cap \mathfrak{C}^k_{\mathrm{Rot}}$. Therefore
\begin{equation}\label{estimate for d LS}
    \| A_j [\Phi_t; \mathfrak{d}_t] \|_{L^6(\R^2) \to L^6(\R^2 \times I)} \lesssim_{\varepsilon} 2^{Mk} 2^{j(5/6 - \varepsilon_0)} \| \mathfrak{c} \|_{C^N}.
\end{equation}
Combining \eqref{estimate for c LS} and \eqref{estimate for d LS} in \eqref{Sobolev embedding LS} concludes the argument  of Proposition~\ref{Lp proposition} for $[\Phi; \mathfrak{c}] \in \mathfrak{C}^{k}_{\mathrm{Cin}} \cap \mathfrak{C}^{k}_{\mathrm{Rot}}$.

It remains to prove \eqref{non-singular part}. Let $[\varphi;\mathbbm{c}]$ be the FIO pair associated to $[\Phi;\fc^{\dagger}] \in \mathfrak{C}_{\mathrm{Cin}}^{k}$, defined as in \eqref{n Fourier phase} and \eqref{FT kernel 2}. Thus,
\begin{equation*}
A_j [\Phi_t; \fc_t^\dagger]f(x)=\mathcal{F}_{-1/2}^{2^j}[\varphi; \bc]f(\vec{x}) + \mathcal{E}_jf(\vec{x}),
\end{equation*}
where the operator $\mathcal{E}_j$ arises from the errors in \eqref{kernel localisation} and \eqref{FT kernel 2}. The smoothing term $\mathcal{E}_j$ can be easily estimated using repeated integration-by-parts and the rapid decay from \eqref{bounds error 1} and \eqref{error bounds 2}. 

Turning to the main term $\mathcal{F}^{2^j}_{-1/2}[\varphi; \bc] f$, the condition \textbf{C}$_{\delta})_k$ together with \eqref{H0}, \eqref{H1} and \eqref{H2} imply that $[\varphi;\mathbbm{c}] \in \mathbb{A}^{k} := \mathbb{A}(\delta^{-M_{\circ}}2^{M_{\circ} k})$ (in the sense of Definition~\ref{FIO class}) for some absolute constant $M_{\circ} \geq 1$. Thus, Theorem~\ref{local smoothing theorem} implies that
\begin{align*}
    \|\mathcal{F}_{\mu}^{\lambda}[\varphi; \bc]\|_{L^6(\R^2) \to L^6(\R^3)} &\lesssim_{\varepsilon} \delta^{-M} 2^{M k} \lambda^{1/6 + \mu + \varepsilon} \|\bc\|_{C^N}.
\end{align*}
The case of interest is given by $\mu = -1/2$; note that for this value the $\lambda$ exponent is $-1/3 + \varepsilon$, corresponding to the $2^j$ exponent in \eqref{non-singular part}.

For the remaining cases i), ii), iv) and v) of the proposition the argument is similar but somewhat easier. Indeed, here the condition \textbf{C})$_k$ provides favourable lower bounds for the various curvatures and this obviates the need to form any decomposition $\fa = \fa^{(\delta)} + \fa^{\dagger}$ (one may bound $
A_j[\Phi;\fa]$ directly using Theorem~\ref{local smoothing theorem}).

\end{proof}




\section{The global maximal function}\label{sec:global}

It remains to extend the bound for the local maximal function from Theorem~\ref{thm local maximal function} to the bound on the `global' maximal function from Theorem~\ref{main theorem}. This is the last step in the proof of Theorem~\ref{circular maximal theorem}.

\begin{proof}[Proof of Theorem~\ref{main theorem}] Break the operator according to the relative size of $r$ with respect to $t$, thus:
\begin{equation*}
    \sup_{t > 0}|A_t f(u, r)| =\sup_{T\in \Z} \sup_{2^T\le t< 2^{T+1}} (\sum_{m\ge 10}+\sum_{m\le -10}+\sum_{|m|< 10}) \beta^{m+T}(r) |A_t f(u, r)|.
\end{equation*}
Each of the three terms is estimated separately. Of these, the first case (corresponding to $t\ll r$) presents the most interesting features.




\subsection*{The first term: $t\ll r$}
The orthogonality relation \eqref{orthogonality relation} induces spatial orthogonality and it therefore suffices to show that
\begin{equation*}
\Big\|\sup_{T\in \Z} \sup_{2^T\le t \leq 2^{T+1}} \sum_{m\ge 10}  \beta^{m+T} \cdot |A_t f| \Big\|_{L^p(\R\times [2^W, 2^{W+1}])} \lesim \|f \chi_{\R\times [2^{W-1}, 2^{W+2}]}\|_p,
\end{equation*}
uniformly in $W\in \Z$. By the rescaling $(u,r,t;v,\rho) \mapsto (2^{2W}u, 2^W r, 2^W t; 2^{2W}v, 2^W \rho)$, the problem reduces to the case $W=0$, and therefore one needs to only show that
\begin{equation*}
 \Big\| \sup_{T\le -5} \sup_{2^T\le t \leq 2^{T+1}} \beta^0 \cdot |A_t f| \Big\|_{L^p(\R \times [1,2])} \lesssim \| f \|_{p}.
\end{equation*}

For fixed $T \leq -5$, decompose $f$ into frequency localised pieces
\begin{equation*}
f = P_{\leq -T} f + \sum_{k=1}^{\infty} P_{-T+k} f,    
\end{equation*}
where $(P_{\leq m} f) \, \,  \widehat{ \, }\, \,(\xi) := \eta^{m} (|\xi|) \widehat{f}(\xi)$ and $(P_m f) \, \, \widehat{ \, } \, \, (\xi) := \beta^{m}(|\xi|) \widehat{f}(\xi)$ for the functions $\eta^m$ and $\beta^m$ defined in \eqref{beta function}. A routine computation shows that the precomposition of the above maximal operator with $P_{\leq -T}$ is pointwise dominated by the Hardy--Littlewood maximal function. Consequently, for $p>2$ it suffices to show that
\begin{equation*}
\Big\| \sup_{T\le -5}\sup_{2^T\le t \leq 2^{T+1}} \beta^0\cdot  |A_t P_{-T+k}f|\Big\|_p \lesim 2^{-\varepsilon_p k}\|f\|_p
\end{equation*}
and Littlewood--Paley theory further reduces the problem to proving
\begin{equation}\label{fixed high frequency single scale}
\Big\|\sup_{2^T\le t \leq 2^{T+1}} \beta^0\cdot |A_t P_{-T+k}f| \Big\|_p \lesim 2^{-\varepsilon_p k}\|f\|_p,
\end{equation}
uniformly in $T \leq -5$. The rescaling $(u,r,t;v,\rho) \mapsto (2^{2T}u, 2^T r, 2^T t; 2^{2T}v, 2^T \rho)$ transforms \eqref{fixed high frequency single scale} into
\begin{equation*}
\Big\|\sup_{1 \le t \leq 2} \beta^{-T}\cdot |A_t P^T_{k}f| \Big\|_p \lesim 2^{-\varepsilon_p k}\|f\|_p,
\end{equation*}
where $P^T_k$ denotes the anisotropic frequency projection associated to the multiplier $\beta^k\big(|(2^{-T}\xi_1, \xi_2)|\big)$. 

The situation in the last display is close to the case $m=-T>0$ in the decomposition \eqref{190715e5.1}, although a direct application of Theorem~\ref{initial decomposition lemma} iii) will not give the desired decay in $j$. Instead, we decompose the operator $A$ as a sum of frequency localised operators $A_{j}$ as in \eqref{frequence localised definition} and appeal to Proposition~\ref{frequency decomposition prop} iv). First, for fixed $T \leq - 5$, write
\begin{equation*}
\beta^{-T}(r)\cdot A_t P_k^T f(u,r)=\sum_{\vec{\sigma}\in \Z^2} 2^{-T} A[\Phi_t; \fa_t^{-T, \vec{\sigma}}]P^T_kf (u,r),
\end{equation*}
where $\fa_t^{-T,\vec{\sigma}}$ is as in \eqref{m>0 def}. The relations \eqref{orthogonality relation} ensure that $|r-\rho| \lesssim 1$ and $|u-v| \lesssim 2^{-T}$, so by spatial orthogonality it suffices to prove
$$
\Big\|\sup_{1 \le t \leq 2} |A [\Phi_t; \fa_t^{-T,\vec{\sigma}} ]P^T_{k}f| \Big\|_p \lesim 2^{T} 2^{-\varepsilon_p k}\|f\|_p
$$
uniformly in $\vec{\sigma} \in \Z^2$. A further rescaling $(u,v)=(2^{-T}u, 2^{-T}v)$ transforms the above estimate into
\begin{equation}\label{desired frequency localised estimate for big r}
\big\| \sup_{1 \le t \leq 2} |A [\Phi_t^{-T}; \tilde{\fa}_t^{-T,\vec{\sigma}} ]P_{k}f| \big\|_p \lesim 2^{-\varepsilon_p k}\|f\|_p,
\end{equation}
where now $P_k$ is the usual dyadic frequency projection at scale $2^k$ and $\Phi^{-T}$ and $\tilde{\fa}^{-T,\vec{\sigma}}$ are defined as in \eqref{definition phi m}; in particular, $[\Phi^{-T}; \fa^{-T, \vec{\sigma}}] \in \mathfrak{A}^{0}_{\mathrm{Cin}} \cap \mathfrak{A}^0_{\mathrm{Rot}}$. Decompose $A[\Phi^{-T}_t; \fa_t^{-T, \vec{\sigma}}] = \sum_{j \geq 0} A_{j} [\Phi^{-T}_t; \fa_t^{-T, \vec{\sigma}}]$ as in \eqref{frequence localised definition}. Then, for fixed $k >0$, one needs to understand 
\begin{equation}\label{Littlewood-Paley terms}
  A_{j} [\Phi^{-T}_t; \fa_t^{-T, \vec{\sigma}}]P_k f (u,r)= \int_{\hat{\R}^2} \widetilde{K}^{2^{j}} (u,r,t; \xi) \beta^k(\xi) \hat{f}(\xi) \, \ud \xi  
\end{equation}
for $j \geq 0$, where $\widetilde{K}^{2^{j}}$ is as in \eqref{oscillatory kernel}. 

The main contribution arises from the terms with  $|j-k| \leq 5$. Here we   appeal to Proposition~\ref{frequency decomposition prop} iv), which yields
\begin{equation*}
 \big\| \sup_{1 \le t \leq 2} |A_{j} [\Phi_t^{-T}; \tilde{\fa}_t^{-T,\vec{\sigma}} ]P_{k}f| \big\|_p \lesim 2^{-k \varepsilon_p} \|f\|_p,   
\end{equation*}
with  some $\varepsilon_p>0$ when $p>2$.

Now consider the case $|j-k|>5$ in \eqref{Littlewood-Paley terms}. In our present rescaled situation we have
 $|\partial_{(v,\rho)} \Phi^{-T}| \sim 1$  and also favourable upper bounds for the higher $(v,\rho)$-derivatives. Hence, arguing as in \S\ref{sec:reduction to 2 Fourier variables},
using  repeated integration-by-parts, we obtain
\begin{equation*}
   |\partial_\xi^{\alpha} [e^{-2\pi i \langle (u,r), \xi \rangle} \widetilde{K}^{2^{j}} (u,r,t;\xi)] | \lesssim \min \{2^{-jN/2}, 2^{-{kN/2}}\} \, (1+|\xi|)^{-N/2} 
\end{equation*}
for all $(u,r,t) \in \supp \tilde{\fa}^{-T, \vec{\sigma}}$, $\xi \in \supp \beta^k$, $\alpha \in \N_0^2$ such that $|\alpha| \leq N$. This yields, via another integration-by-parts,
\begin{equation*}
  |A_{j} [\Phi^{-T}_t; \fa_t^{-T, \vec{\sigma}}]P_k f (u,r)| \lesssim \big(2^{-jN/2} \wedge 2^{-kN/2}\big) \int_{\R^2} \frac{f(v,\rho)}{(1+|(u,r)-(v,\rho)|)^{N/2}} \, \ud v \, \ud \rho,
\end{equation*}
which readily implies that
\begin{equation*}
   \big\| \sup_{1 \le t \leq 2} |A_{j} [\Phi_t^{-T}; \tilde{\fa}_t^{-T,\vec{\sigma}} ]P_{k}f| \big\|_p \lesim \big(2^{-jN/2} \wedge 2^{-kN/2}\big) \|f\|_p 
\end{equation*}
for $1 \leq p \leq \infty$, whenever  $|j-k| > 5$.  Combining the above observations, one obtains the desired estimate \eqref{desired frequency localised estimate for big r}. 




\subsection*{The second term: $t\gg r$.} By the triangle inequality, for all $p>2$ it suffices to show that
\begin{equation*}
\big\| \sup_{T\in \Z} \sup_{2^T\le t< 2^{T+1}} \beta^{m+T} \cdot |A_t f| \big\|_{p} \lesim 2^{\varepsilon_p m}\|f\|_p
\end{equation*}
holds uniformly in $m$ for some $\varepsilon_p>0$. The orthogonality relation \eqref{orthogonality relation} ensures that $|t-\rho| \leq r \sim 2^{m+T} \lesssim 2^T$. This induces spatial orthogonality between the $t$ and $\rho$ variables and reduces the analysis to proving
\begin{equation*}
\big\| \sup_{2^T\le t< 2^{T+1}} \beta^{m+T} \cdot |A_t f| \big\|_{p} \lesim 2^{\varepsilon_p m}\|f\|_p
\end{equation*}
uniformly in $T\in \Z$. By the rescaling $(u,r,t;v,\rho) \mapsto (2^{2T}u, 2^T r, 2^T t; 2^{2T}v, 2^T \rho)$, it suffices to consider the case $T=0$. The resulting term corresponds to
\begin{equation*}
\sum_{\vec{\sigma} \in \Z^2} 2^{m/p} A[\Phi_t; \fa_t^{m, \vec{\sigma}}]f(u,r)
\end{equation*}
in \eqref{190715e5.1}, whose $L^p$ norm is bounded by $2^{m \varepsilon_p}$  for some $\varepsilon_p>0$ if $p>2$ via Theorem~\ref{initial decomposition lemma} iv), using the orthogonality arguments in the proof of Theorem \ref{thm local maximal function}.

\subsection*{The third term: $t\sim r$.} Without loss of generality, by replacing $\beta$ with a cutoff function with slightly larger support, it suffices to bound the term corresponding to $m = 0$. Assuming $f$ is non-negative, for each fixed $T$ perform a decomposition of the operator similar to that in \eqref{190715e5.5} and \eqref{190715e5.6} by dominating
\begin{align*}
   \beta^T(r)\cdot A_t f(u,r) &\lesssim \sum_{\substack{(k,\ell) \in \Z^2 \\k \geq -4 \\ k-3\le \ell < \ell (k)}} \sum_{\vec{\sigma} \in \Z^2} 2^{k(1-1/p)+T}A[\Phi_t; (\fa_T^{k,\ell,\vec{\sigma}})_t]f \\
   & \qquad + \sum_{\substack{k\in \Z \\ k \geq -4}} \sum_{\vec{\sigma} \in \Z^2}2^{k(1-1/p)+T}A[\Phi_t; (\fc_T^{k, \ell,\vec{\sigma}})_t]f 
\end{align*}
where
\begin{align*}
    \fa_T^{k,\ell,\vec{\sigma}}(u,r,t; v,\rho) &:=\fa^{k,\ell,\vec{\sigma}}(2^{-2T}u,2^{-T}r, 2^{-T}t; 2^{-2T}v, 2^{-T}\rho), \qquad  \ell < \ell(k), \\
    \fc_T^{k,\vec{\sigma}}(u,r,t;v,\rho) &:=  \fc^{k,\vec{\sigma}}(2^{-2T}u,2^{-T}r, 2^{-T}t; 2^{-2T}v, 2^{-T}\rho).
 \end{align*}
By the triangle inequality, for all $p > 2$ it suffices to prove 
\begin{align*}
\Big\| \sup_{T\in \Z}\sup_{2^T\le t\le 2^{T+1}} \sum_{ k-3 \leq\ell < \ell(k)} \sum_{\vec{\sigma} \in \Z^2}  2^{k(1-1/p)+T} A[\Phi_t; (\fa_T^{k,\ell,\vec{\sigma}})_t]f\Big\|_p &\lesim 2^{-\varepsilon_p k}\|f\|_p, \\
\Big\| \sup_{T\in \Z}\sup_{2^T\le t\le 2^{T+1}} \sum_{\vec{\sigma} \in \Z^2} 2^{k(1-1/p)+T} A[\Phi_t; (\fc_T^{k,\ell,\vec{\sigma}})_t]f \Big\|_p &\lesim 2^{-\varepsilon_p k}\|f\|_p
\end{align*}
for some $\varepsilon_p>0$. After fixing $k$, spatial orthogonality becomes available: the variable $\rho$ is localised at $\rho \sim 2^{-k+T}$. Therefore, in order to show the above estimates, it suffices to prove
\begin{align*}
   \Big\| \sup_{2^T\le t\le 2^{T+1}} \sum_{k-3 \leq\ell < \ell(k)} \sum_{\vec{\sigma} \in \Z^2} 2^{k(1-1/p)+T} A[\Phi_t;(\fa_T^{k,\ell,\vec{\sigma}})_t]f \Big\|_p &\lesim 2^{-\varepsilon_p k}\|f\|_p, \\
   \Big\| \sup_{2^T\le t\le 2^{T+1}} \sum_{\vec{\sigma} \in \Z^2} 2^{k(1-1/p)+T} A[\Phi_t; (\fc_T^{k,\ell,\vec{\sigma}})_t]f\Big\|_p &\lesim 2^{-\varepsilon_p k}\|f\|_p,
\end{align*}
uniformly in $T$. By the rescaling $(u,r,t;v,\rho) \mapsto (2^{2T}u, 2^T r, 2^T t; 2^{2T}v, 2^T \rho)$, it suffices to only consider the case $T=0$. This follows from Theorem~\ref{initial decomposition lemma} i) and ii) using the arguments in the proof of Theorem \ref{thm local maximal function} (following the statement of Theorem~\ref{initial decomposition lemma}).
\end{proof}




\appendix

 \appendix
\section{Lemmata on integration-by-parts}
 The proofs on oscillatory integrals in \S \ref{two-parameter-sect} use a lemma which keeps track of  the terms that occur in the repeated integration-by-parts arguments. 
Assume  that $z\mapsto h(z) \in C^\infty_c$ 
(and keep track of the $C^N$-norms of $h$), and that $\nabla\Theta\neq 0$
 on
$\supp(h)$.
Define a differential operator $\cL$ by 
\[\cL h= \text{div} \Big( \frac{h\nabla\Theta}{|\nabla\Theta|^2}\Big).\]
Then, by integration by parts,
\[\int_{\R^d} e^{i\la\Theta(z) } h(z) \, \ud z=
i^N\la^{-N} \int_{\R^d} e^{i\la\Theta(z) } \cL^N\!h(z)\, \ud z
\] and thus
\begin{align}
\notag\Big| \int_{\R^d} e^{i\la\Theta(z) } h(z)  \, \ud z\Big| &\le \lambda^{-N} \int_{\R^d}|\cL^N\!h(z)|\,\ud z
\\
&\le \la^{-N}\meas(\supp\chi) \sup_{z \in \R^d}|\cL^N\!h(z)|.
\label{supestforintbyparts}
\end{align}
A careful analysis of the term 
 $\cL^N\!h$ is needed for 
various integration-by-parts arguments in this paper and elsewhere in the  literature, but   a detailed analysis is  often left to the reader. 
 For an explicit reference, a straightforward induction proof of the following lemma is contained e.g. in  the appendix of \cite{ACPS} (and probably elsewhere). 
 
We shall introduce the following notation. We  say that a term  
is of {\it type $(A,j)$} if it is of the form $h_j/|\nabla_z\Theta|^j$ where $h_j$ is a $z$-derivative of order $j$ of $h$. 
 A term  of {\it type $(B,0)$}  is equal to $1$. 
 A term is of {\it type $(B,j)$} for some $j\ge 1$  if it  is of the form $\Theta_{j+1}/|\nabla_z\Theta|^{j+1}$ where $\Theta_{j+1}$ is a $z$-derivative of order $j+1$ of $\Theta$. 
\begin{lemma}\label{intbypartslem}
Let $N=0,1,2,\dots$. Then \[\cL^N\!h = \sum_{\nu=1}^{K(N,d)} c_{N,\nu} h_{N,\nu}\] where $K(N,d)>0$,  $c_{N,\nu}$ are absolute constants independent of $h$ and $\Theta$, and each function $h_{N,\nu}$ is of the form \footnote{The product $\prod_{\ell=1}^{M_\nu}$ is interpreted to be $1$ if $M_\nu=0$, i.e. $j_{A,\nu}=N$.}
\Be P_\nu(\tfrac{\nabla \Theta}{|\nabla\Theta|}) \beta_{A,\nu} \prod_{\ell=1}^{M_\nu} \gamma_{\ell,\nu}
\label{explicit} 
\Ee such that each $P_\nu $ is a polynomial of $d$ variables (independent of $h$ and $\Theta$), $\beta_{A,\nu}$ is of type $(A,j_{A,\nu}) $ for some $j_{A,\nu}\in \{0,\dots, N\}$ and the terms $\gamma_{\ell,\nu} $ are of type $(B,\ka_{\ell,\nu} )$ for some $\kappa_{\ell,\nu} \in \{1, \dots, N\}$, so that for  $\nu=1,\dots, K(N,d)$ 
\Be\label{powerssumtoN}
j_{A,\nu}+\sum_{\ell=1}^{M_\nu} \ka_{\ell,\nu} =N.
\Ee
\end{lemma}

\noi{\bf Example.} 
In \S\ref{two-parameter-sect}  we use the Lemma \ref{intbypartslem} in the form of Corollary \ref{corollaryapp} below, choosing  \Be\label{choiceofTheta}\Theta(z)= \Psi(x;z)-\Psi(y;z),\Ee for  fixed $x=(x',x_d)$, $y=(y',y_d) \in \R^d$. 
Our differential operator $\cL=\cL_{x,y}$ depends then on $x,y$.

\begin{corollary} \label{corollaryapp}
Let   $h\in C^{N}(\bbR^d)$ be compactly supported. Let $\rho(x,y)>0$, 
and assume that for all $z$ in a neighborhood of $\supp h$
\begin{subequations}
\Be  |\nabla_z\Psi(x;z)-\nabla_z\Psi(y;z)| \gc \rho(x,y).
\label{lowerbdpartials}
\Ee Let $R(x,y)\ge 1$  and assume that for  all  $z$-derivatives up to order $|\alpha|\le N+1$,
\Be  |\partial^{\alpha}_z\Psi(x;z)-\partial_z^\alpha\Psi(y;z)| \lc_N R(x,y)\rho(x,y).
\label{upperbtpartials}
\Ee 
\end{subequations}
 Then 
\Be\notag  \Big|\int_{\R^d} e^{i\la(\Psi(x;z)-\Psi(y;z))} h(z) \, \ud z\Big| \lc_N \la^{-N}
\meas(\supp h) \max_{j=0,\dots, N}\frac {  \|h\|_{C^j}R(x,y)^{N-j}}{\rho(x,y)^N}.
\Ee
\end{corollary}
\begin{proof} We use \eqref{supestforintbyparts} and the assertion follows from 
\Be  |\cL_{x,y}^Nh(z)|\lc_N \max_{j=0,\dots, N}\frac{\|h\|_{C^j}R(x,y)^{N-j}}{\rho(x,y)^N}.
\label{LNbd1}\Ee
To see this use Lemma \ref{intbypartslem} with the choice \eqref{choiceofTheta}. Observe that  by \eqref{lowerbdpartials} 
an  expression of type $(A,j)$ is bounded by a constant times $ \|h\|_{C^j}(\rho(x,y))^{-j}$. 
By \eqref{lowerbdpartials} and \eqref{upperbtpartials} an  expression of type
$(B,\kappa)$ is  bounded by a constant times
 $ R(x,y) (\rho(x,y))^{-\kappa}$.  We use \eqref{powerssumtoN} to see that  the  expression corresponding to  \eqref{explicit} is 
 bounded by \[C_N \frac{\|h\|_{C^{j_{A,\nu}}} R(x,y)^{M_\nu}}{
 (\rho(x,y))^{j_{A,\nu}+\sum_{\ell=1}^{M_\nu} \kappa_{\ell,\nu}} }
 \lc_N \frac{\|h\|_{C^{j_{A,\nu}}} R(x,y)^{N-j_{A,\nu}} }{\rho(x,y)^{N}}\] 
 and hence  we get \eqref{LNbd1}.
 \end{proof}

\subsubsection*{Applications of Corollary  \ref{corollaryapp}} Here $0<\delta_\circ\le 1$ and $m>0$. 

\begin{itemize}
    \item In the proof of Proposition \ref{2par-Hoerprop}, Corollary \ref{corollaryapp} is applied with
the choice of 
$\rho(x,y):=|x'-y'|+\delta_\circ|x_d-y_d|$, $R(x,y)\lc 1$ and the $C^N$ norm of the amplitude is $O(1)$.
\item In the proof of Lemma \ref{off-diagonal lemma} 
Corollary \ref{corollaryapp} is applied with
$\rho(x,y):=|x'-y'|$ and $R(x,y)\lc 1$, and the $C^N$ norm of the amplitude is $O(2^{mN})$.
\item In the proof of Lemma \ref{diagonal lemma-M} the $d-1$-dimensional version of 
Corollary \ref{corollaryapp} is applied with
$\rho(x',y'):=|x'-y'|$ and $R(x',y')\lc 1$, and the $C^N$ norm of the amplitude is $O(2^{mN})$.
\item In the proof of Lemma \ref{diagonal lemma-m}
Corollary \ref{corollaryapp} is applied with the choices of
$\rho(x,y):=|x'-X_\nu(x_d,y;z_\nu)|+\delta_\circ 2^{-m}|x_d-y_d|$,  and $R(x,y)\lc 2^m$,  and the $C^N$ norm of the amplitude is $O(2^{mN})$.
\end{itemize}

\section{Computations related to the defining function}

\subsection{Derivative dictionary}\label{derivative appendix} For reference, here some derivatives are computed for the specific defining function $\Phi_t$ in \eqref{defining function}. Recall,
\begin{equation*}
  \Phi(u, r, t;v, \rho) := (u-v)^2 - \Big(\frac{b}{2}\Big)^2\big( 4r^2\rho^2 - (r^2 + \rho^2 - t^2)^2 \big)
\end{equation*}
so that the first order derivatives are
\begin{equation*}
     \partial_u \Phi_t = 2(u-v) , \qquad \partial_r \Phi_t = -b^2 r (t^2-r^2+\rho^2) 
\end{equation*}
and
\begin{equation*}
   \partial_v \Phi_t = -2(u-v), \qquad \partial_\rho \Phi_t = -b^2 \rho (t^2+r^2-\rho^2) 
\end{equation*}
together with the time derivative
\begin{equation*}
  \partial_t \Phi_t = b^2t(t^2 - r^2 - \rho^2).  
\end{equation*}
Of course $\partial^2_{ru} \Phi_t = \partial_{\rho u}^2 \Phi_t = \partial_{r v}^2 \Phi_t = \partial_{\rho v}^2 \Phi_t =0$ whilst the non-vanishing second order derivatives are
\begin{gather*}
 \partial^2_{uu} \Phi_t = \partial^2_{vv} \Phi_t = 2, \qquad \partial^2_{uv} \Phi_t = -2, \\  
\partial^2_{r r} \Phi_t = -b^2(t^2-3r^2+\rho^2), \qquad \partial^2_{r \rho} \Phi_t = -2b^2r\rho, \qquad \partial^2_{\rho \rho} \Phi_t = -b^2(t^2+r^2-3\rho^2)
\end{gather*}
and the time derivatives
\begin{equation*}
  \partial_{t r}^2 \Phi = -2 b^2 t r \qquad \textrm{and}  \qquad \partial^2_{t \rho}\Phi = - 2 b^2 t \rho.  
\end{equation*}
Finally, the third order derivative relevant to the argument are 
\begin{equation*}
   \partial_{\rho rr}^3 \Phi_t = -2 b^2 \rho \qquad \text{and} \qquad \partial_{\rho \rho r}^3=-2b^2r. 
\end{equation*}
With these formul\ae\ in hand, it is a simple computation to obtain the expressions \eqref{fixed t rotational curvature} and \eqref{fixed r rotational curvature} for the rotational curvature,
   \begin{align*}
   \mathrm{Rot}(\Phi_t) (u,r;v,\rho) &= 4b^4rt^2\rho (t^2 - r^2 - \rho^2), \\
   \mathrm{Rot}(\Phi_r^{\star }) (u,t;v,\rho) &= 4b^4 r^2t\rho(r^2 - t^2 - \rho^2),
\end{align*}
as well as the key identity \eqref{key identity},
\begin{equation*}
    \mathrm{Rot}(\Phi_t)(u,r;v,\rho) = 4b^2 r t \rho (\partial_t \Phi_t)(u,r;v,\rho),
\end{equation*}
and expressions \eqref{projection formula} and \eqref{cinematic formula} related to the cinematic curvature
\begin{align*}
\mathrm{Proj}(\Phi)(u,r,t;v,\rho) &= -8b^4rt\rho(r^2 - t^2), \\
 \mathrm{Cin}(\Phi)(u,r,t;v,\rho) &=  64b^8r^3t^3\rho^3(r^2 - t^2)
 \end{align*}
for $(v,\rho) \in \Sigma_{u,r,t}$.

\subsection{Rescaling}\label{rescaling appendix} It is useful to note how the expressions in the previous subsection behave under rescaling. Given $k, \tau \in \Z$ and $\varepsilon, \delta \in \Z^2$, let $\Phi^{k, \varepsilon, \tau, \delta}:= 2^k\, \Phi \circ D^{\varepsilon, \tau, \delta}$ where
\begin{equation*}
    D^{\varepsilon, \tau, \delta}(u,r,t;v,\rho):= (2^{\varepsilon_1} u, 2^{\varepsilon_2}r, 2^{\tau}t; 2^{\delta_1}v, 2^{\delta_2}\rho).
\end{equation*}
Then
\begin{equation*}
    \partial^{\alpha}_x \partial^\beta_z \partial^\gamma_t \Phi^{k,\varepsilon, \tau, \delta} (x,t;z)= 2^{k} 2^{\varepsilon \cdot \alpha} 2^{\delta \cdot \beta} 2^{\tau \gamma} (\partial^{\alpha}_x \partial^\beta_z \partial^\gamma_t \Phi ) \circ D^{\varepsilon, \tau, \delta}(x,t;z)
\end{equation*}
for all $\alpha, \beta \in \N_0^2$, $\gamma \in \N_0$. In particular,
\begin{equation*}
    \mathrm{Rot} (\Phi^{k, \varepsilon, \tau, \delta}_t) (x;z) = 2^{3k} 2^{|\varepsilon| + |\delta|} \mathrm{Rot} ( \Phi_{2^\tau t}) \circ D^{\varepsilon, \delta}(x;z) 
\end{equation*}
where $D^{\varepsilon, \delta}(x;z):= (2^\varepsilon x; 2^\delta z)$, and the rescaled key identity becomes
$$
\mathrm{Rot} (\Phi^{k, \varepsilon, \tau, \delta}_t) (x;z) = 4 b^2 r \rho t 2^{\varepsilon_2 + \delta_2} 2^{2k} 2^{|\varepsilon| + |\delta|} \partial_t \Phi^{k,\varepsilon, \tau, \delta} (x,t;z).
$$
Furthermore,
\begin{align*}
  \kappa (\Phi^{k, \varepsilon, \tau, \delta})(\vec{x};z) &= 2^{3k} 2^{2|\delta|} \kappa (\Phi) \circ D^{\varepsilon, \tau, \delta} (\vec{x}; z),  \\
  \mathrm{Proj}(\Phi^{k, \varepsilon, \tau, \delta})(\vec{x}; z) &= 2^{3k} 2^{|\varepsilon|+ \tau + |\delta| } \mathrm{Proj} (\Phi) \circ D^{\varepsilon, \tau, \delta} (\vec{x};z), \\
  \mathrm{Cin}(\Phi^{k, \varepsilon, \tau, \delta})(\vec{x}; z) &= 2^{6k} 2^{ |\varepsilon| + \tau + 3 |\delta| } \mathrm{Cin}(\Phi)\circ D^{\varepsilon, \tau, \delta}(\vec{x};z).
\end{align*}





\end{document}